\documentclass[reqno]{amsart}

\allowdisplaybreaks[3]
\usepackage{graphicx}
\usepackage{amsfonts}
\usepackage{mathrsfs}
\usepackage{fancyhdr}
\usepackage{amsthm}
\usepackage{empheq}
\usepackage{cases}
\usepackage{amsmath,amssymb,bm}
\usepackage{enumitem}
\usepackage{xcolor}
\usepackage{upgreek}
\usepackage{booktabs}
\usepackage{appendix}

\providecolor{added}{rgb}{0,0,1}
\providecolor{deleted}{rgb}{1,0,0}
\usepackage[
bookmarks=true,  %generates bookmarks for all entries in the "Table of Contents"
bookmarksnumbered=true, %includes chapter-/header-/section-/subsection-/... -
colorlinks=true, pdfstartview=FitV,
linkcolor=red,
citecolor=blue,
urlcolor=blue]{hyperref}

\def\cC{\mathcal{C}}

\def\cD{\mathcal{D}}
\def\fD{\mathfrak{D}}

\def\EE{\mathbb{E}}
\def\PP{\mathbb{P}}
\def\cE{\mathcal{E}}

\def\cF{\mathcal{F}}
\def\cH{\mathcal{H}}
\def\cI{\mathcal{I}}

\def\cJ{\mathcal{J}}
\def\fJ{\mathfrak{J}}

\def\cK{\mathcal{K}}

\def\cN{\mathcal{N}}

\def\RR{\mathbb{R}}
\def\cS{\mathcal{S}}

\def\cS{\mathcal{S}}

\def\fT{\mathfrak{T}}

\def\cZ{\mathcal{Z}}
\def\Om{\Omega}

\newcommand\bP{\mathbf{P}}
\newcommand\HH{\mathfrak H}

\def\al{{\alpha}}
\def\ls{{\lesssim}}

\def\es{{\simeq}}

\newcommand{\ep}{\varepsilon}

\newcommand{\si}{\sigma}

\newcommand{\sgn}{\text{sgn}}
\newcommand{\1}{{\bf 1}}

\newcommand{\blc}{\big(}
\newcommand{\brc}{\big)}
\newcommand{\Blc}{\Big(}
\newcommand{\Brc}{\Big)}
\newcommand{\blk}{\big[}
\newcommand{\brk}{\big]}
\newcommand{\Blk}{\Big[}
\newcommand{\Brk}{\Big]}
\newcommand{\lc}{\left(}
\newcommand{\rc}{\right)}
\newcommand{\lk}{\left[}
\newcommand{\rk}{\right]}
\newcommand{\lt}{\left }
\newcommand{\rt}{\right}

\newtheorem{thm}{Theorem}[section]
\newtheorem{lem}[thm]{Lemma}
\newtheorem{prop}[thm]{Proposition}
\newtheorem{rmk}[thm]{Remark}

\newtheorem{defn}[thm]{Definition}
\theoremstyle{defn}

\theoremstyle{remark}
\numberwithin{equation}{section}

\begin{document}

\title{Nonlinear  stochastic wave Equation
driven by rough noise}

\author{Shuhui Liu}
\address{School of Mathematics, Shandong University, Jinan, Shandong 250100, China}
\email{shuhuiliusdu@gmail.com}
\thanks{SL was supported by the China Scholarship Council.}

\author{Yaozhong Hu}
\address{Department of Mathematical and Statistical Sciences, University of Alberta, Edmonton, AB T6G 2G1, Canada}
\email{yaozhong@ualberta.ca}
\thanks{YH was supported by the NSERC discovery fund and a startup fund of University of Alberta.}

\author{Xiong Wang*}
\address{Department of Mathematical and Statistical Sciences, University of Alberta, Edmonton, AB T6G 2G1, Canada}
\email{xiongwang@ualberta.ca}
\thanks{*Corresponding author: xiongwang@ualberta.ca}

%    General info
%\subjclass[2000]{Primary 54C40, 14E20; Secondary 46E25, 20C20}

\date{}


\begin{abstract}
In this paper, we obtain  the existence and uniqueness of the strong solution to one spatial dimension  stochastic wave equation
$\frac{\partial^2 u(t,x)}{\partial t^2}=\frac{\partial^2 u(t,x)}{\partial x^2}+\sigma(t,x,u(t,x))\dot{W}(t,x)$ assuming $\sigma(t,x,0)=0$, where $\dot W$ is a mean zero Gaussian noise which is white in time and fractional in space with Hurst parameter $H\in(1/4, 1/2)$. 
%The condition $H>1/4$ is proved to be necessary for a strong  solution to exist for the hyperbolic Anderson model.
\end{abstract}

\maketitle
{\hypersetup{hidelinks}
\tableofcontents}
\section{Introduction}
In this paper, we consider the following one (spatial) dimensional stochastic nonlinear wave  equation (SWE for short) driven by  rough spatial  Gaussian noise which is white in time and fractional in space:
\begin{equation}\label{eq.SWE}
\begin{cases}
  \frac{\partial^2 u(t,x)}{\partial t^2}=\frac{\partial^2 u(t,x)}{\partial x^2}+\sigma(t,x,u(t,x))\dot{W}(t,x),\quad   t\in[0,T],\quad
   x\in\RR\,, \\
  u(0,x)=u_0(x)\,,\quad\frac{\partial}{\partial t}u(0,x)=v_0(x)\,.
\end{cases}
\end{equation}
Here    $W(t,x)$ is a centered Gaussian process with covariance given by
\begin{equation}\label{CovW}
  \EE[W(t,x)W(s,y)]=\frac 12 (s\wedge t) \blc |x|^{2H}+|y|^{2H}-|x-y|^{2H} \brc
\end{equation}
  and   $\dot{W}(t,x)=\frac{\partial^2  }{\partial t\partial x}W(t,x)$.
%   We shall study  the existence and uniqueness of the random field (mild) solution.
   The main feature of this work is our assumption that the Hurst parameter   $H\in (\frac 14, \frac 12)$.  Namely, the noise is   rough  and fractional  in space variable  and white in time variable.
When    the noise is general Gaussian which is  white in time and satisfies  the so-called Dalang's condition,
%(which implies $H>1/2$ when restricted to the fractional noise),
 there are some results about the  well-posedness of the equation and   the properties of the solutions (e.g. \cite{Dalang2009,  sanz2009, HHN2014}).
%  for general Gaussian noises
%which is white in time.
If we apply Dalang's condition to  fractional Gaussian noise, then we need to assume the spatial Hurst  parameter $H\ge 1/2$.      When
$H< 1/2$, namely, when  the noise is rough in space (in this case the spatial dimension must be one dimensional), there are very limited results. The only result we know, to the best of our knowledge, is  the work     \cite{BJQ2015}, where   the noise coefficient   $\sigma(t,x,u)=au+b$ is affine.   There has been no work to tackle  the case when $\si(t,x,u)$ is nonlinear  (or not affine)  function of $u$.   On the other hand,
%there are recently two  works for the  stochastic heat equation (SHE)   for this spatial rough noise. W
when $\frac{\partial ^2}{\partial  t^2}$ on the left hand of \eqref{eq.SWE} is replaced by $\frac{\partial }{\partial t }$, this is, in the case nonlinear stochastic heat equations (SHE for short) driven by   spatial rough noise,  the authors   of  \cite{HHLNT2017} studied   the equation in  the case $\sigma(t,x, 0)=0$. They prove the strong existence and uniqueness of solution.   This condition
 $\sigma(t,x,0)=0$  is  removed     in \cite{HW2019},
 where the authors obtained the   existence of  weak
 solution.

Our objective in this paper is to obtain  the strong existence and uniqueness of the SWE \eqref{eq.SWE} while we still assume     $\sigma(t,x,0)=0$.  The  reason  we extend the work of \cite{HHLNT2017}     under this condition is that  one can obtain the existence and uniqueness of strong solution (or mild solution) in a   solution space which is much simpler to deal with.   It seems too much involved to remove the restriction $\sigma(t,x,0)=0$
since in this case we believe that we need to introduce a weighted space for the solution and to study the interaction between the wave Green's kernel and the weight.

Even in  the case  $\sigma(t,x,0)=0$ there are   mainly two  difficulties  to study \eqref{eq.SWE} or its SHE analogue. The first one is that one cannot  bound the $L_p$ norm of $\int_0^t \int_{\RR}  h_t(s,y) W(ds,dy)$ by the $L_p$  norm of $h_t(s,y)$, instead,  one has to use the $L_p$ norm of $h_t(s,y)$ itself plus the $L_p$ norm of  its fractional derivative,  where
$h_t(s,y)=G_{t-s}(x,y)  \sigma(s,y, u(s,y))$ and  $G_ t (x,y)$ is the heat or wave kernel.
This makes thing very much sophisticated.
In particular, as indicated in
\cite{HHLNT2017, HW2019},   due to the existence of our rough noise $\dot{W}$ we need to bound $|\sigma(u_1)-\sigma(u_2)-\sigma(v_1)+\sigma(v_2)|$ by  a multiple of
 $ |u_1 - u_2 - v_1+v_2 |$ (which is possible only in the   affine case).   To get around this difficulty
the authors in \cite{HHLNT2017, HW2019} use a priori  bound of $L_p\times L_\infty$ norm
$\EE  \sup_{0\le t\le T}
|u(t,x)|_{L_p(\RR)}^p$ and the similar norm of the fractional derivative of $u(t,x)$ for the solution
$u(t,x)$.  This immediately poses a new challenge  which is our second difficulty since $\int_0^t \int_{\RR}  h_t(s,y) W(ds,dy)$ is not a   martingale in $t$
(nor it is a semimartingale), it is hard to
bound the $L_p$ norm of  $\sup_{0\le t\le T}\int_0^t \int_{\RR}  h_t(s,y) W(ds,dy)$ since we can no longer use the powerful Burkholder-Davis-Gundy inequality.  In the case of SHE, this is overcame by a clever exploitation of the semigroup property of the heat kernel.  This idea is not reproducible   in SWE simply because    the wave kernel $G_t(x,y)$ does not have the semigroup property, unfortunately!
To surmount this barrier we shall decompose the simple wave kernel $G_t(x-y)$ to four complicated parts so that we can bound the $L_p$ norm of
$\sup_{0\le t\le T}\int_0^t \int_{\RR}  h_t(s,y) W(ds,dy)$  by the $L_p$ norm of $h_t(s,y)$ itself plus the $L_p$ norm of  its fractional derivative.
Of course, one also needs to bound
  $L_p$ norm of the
$\sup_{0\le t\le T}$ norm  of the fractional derivative of $\int_0^t \int_{\RR}  h_t(s,y) W(ds,dy)$.
This will be the main effort of this work. After achieving this estimation, the proof of the existence and uniqueness of the mild solution is routine.

In the study of fractional noise, the number $1/4$ seems to be a magic number. It appears in a number of occurrences. Here we are interested in the problem if $H>1/4$ is necessary for  \eqref{eq.SWE} to have a classical ($L_2$) solution. We shall provide an  affirmative answer. To this end we consider the hyperbolic Anderson model, namely,
  $\sigma(t,x,u )=u $.  In this case  the equation  \eqref{eq.SWE} becomes
\begin{equation}\label{eq.HAE}
\begin{cases}
  \frac{\partial^2 v(t,x)}{\partial t^2}=\frac{\partial^2 v(t,x)}{\partial x^2}+v(t,x)\dot{W}(t,x),\quad   t\in[0,T],\quad
   x\in\RR\,, \\
  v(0,x)=u_0(x)\,,\quad\frac{\partial}{\partial t}v(0,x)=v_0(x)\,.
\end{cases}
\end{equation}
Under some conditions on the initial data,
we shall prove that $v(t,x)$ is square integrable
only if $H>1/4$.  After the completion of this work, we discover that the necessity of $H>1/4$ is implied in  \cite[Proposition 3.4]{SSX2019}.   To make the paper more comprehensive, we keep our alternative proof of the necessity of $H>1/4$. Our method may be useful to study the properties of   \eqref{eq.SWE} with additive noise ($\sigma\equiv 1$). Let us also mention  a recent work \cite{CH2021}  that for the parabolic Anderson model when the dimension $d=1$ and when the noise is white in time and fractional in space with Hurst parameter $H$, then $H>1/4$ is also the necessary and sufficient condition for the solution to be  square integrable.

Here is the organization of this paper.  In Section
\ref{s.2} we   briefly recall  some necessary
concept about stochastic integral and wave kernel and
so on to fix the notations  used in the paper and
we also state our main results obtained in this work.
Sections \ref{s.3}  and \ref{s.4} are the core of the paper.  In Section \ref{s.3}  we decompose the wave kernel into four parts and then we use this decomposition to obtain  the necessary   bound   of the stochastic integral  (stochastic convolution with the  wave kernel).  There are a lot of computations to obtain the bound for the stochastic convolution. We postpone some of these computations in the Appendix \ref{s.6} and \ref{Lemma for 3.3}.
Section \ref{s.4} obtains the existence and uniqueness  of the strong solution. Some of the computations are moved to Appendix \ref{Appen.C} for  the fluency of the proof. Section \ref{s.5} is about the necessity of $H>1/4$ for strong solution to exist.

 Throughout  the paper,   $A \lesssim B$ (and $A \gtrsim B$) means that there are  universal constants $C_1, C_2\in (0, \infty)$  such that $A\le C_1B$ (and $A\ge C_2B$).
% Plus,  $A\simeq B$ means both $A \lesssim B$ and $A \gtrsim B$ hold.
We also denote throughout the paper
\begin{align}
	\Delta_\tau f(t,x):=&\ f(t+\tau ,x)-f(t,x)\,,\label{time_dif} \\
	\fD_h f(t,x):=&\ f(t,x+h)-f(t,x)\,,\label{space_dif}
\end{align}
and
\begin{align}
	\Box_{h,l} f(t,x):=&\ \fD_l \fD_h f(t,x)=\fD_h f(t,x+l)- \fD_h f(t,x)\nonumber  \\
	=&\ [f(t,x+h+l)-f(t,x+l)]-[f(t,x+h)-f(t,x)]\,.\label{box_diff}
\end{align}

\section{Preliminaries and Main results}\label{s.2}
Let $(\Om,\cF, \PP)$ be a complete probability space and let $W=(W(t,x), t\ge 0\,, x\in \RR)$ be a mean zero Gaussian random field  whose covariance is given by \eqref{CovW}. For any $t\ge 0$,  $\cF_t=\si(W(s,x)\,,  s\in [0,  t], x\in \RR)$ be the $\si$-algebra generated by the Gaussian field $W$. We recall briefly some notations and  facts in  \cite{HHLNT2017} and refer to
that reference for more details.

Denote $\cS$ the set of smooth functions on $\RR_+\times \RR$ with compact support. For any $f,g\in \cS$, define
\begin{equation}
\langle f, g\rangle_\HH=c_{H}^2 \int_{\RR_{+}\times \RR^2} [f(t,x+y)-f(t,x)][g(t,x+y)-g(t,x)] |y|^{2H-2} dxdydt\,, \label{eq.def_H}
\end{equation}
where
\[
 c_{H}^2  = H(\frac12 -H) \lk\Gamma\Blc H+\frac 12\Brc\rk^{-2}\left(\int_{0}^{\infty} \Blk(1+t)^{H-\frac 12}-t^{H-\frac 12}\Brk^2 dt+\frac{1}{2H}\right) \,.
\]
Let $\HH$ be the Hilbert space obtained by completing
$\cS$ with respect to the scalar product
$\langle \cdot, \cdot \rangle_\HH$.
Let us start with the stochastic integration of elementary process with respect to $W$, and then extend it to general process.
\begin{defn}\label{Ele_def}A random field $f=(f(t,x),(t,x)\in \RR_+\times \RR)$ is called adapted to the filtration $\cF_t$ if $f(t,x)\in \cF_t$ for all $(t,x)\in \RR_+\times \RR$.
 An elementary process $g$ is $\cF_t$-adapted random field     of  the following form:
\[
g(t,x)=\sum_{i=1}^n\sum_{j=1}^mX_{i,j}\1_{(a_i,b_i]}(t)\1_{(c_j,d_j]}(x)\,,
\]
where $n$ and $m$ are  positive integers, $0\le a_1< b_1 < \cdots<a_n< b_n<+\infty$, $c_j<d_j$ and $X_{i,j}$ are $\mathcal{F}_{a_i}$-measurable random variables for $i=1,\cdots,n, j=1, \cdots, m$. The  stochastic integral of such an elementary process $g$ with respect to $W$ is defined as
\begin{align}\label{elem_def}
\int_{\mathbb{R}_+\times \mathbb{R}}&g(t,x)W(dt,dx)=\sum_{i=1}^n\sum_{j=1}^mX_{i,j}W\lt(\1_{(a_i,b_i]}\otimes\1_{(c_j,d_j]}\rt)\nonumber\\
&=\sum_{i=1}^n\sum_{j=1}^mX_{i,j}\lt[W(b_i,d_j)-W(a_i,d_j)-W(b_i,c_j)+W(a_i,c_j) \rt].
\end{align}
\end{defn}
%It is easy to verify that
%\[
%\EE \left(\left[\int_{\mathbb{R}_+\times \mathbb{R}} g(t,x)W(dt,dx)\right]^2\right)=\EE \|g\|_\HH^2\,.
%\]
%% Define the Riemann-Liouville fractional integral of order $\beta$ of a function $\phi$ by
%%$$I_{-}^{\beta}\phi(t,x)=\frac{1}{\Gamma(\beta)}\int_x^{\infty}\phi(t,y)(y-x)^{\beta-1}dy.$$
%Set
%\begin{equation*}
%\HH:=\lt\{\phi:\mathbb{R}_+\times\mathbb{R};\exists \psi\in L^2(\mathbb{R}_+\times\mathbb{R})\  s.t.\ \psi(t,x)=I_{-}^{\frac12-H}\phi(t,x)  \rt\}.
%\end{equation*}
In fact, we have the following proposition (e.g. \cite{HHLNT2017}).
% shows that the stochastic integral concerning $W$ in Definition \ref{Ele_def} can be extended to general process.
\begin{prop}\label{Int_Gen}
Let $\Lambda_{H}$ be the space of adapted random field  $g$ defined on $\mathbb{R}_+\times\mathbb{R}$ such that $g\in\HH$ a.s.  and $\mathbb{E}[\|g\|_{\HH}^2]<\infty$.   Then we have the following statements.
\begin{enumerate}
\item The space of elementary process defined in Definition \ref{Ele_def} is dense in $\Lambda_{H}$;
\item For $g\in\Lambda_{H}$, the stochastic integral $\int_{\mathbb{R}_+\times\mathbb{R}}g(t,x)W(dt,dx)$ is defined as the $L^2(\Omega)$-limit of stochastic integrals of elementary processes approximating $g(t,x)$ in $\Lambda_{H}$, and  for this stochastic integral we have the following isometry equality
    $$\mathbb{E}\lt[\lt(\int_{\mathbb{R}_+\times\mathbb{R}}g(t,x)W(dt,dx)\rt)^2 \rt]=\mathbb{E}[\|g\|_{\HH}^2].$$
\end{enumerate}
\end{prop}

Now we introduce some    norms and spaces  used in this paper. Let $(B,\| \cdot \|_B)$ be a Banach space  with the norm $\| \cdot \|_B$.  Let   $\beta\in(0,1)$ be a fixed number.
 For  any  function $f:\RR\rightarrow B$  denote
 \begin{equation}\label{NBNorm}
 \begin{split}
 	\cN_{\beta}^{B}f(x):=&\lt(\int_{\RR}\|\fD_h f(x)\|_B^2|h|^{-1-2\beta}dh\rt)^{\frac 12}\,,
% 	  \\
% 	=&\lt(\int_{\RR}\|f(x+h)-f(x)\|_B^2|h|^{-1-2\beta}dh\rt)^{\frac 12}\,,
 \end{split}
 \end{equation}
 if the above quantity is finite, where we recall $\fD_h f(x)=f(x+h)-f(x)$.
 When $B=\RR$, we abbreviate the notation $\cN_{\beta}^{\RR}f$  as  $\cN_{\beta}f$. With this notation, the norm of the homogeneous Sobolev space $\dot{H}^{\beta}$ can be given by using $\cN_\beta f$:    $\|f\|_{\dot \cH_\beta}  = \|\cN_{\beta}f\|_{L^2(\RR)}$.  As in \cite{HHLNT2017} throughout  this paper  we are particularly interested in the case $B=L^p(\Omega)$, and in this case we  denote $ \cN_{\beta}^{B}$ by $\cN_{\beta,p}$:
 \begin{equation}\label{NpNorm}
   \cN_{\beta,p}f(x):=\lt(\int_{\RR}\|\fD_h f(x)\|^2_{L^p(\Omega)} |h|^{-1-2\beta}dh\rt)^{\frac 12}.
 \end{equation}
We shall set $\beta=\frac12-H$. The following    Burkholder-Davis-Gundy  inequality is well-known (see  e.g.  \cite{HHLNT2017,HW2019}).
\begin{prop}\label{Prop.HBDG}
   Let $W$ be the Gaussian noise defined by the covariance \eqref{CovW}, and let    $f\in\Lambda_H$  be a predictable random field. Then for any $p\geq2$ we have
   \begin{equation}\label{e.bdg}
     \begin{split}
        \bigg\|\sup_{0\le r\le t} \int_{0}^{r}\int_{\RR}& f(s,y) W(ds,dy)\bigg\|_{L^p(\Omega)} \\
     \leq & C_{ H}\sqrt{ p}\lt(\int_{0}^{t}\int_{\RR}\lk\cN_{\frac 12-H,p}f(s,y)\rk^2dyds\rt)^{\frac 12} \\
     =& C_{ H}\sqrt{ p}\lt(\int_{0}^{t}\int_{\RR}\int_{\RR}\|\fD_h f(s,y)\|^2_{L^p(\Omega)} |h|^{2H-2}dh dyds\rt)^{\frac 12}\,,
     \end{split}
\end{equation}
where $C_{H}$ is a constant depending only on $H$, $\cN_{\frac 12-H,p}f(s,y)$ denotes the application of $\cN_{\frac 12-H,p} $  to the space variable $y$ and $\fD_h$ is defined by  \eqref{space_dif}.
\end{prop}
%
%Let us denote the weight function
%%\begin{equation}\label{Weight.Lamd}
%%  \lambda(x):=c_{\lambda}(1+|x|)^{-\lambda}\simeq c_{\lambda}(1+|x|^{\lambda})^{-1}\,,
%%\end{equation}
%where $c_{\lambda}$ is a constant so that $\int_{\RR}\lambda(x)dx=1$.

We introduce the solution space $\cZ^p(T)$. It consists of all continuous functions $f$ from $[0,T]\times\RR$ to $L^p(\Omega)$ such the following norm  is  finite:
\begin{align}\label{def.ZNorm}
  \|f\|_{\cZ^p(T)}=\,&\|f\|_{\cZ^p_{1}(T)} +\|f\|_{\cZ^p_{2}(T)}\\
    :=\,&\sup_{t\in[0,T]} \big\| f(t,\cdot)\big\|_{L^p(\Omega\times \RR)} +\sup_{t\in[0,T]}  \cN^{*}_{\frac 12-H,p}f(t) \,,\nonumber
\end{align}
where $\big\| f(t,\cdot)\big\|_{L^p(\Omega\times \RR)}=\lc\int_{\RR}\EE[|f(t,x)|^p]dx \rc^{1/p}$ and
\begin{equation*}
	\cN^{*}_{\frac 12-H,p}f(t):=\lc\int_{\RR} \big\| \fD_h f(t,\cdot)\big\|^2_{L^p(\Omega\times \RR)} |h|^{2H-2}dh \rc^{\frac 12}\,.
\end{equation*}
%And also we define$\footnote{We may delete this norm.}$
%\begin{align}\label{def.XNorm}
%  \|f\|_{\fX_{\lambda}^p(T)}=\,&\|f\|_{\fX^p_{\lambda,1}(T)} +\|f\|_{\fX^p_{\lambda,2}(T)}\\
%    :=\,&\sup_{t\in[0,T],x\in\RR} \lambda^{\frac 1p}(x)\|f(t,x)\|_{L^p(\Omega)} \nonumber\\
%    &\qquad+\sup_{t\in[0,T],x\in\RR}\lambda^{\frac 1p}(x) \cN_{\frac 12-H,p}f(t,x) \,.\nonumber
%\end{align}
It is proved that  $\cZ^p(T)$ is a Banach space
(e.g. \cite[Section 4.1]{HHLNT2017}).

After defining the stochastic integral, let us return to the stochastic wave equation.  Since we are working in dimension $d=1$,   the Green's function associated with
\eqref{eq.SWE} is
 \begin{equation}\label{WaveKerG}
  G_t(x)=\frac 12 \1_{\{|x|<t\}}\,,\quad  t\in \RR_+\,, x\in \RR  \,.
 \end{equation}
Notice that $G_t(x)$ does not satisfy semigroup property.

Now we give    the  definitions of strong and weak solutions
to \eqref{eq.SWE}.

\begin{defn}\label{Def.Sol}
 Let $\left\{u(t,x)\,, t\ge 0\,, x\in \RR\right\}$ be a real-valued adapted  random field  such that for all fixed $t\in[0,T]$ and $x\in \RR$, \  the random field  \[
 \{G_{t-s}(x-y)\sigma(u(s,y))\1_{[0,t]}(s)\,, (s,y)\in   \RR_+\times \RR\}
 \]
  is integrable with respect to $W$(namely it is in $\Lambda_H$).
\begin{enumerate}
\item[(i)]\     We say that $u(t,x)$ is a  \emph{strong  (mild, random field)   solution}  to \eqref{eq.SWE} if for all $t\in[0,T]$ and $x\in\RR$ we have almost surely
   \begin{align}
     u(t,x)&=\frac{\partial}{\partial t} G_t\ast u_0(x)+G_t\ast v_0(x)+G_t\circledast \sigma(\cdot,\cdot,u)(x) \nonumber\\
     &=I_0(t,x) +\int_{0}^{t}\int_{\RR}G_{t-s}(x-y)\sigma(s,y,u(s,y))W(ds,dy)\,,\label{eq.MildSol}
    \end{align}
 where
 \begin{align}\label{def.I_0}
		I_0(t,x):= &G_t\ast v_0(x)
		+\frac{\partial}{\partial t} G_t\ast u_0(x) \nonumber \\
		=&\frac 12\int_{x-t}^{x+t}v_0(y)dy+\frac{1}{2}[u_0(x+t)+u_0(x-t)]\,.
 \end{align}
\item[(ii)]\      We say   \eqref{eq.SWE}   has a \emph{weak solution}   if there exists a probability space with a filtration $(\widetilde{\Omega},\widetilde{\cF},\widetilde{\bP},\widetilde{\cF}_t)$,  an    $  \widetilde{\cF}_t  $-adapted  Gaussian random field $\widetilde{W}$ identical to $W$ in law,  and an  $  \widetilde{\cF}_t  $-adapted  random field\
   $\left\{u(t,x)\,, (t,x)\in\RR_+\times
    \RR\right\}$ on this probability space  $(\widetilde{\Omega},\widetilde{\cF},\widetilde{\bP},\widetilde{\cF}_t)$ such that $u(t,x)   $ is a mild solution with respect to
   $(\widetilde{\Omega},\widetilde{\cF},\widetilde{\bP},\widetilde{\cF}_t)$  and  $\widetilde{W}$.
   \end{enumerate}
\end{defn}

To obtain  the existence and uniqueness  of strong (mild) solution to \eqref{eq.SWE}, we make the following assumptions on $\sigma$.
\begin{enumerate}[start=1,label=\textbf{(H\arabic*)}]
  \item\label{H1}  $\sigma(t,x,u)$ is jointly continuous
  over {  $[0, T]\times \RR^2$}, $\sigma(t,x,0)=0$,   and it is    Lipschitz in $u$  (uniformly in $t$ and $x$).  This means
%  \begin{equation}\label{Lip}
%    \sup\limits_{t\in[0,T],x\in\RR}|\sigma(t,x,u)-\sigma(t,x,v)|\leq C (|u-v|)
%  \end{equation}
%for some positive constant $C $.  We also assume that  it is uniformly  Lipschitzian  in $u$, namely,
\  $\forall\  u,v\in\RR$
  \begin{equation}\label{LipSigam}
    \sup\limits_{t\in[0,T],x\in\RR}
    |\sigma(t,x,u)-\sigma(t,x,v)|\leq C |u-v|\,,
  \end{equation}
  for some constant $C >0$.
\end{enumerate}
  One easily  observes that the hypothesis \eqref{LipSigam} and the condition $\si(t,x,0)=0$ imply  that
\begin{equation}\label{LinearG}
	\sup\limits_{t\in[0,T],x\in\RR}|\sigma(t,x,u)|\leq C|u|\,,
\end{equation}
for some constant $C >0$.
%  where and throughout the paper, we use $A\ls B$ or $B\gs A$ to denote that there is a constant $C$ such that $A\le CB$ or $A \ge CB$.  If $A\ls B$ and $A\gs B$,   then we denote $A\es B$.
\begin{enumerate}[start=2,label=\textbf{(H\arabic*)}]
\item\label{H2} Assume
% that   $\sigma(t,x,u)\in\cC^{0,1,1}([0, T]\times\RR^2)$ satisfies the following conditions:
  $|\frac{\partial}{\partial u} \sigma(t,x,u)|$ and $|\frac{\partial^2}{\partial x \partial u}\sigma(t,x, u)|$ exist and are uniformly bounded, i.e. there is some constant  $C>0$ such that
  \begin{align}
  \sup_{t\in[0,T],x\in\RR,u\in\RR} \lt|\frac{\partial}{\partial u} \sigma(t,x, u)\rt| &\leq C\,; \label{DuSigam}\\
  \sup_{t\in[0,T],x\in\RR,u\in\RR} \lt|\frac{\partial^2}{\partial x \partial u}\sigma(t,x, u)\rt| &\leq C   \,. \label{DxuSigam}
  \end{align}
Moreover, we assume
  \begin{equation}\label{DuSigamAdd}
   \sup_{t\in[0,T],x\in\RR}\left| \frac{\partial}{\partial u} \sigma(t,x,u_1)-\frac{\partial}{\partial u} \sigma(t,x, u_2)\right| \le C  |u_1-u_2| \,.
   \end{equation}
\end{enumerate}
Notice that \eqref{LipSigam} is a consequence of \eqref{DuSigam}. But we keep the former one in the assumption
\ref{H1}  since we shall use  \ref{H1}  for the existence of the weak solution and \ref{H2}  for the existence and uniqueness of the strong solution.
%One can observe that the hypothesis \ref{H2} implies that

Now we   state the main results of this paper.
\begin{thm}\label{Thm_1}
%	Denoting
%	\begin{equation}\label{def.I_0}
%		I_0(t,x):=\frac 12\int_{x-t}^{x+t}v_0(y)dy+\frac{1}{2}[u_0(x+t)+u_0(x-t)]\,.
%	\end{equation}
%	$$I_0(t,x)=\frac 12\int_{x-t}^{x+t}v_0(y)dy+\frac{1}{2}[u_0(x+t)+u_0(x-t)]\,. $$
	Assume that $\sigma(t,x,u)$ satisfies the hypothesis \ref{H1} and that $I_0(t,x)$ is in $\cZ^p(T)$ for some $p>\frac{2}{4H-1}$. Then, there   exists a weak solution  to   \eqref{eq.SWE} whose  sample paths   are in $\cC([0,T]\times\RR)$ almost surely.  Moreover, for any $\gamma< H-\frac{1}{p}$, the process $u(t,x)$ is almost surely H\"older continuous of exponent $\gamma$ with respect to $t$ and $x$ on any compact subsets of  $[0, T]\times \RR$.
	\end{thm}

\begin{thm}\label{Thm_2}
	Assume that $\sigma(t,x,u)$ satisfies the  hypothesis \ref{H2} and that $I_0(t,x)$ is in $\cZ^p(T)$ for some $p>\frac{2}{4H-1}$. Then \eqref{eq.SWE} has a unique strong solution whose  sample paths are   in $\cC([0, T]\times\RR)$ almost surely.  Moreover,
	the random field  $u(t,x)$ is   H\"{o}lder continuous a.s. on compact subsets of   $[0, T]\times\RR$ with the same exponent as in Theorem \ref{Thm_1}.
\end{thm}

\begin{thm}\label{Thm_3}
If the hyperbolic Anderson model  \eqref{eq.HAE} has a solution in $\cZ^p(T)$ for some  $p\ge 2$ and for some $ T> 0$, then the Hurst parameter $H$ must satisfy  $H>1/4$.
%In other words, when $H\leq \frac 14$ \tcb{and $T\geq 2$, there exists some initial conditions $u_0(x)$ and $v_0(x)$ such that $I_0(t,x)\in\cZ^p(T)$} but the second moment of $v(t,x)$ when $t=2$ explodes, i.e.
%\begin{align}\label{secV}
%	\|v(2,x)\|_{L^p(\Omega)} \geq \|v(2,x)\|_{L^2(\Omega)}= \infty \quad \text{for all }x\in\RR \,.
%%	=&1+\sum_{n=1}^{\infty} n! \|\tilde{f}_n (\cdot;t,x)\|^2_{\cH^{\otimes n}}
%%	= 1+\sum_{n=1}^{\infty}\frac{1}{n!} \Phi_n (t)\,.
%\end{align}
\end{thm}

\section{Uniform moment bounds}\label{s.3}
In this section, we obtain  the uniform moment estimates of the stochastic convolution with the  noise $\dot{W}$ which  appears in the definition of the mild solution.
%Next, we can treat the uniform moment estimates of the solution to SPDE \eqref{eq.SWE} similarly.
These estimates are used later on  to prove the existence and uniqueness of solution to SWE \eqref{eq.SWE}.

\subsection{Uniform moment bounds of stochastic convolution}
%Firstly, we study the stochastic convolution with respect to the wave kernel.
Define
\begin{equation}\label{eq.StoConv}
  \Phi(t,x)=\int_{0}^{t}\int_{\RR} G_{t-s}(x-y)v(s,y)W(ds,dy)\,,
\end{equation}
where $G_t(x)$ is the Green's function associated with the wave operator \eqref{eq.SWE}, given by
\eqref{WaveKerG}.
%The stochastic integral can be understood as in \cite{HHLNT2017,HW2019}.

As we mentioned before, the major difficulty  here is that  the  wave Green's function  $G_t(x)$ does not satisfy the semigroup property so that the stochastic Fubini technique  used for stochastic heat equation is no longer  applicable (see Remark 4.3 in \cite{HW2019}). To get around this obstacle, we decompose it into  sum of convolutions of some `nice' kernels. More precisely, we have  the following simple and important lemma which is the key starting point of   our approach and which plays the role of semigroup property of the heat kernel when  the heat equation is investigated  (e.g. \cite{HHLNT2017, HW2019}).
\begin{lem}\label{kernel-sum}
The wave kernel $G_t(x)=\frac 12 \1_{\{|x|<t\}}$ can be expressed as
\begin{equation}\label{eq.SumKer}
  \begin{split}
  G_{t-s}(x-y)=&  \int_{\RR} \cC_{\beta}(t-r,x-z)\cS_{1-\beta}(r-s,z-y)dz \\
  +&\int_{\RR} \cS_{\alpha}(t-r,x-z)\cC_{1-\alpha}(r-s,z-y)dz \\
  +& \int_{\RR} \cS(t-r,x-z)\cE(r-s,z-y)dz \\
  +& \int_{\RR} \cE(t-r,x-z)\cS(r-s,z-y)dz\,,
  \end{split}
\end{equation}
where $\alpha,\beta\in(0,1)$, $\cS(t,x)=\cS_{1}(t,x)=G_t(x)=\frac 12 \1_{\{|x|<t\}}$ and
\begin{equation}\label{eq.cC_alpha}
	\begin{cases}
		\cE(t,x) :=\frac{1}{\pi} \frac{t}{t^2+x^2}\,, \\
  %\label{eq.cE}\\
  \cS_{\alpha}(t,x) :=\frac{\Gamma(1-\alpha)}{2\pi}\cos\lc\frac{\alpha\pi}{2}\rc \lk(t+|x|)^{\alpha-1}+\hbox{\rm sgn}(t-|x|)\big|t-|x|\big|^{\alpha-1}\rk\,, \\
  %,\label{eq.cS_alpha}\\
  \cC_{1-\alpha}(t,x) :=\frac{\Gamma(\alpha)}{2\pi}\bigg[\cos\lc\frac{\alpha\pi}{2}\rc\blk \big|t+|x|\big|^{-\alpha}+\big|t-|x|\big|^{-\alpha}\brk\\
 \qquad\qquad\qquad\qquad -2 \cos\lc \alpha\tan^{-1}\lc\frac{|x|}{t}\rc\rc [t^2+x^2]^{-\frac{\alpha}{2}}\bigg]\,.
	\end{cases}
\end{equation}
\end{lem}

\begin{proof}
We prove \eqref{eq.SumKer} via  Fourier transform
\[
\hat{f}(\xi)= \cF  [f(\xi)]= \int_\RR e^{-\iota x\xi}f(x) dx\,,\quad\hbox{where} \quad \iota=\sqrt{-1}\,.
\]
The   Fourier transform of  $G_{t+s}(x)$ is
\[
\hat{G}_{t+s}(\xi)=\frac{\sin((t+s)|\xi|)}{|\xi|}\,.
\]
We can decompose  $\hat{G}_{t+s}(\xi)$ into  the summation of  following four items:
\begin{align*}
  \hat{G}_{t+s}(\xi)&=\frac{\sin(t|\xi|)\cos(s|\xi|)}{|\xi|}+\frac{\sin(s|\xi|)\cos(t|\xi|)}{|\xi|}\\
    &=\frac{\sin(t|\xi|)}{|\xi|^{\alpha}}\cdot \frac{\cos(s|\xi|)-e^{-s|\xi|}}{|\xi|^{1-\alpha}}+\frac{\sin(t|\xi|)}{|\xi|}\cdot e^{-s|\xi|}\\
    &+\frac{\sin(s|\xi|)}{|\xi|^{\beta}}\cdot \frac{\cos(t|\xi|)-e^{-t|\xi|}}{|\xi|^{1-\beta}}+\frac{\sin(s|\xi|)}{|\xi|}\cdot e^{-t|\xi|} \,.
\end{align*}
On the other hand, the Fourier transforms of  $\cE(t,x)$, $\cS_{\alpha}(t,x)$ and $\cC_{1-\alpha}(t,x)$  are given as follows
(see Lemma \ref{Lem_A.1}):
\begin{equation}\label{eq.Fourier}
  \hat{\cE}(t,\xi)=e^{-t|\xi|}\,,\quad
  \hat{\cS}_{\alpha}(t,\xi)=\frac{\sin(t|\xi|)}{|\xi|^{\alpha}}\,,\quad
  \hat{\cC}_{1-\alpha}(t,\xi)=\frac{\cos(t|\xi|)-e^{-t|\xi|}}{|\xi|^{1-\alpha}}\,.
\end{equation}
We then  conclude  the proof of  \eqref{eq.SumKer} by the fact the Fourier transformation transforms the convolution to product.
\end{proof}

\begin{rmk}
Readers may wonder why we don't use the following simpler decomposition as we originally attempted:
\begin{align*}
  \hat{G}_{t+s}(\xi)&=\frac{\sin((t+s)|\xi|)}{|\xi|} \\
    &=\frac{\sin(t|\xi|)\cos(s|\xi|)}{|\xi|}+\frac{\sin(s|\xi|)\cos(t|\xi|)}{|\xi|}\\
    &=\frac{\sin(t|\xi|)}{|\xi|^{\alpha}}\cdot \frac{\cos(s|\xi|)}{|\xi|^{1-\alpha}}+\frac{\cos(t|\xi|)}{|\xi|^{\beta}}\cdot\frac{\sin(s|\xi|)}{|\xi|^{1-\beta}}\,.
\end{align*}
The reason is that the following quantity
\begin{align*}
	C_\beta(t,x):=\cF^{-1}\lk \frac{\cos(t|\xi|)}{|\xi|^{\beta}} \rk=c_\beta \lk(t+|x|)^{\beta-1}+|t-|x||^{\beta-1} \rk
\end{align*}
is not integrable. w.r.t. $x\in \RR$ when $0\leq\beta\leq1$.
\end{rmk}

Analogously to idea used in   \cite{HHLNT2017}, we shall seek the solution of \eqref{eq.SWE} in the space  $\cZ^p(T)$. To this end we need to bound   the   $\|\cdot\|_{\cZ^p(T)}$ norm   of the stochastic convolution $\Phi(t,x)$   defined by \eqref{eq.StoConv} and its variant $\cN_{\frac 12-H}\Phi(t,x)$   as stated in the following theorem.

\begin{prop}\label{prop_est}
  For the stochastic convolution $\Phi(t,x)$, we have the following estimates:
  \begin{enumerate}[leftmargin=*]
  \item[\textbf{(i)}] If  $p>\frac{1}{H}$,
 % (and $1-H<\alpha<1-\frac 1p$)
  then
     \begin{equation}\label{Est.Zl}
       \Big\|\sup\limits_{t\in[0,T], x\in \RR}\ \left|\Phi(t,x)\right| \Big\|_{L^p(\Omega)}\leq C_{T,p,H}\|v\|_{\cZ^p(T)}\,.
     \end{equation}
  \item[\textbf{(ii)}] If  $p>\frac{2}{4H-1}$,
  %(and $\frac 32-2H<\alpha<1-\frac 1p$),
   then
     \begin{equation}\label{Est.Z2.N}
      \Big\|\sup\limits_{t\in[0,T], x\in \RR}\left|  \cN_{\frac 12-H}\Phi(t,x)\right| \Big\|_{L^p(\Omega)} \leq C_{ T,p,H}\|v\|_{\cZ^p(T)}\,.
     \end{equation}
   \end{enumerate}
\end{prop}
\begin{proof} We shall use Lemma \ref{kernel-sum} to prove this proposition.  We divide the proof into two  steps.

\noindent \textbf{Step 1:} In this step, we shall prove part \textbf{(i)}  of the proposition.
% In the part \textbf{(i)}, we need to consider $\sup\limits_{t\in[0,T], x\in \RR}\left| \Phi(t,x)\right| $ which can be viewed as the size of the random field $\Phi(t,x)$.
 For any $\theta\in(0,1)$ and $i=1,2,3,4$,   set
  \begin{align}
    J^{\cK_{i}}_{\theta}(r,z):=\int_{0}^{r}\int_{\RR}(r-s)^{-\theta}\cK_{i}(r-s,z-y)v(s,y)W(dy,ds)\,, \label{def.JK_theta}
  \end{align}
  where
  \begin{equation}\label{def.cK}
  	\cK_{1}=\cC_{\alpha},~\cK_{2}=\cS_{\alpha},~\cK_{3}=\cS,~\text{and}~\cK_{4}=\cE.
  \end{equation}
%  $\cK_{1}=\cC_{\alpha},\cK_{2}=\cS_{\alpha},\cK_{3}=\cS,\cK_{4}=\cE$.
  %\tcg{ The choice of $\cK_{\alpha}$ is among $\cC_{\alpha}$, $\cS_{\alpha}$, $\cS$ or $\cE$ introduced as before.
  And we define $\bar{\cK}_{i}$ to be the complements  of $\cK_{i}$ according to \eqref{eq.SumKer}, namely,
  \begin{equation}\label{def.cK_conj}
  	\bar{\cK}_{1}=\cS_{1-\alpha},~\bar{\cK}_{2}=\cC_{1-\alpha},~\bar{\cK}_{3}=\cE,~\text{and}~\bar{\cK}_{4}=\cS.
  \end{equation}
%  $\bar{\cK}_{1}=\cS_{1-\alpha}$, $\bar{\cK}_{2}=\cC_{1-\alpha}$, $\bar{\cK}_{3}=\cE$ and $\bar{\cK}_{4}=\cS$.
Let us set
\begin{align*}
   \Phi_{i}(t,x):=\frac{\sin(\theta\pi)}{\pi} \int_{0}^{t}\int_{\RR}(t-r)^{\theta-1}\bar{\cK}_{i}(t-r,x-z)J^{\cK_{i}}_{\theta}(r,z)dzdr \,,\quad  i=1,2,3,4 \,.
%    \frac{\sin(\theta\pi)}{\pi}\int_{0}^{t}\int_{\RR}(t-r)^{\theta-1}\cS_{1-\alpha}(t-r,x-z)J^{\cC_{\alpha}}_{\theta}(r,z)dzdr \\
%      &+\frac{\sin(\theta\pi)}{\pi}\int_{0}^{t}\int_{\RR}(t-r)^{\theta-1}\cC_{1-\alpha}(t-r,x-z)J^{\cS_{\alpha}}_{\theta}(r,z)dzdr \nonumber \\
%      &+\frac{\sin(\theta\pi)}{\pi}\int_{0}^{t}\int_{\RR}(t-r)^{\theta-1}\cS(t-r,x-z)J^{\cE}_{\theta}(r,z)dzdr \nonumber \\
%      &+\frac{\sin(\theta\pi)}{\pi}\int_{0}^{t}\int_{\RR}(t-r)^{\theta-1}\cE(t-r,x-z)J^{\cS}_{\theta}(r,z)dzdr\,. \nonumber
  \end{align*}
Then a stochastic version of Fubini's theorem and Lemma \ref{kernel-sum} yield
%that $\Phi(t,x)$ can be expressed by $\sum_{i=1}^4\Phi_{i}%(t,x)$,
\begin{align}\label{eq.Fubini}
\Phi(t,x)=& \int_{0}^{t}\int_{\RR} G_{t-s}(x-y)v(s,y)W(ds,dy)\nonumber\\
=& \frac{\sin(\theta\pi)}{\pi} \int_{0}^{t}\int_{\RR} \int_{s}^{t} (t-r)^{\theta-1} (r-s)^{-\theta} dr \times G_{t-s}(x-y) v(s,y)W(dy,ds)\nonumber\\
=& \sum_{i=1}^{4}\frac{\sin(\theta\pi)}{\pi} \int_{0}^{t}\int_{\RR} \int_{s}^{t}\int_{\RR} (t-r)^{\theta-1} (r-s)^{-\theta} \nonumber\\
&\qquad\qquad\qquad\times \bar{\cK}_{i}(t-r,x-z)\cK_{i}(r-s,z-y)dzdr\times v(s,y)W(dy,ds) \nonumber \\
%=&\sum_{i=1}^{4}\frac{\sin(\theta\pi)}{\pi} \int_{0}^{t}\int_{\RR}(t-r)^{\theta-1} \bar{\cK}_{i}(t-r,x-z)\nonumber \\
%	&\qquad\times \int_{0}^{r}\int_{\RR}(r-s)^{-\theta} \cK_{i}(r-s,z-y)v(s,y)W(dy,ds) dzdr\nonumber\\
=&\sum_{i=1}^{4}\frac{\sin(\theta\pi)}{\pi} \int_{0}^{t}\int_{\RR}(t-r)^{\theta-1} \bar{\cK}_{i}(t-r,x-z)J^{\cK_{i}}_{\theta}(r,z)dzdr\nonumber\\
=&\sum_{i=1}^{4}\Phi_{i}(t,x)\,,
\end{align}
where we have applied the identity
\[
 \int_s^t (t-r)^{\theta-1} (r-s)^{-\theta} dr=\frac{\pi}{\sin(\theta\pi)}\,, \quad \theta\in(0,1)\,,0\leq s\leq t\,.
\]
This expression is essential  for us  to derive the desired estimates.
% Although it is more complicated than the Stochastic Heat Equation case \cite{HHLNT2017, HW2019}, we can follow some techniques there.
 In the following, we will use $\sum_{i}$ to denote $\sum_{i=1}^{4}$ and $\sup\limits_{t,x}$ to denote $\sup\limits_{t\in[0,T], x\in \RR} $.

  It is clear    by the H\"{o}lder inequality with $1/p+1/q=1$  that for $i=1,\cdots,4$
  \begin{align}
    \sup\limits_{t, x}\ \left|\Phi_{i}(t,x)\right|
    \lesssim&\  \sup\limits_{t, x}\ \int_{0}^{t}(t-r)^{\theta-1}\lc\int_{\RR}|\bar{\cK}_{i}(t-r,x-z)|^{q} dz\rc^{\frac 1q} \nonumber\\
    &\qquad\qquad\qquad\qquad\qquad\quad\times\|J^{\cK_{i}}_{\theta}(r,z)\|_{L^p(\RR)}dr \nonumber\\
    \lesssim&  \lc\sup\limits_{t}\
    \int_{0}^{t}\int_{\RR}r^{q(\theta-1)}|\bar{\cK}_{i}(r,z)|^{q} dzdr\rc^{\frac 1q} \nonumber\\
    %\label{Ineq.Phi}\\
    &\qquad\qquad\qquad\qquad\qquad \times\lc\int_{0}^{T}\|J^{\cK_{i}}_{\theta}(r,z)\|^p_{L^p(\RR)}dr\rc^{\frac 1p}\nonumber\\
    =&(I_{i}^{(1)})^{1/q}\times (I_{i}^{(2)})^{1/p}  \,,
    %\label{J_theta}\,.
%    :=& \sup\limits_{t, x}\lambda^{\frac 1p}(x) \int_{0}^{t}(t-r)^{\theta-1}\lc\int_{\RR}|\cS_{1-\alpha}(t-r,x-z)|^{q}\lambda(z)^{-\frac qp}dz\rc^{\frac 1q}\|J^{\cC_{\alpha}}_{\theta}(r,z)\|_{L_{\lambda}^p(\RR)}dr \nonumber \\
%      +& \sup\limits_{t, x}\lambda^{\frac 1p}(x) \int_{0}^{t}(t-r)^{\theta-1}\lc\int_{\RR}|\cC_{1-\alpha}(t-r,x-z)|^{q}\lambda(z)^{-\frac qp}dz\rc^{\frac 1q}\|J^{\cS_{\alpha}}_{\theta}(r,z)\|_{L_{\lambda}^p(\RR)}dr \nonumber\\
%      +& \sup\limits_{t, x}\lambda^{\frac 1p}(x) \int_{0}^{t}(t-r)^{\theta-1}\lc\int_{\RR}|\cS(t-r,x-z)|^{q}\lambda(z)^{-\frac qp}dz\rc^{\frac 1q}\|J^{\cE}_{\theta}(r,z)\|_{L_{\lambda}^p(\RR)}dr \nonumber\\
%      +& \sup\limits_{t, x}\lambda^{\frac 1p}(x) \int_{0}^{t}(t-r)^{\theta-1}\lc\int_{\RR}|\cE(t-r,x-z)|^{q}\lambda(z)^{-\frac qp}dz\rc^{\frac 1q}\|J^{\cS}_{\theta}(r,z)\|_{L_{\lambda}^p(\RR)}dr \nonumber
  \end{align}
where we change the variables $r\to t-r$ and $z\to x-z$ in the second inequality and then it is clear that $\sup_{t,x}$ becomes $\sup_t$ thanks to the translation invariance in space variable of the function.  This  technique will be freely used in the sequel   without mention.  We shall deal with $I_{i}^{(1)},I_{i}^{(2)}$, $i=1, \cdots, 4$,  term by term in the subsequent paragraphs.

First,   let us deal with $I_{i}^{(1)}$  when $i=1,2$.  The cases   $i=3,4$ can be  treated similarly. When $i=1,$ $\cK_{1}=\cC_{\alpha}$ and $\bar{\cK}_{1}=\cS_{1-\alpha}$ defined as \eqref{eq.cC_alpha}. By the change of variable  $z\to rz$, it is easy to see $I_{1}^{(1)}$  can be bounded as
  \begin{align*}
 I_{1}^{(1)}=&   \sup\limits_{t}  \int_{0}^{t}\int_{\RR}r^{q(\theta-1)}|\cS_{1-\alpha}(r,z)|^{q} dzdr \\
%   \lesssim& \sup\limits_{t,x} \int_{0}^{t}\int_{\RR}(t-r)^{q(\theta-1)}|\cS_{1-\alpha}(t-r,x-z)|^{q} dzdr \\
%   \lesssim& \sup\limits_{t,x}\lambda^{\frac qp}(x)|x|^{\frac {q\lambda}{p}}\int_{0}^{t}\int_{\RR}(t-r)^{q(\theta-1)}|\cS_{1-\alpha}(t-r,z)|^{q}|z|^{\frac {q\lambda}{p}}dzdr \\ &\qquad+\sup_{t}\int_{0}^{t}\int_{\RR}(t-r)^{q(\theta-1)}|\cS_{1-\alpha}(t-r,z)|^{q}dzdr \\
%   \lesssim& \int_{\RR}\lt|\lc t-r+|z|\rc^{-\alpha}+ \sgn(t-r-|z|)\big|t-r-|z|\big|^{-\alpha} \rt|^q |z|^{\frac {q\lambda}{p}}dz \\
%   +&\int_{\RR}\lt|\lc t-r+|z|\rc^{-\alpha}+ \sgn(t-r-|z|)\big|t-r-|z|\big|^{-\alpha} \rt|^q dz \\
   \lesssim&\lk\sup_{t}\ \int_{0}^{t}r^{q\lc\theta-1-\alpha+\frac 1q\rc} dr\rk\times\int_{0}^{\infty}\lt|\lc 1+|z|\rc^{-\alpha}+\sgn(1-|z|)\big|1-|z|\big|^{-\alpha}\rt|^{q}  dz\,.
%   &\qquad+(t-r)^{1-\alpha q}\int_{0}^{\infty}\lt|\lc 1+|z|\rc^{-\alpha}+\sgn(1-|z|)\big|1-|z|\big|^{-\alpha}\rt|^{q}dz\,.
  \end{align*}
  In order to make sure the above integrals converge, we need
  \begin{equation}\label{Ineq.alpha_I1}
    \alpha q<1\,,~(\alpha+1)q>1\quad  \Leftrightarrow  \quad 0< \alpha<\frac 1q=1-\frac 1p\,,
  \end{equation}
%  Therefore by applying H\"{o}lder inequality again, we can get
%  \begin{align*}
%    &\sup\limits_{t, x}\lambda^{\frac 1p}(x) \int_{0}^{t}(t-r)^{\theta-1}\lc\int_{\RR}|\cS_{1-\alpha}(t-r,x-z)|^{q}\lambda(z)^{-\frac qp}dz\rc^{\frac 1q}\|J^{\cC_{\alpha}}_{\theta}(r,z)\|_{L_{\lambda}^p(\RR)}dr \\
%     \lesssim& \sup_{t}\int_{0}^{t}(t-r)^{\theta-\alpha+\frac{\lambda}{p}-1+\frac 1q} \|J^{\cC_{\alpha}}_{\theta}(r,z)\|_{L_{\lambda}^p(\RR)}dr \\
%     &\qquad\qquad\qquad+\sup_{t}\int_{0}^{t}(t-r)^{\theta-\alpha-1+\frac 1q}\|J^{\cC_{\alpha}}_{\theta}(r,z)\|_{L_{\lambda}^p(\RR)}dr \\
%     \lesssim& \bigg(\lk\sup_{t}\int_{0}^{t}(t-r)^{q\lc\theta-\alpha+\frac{\lambda}{p}-1+\frac 1q\rc} dr\rk^{\frac 1q} \\
%     &\qquad\qquad\qquad+\lk\sup_{t}\int_{0}^{t}(t-r)^{q\lc\theta-\alpha-1+\frac 1q\rc} dr\rk^{\frac 1q}\bigg) \lk\int_{0}^{T}\|J^{\cC_{\alpha}}_{\theta}(r,z)\|_{L_{\lambda}^p(\RR)}^p dr\rk^{\frac 1p}\,.
%  \end{align*}
  and also
  \begin{equation}\label{Ineq.theta_I1}
    q\lc\theta-\alpha-1+\frac 1q\rc>-1
    \quad \Leftrightarrow \quad  \theta>1-\frac2q+\alpha \,.
  \end{equation}
%  Thus the Lemma \ref{Lem.Est_J} implies the validity of \textbf{(i)}.
%  we only need to prove for some constant $C$ independent of $r\in[0,T]$ such that
%  \begin{equation}\label{Est.Char1}
%    \EE\|J^{\cC_{\alpha}}_{\theta}(r,\cdot)\|_{L^p(\RR)}^p \leq C\|v\|_{Z^p(T)}^p\,.
%  \end{equation}

%  of second line in \eqref{Ineq.Phi}.

When $i=2$, $\cK_{2}=\cS_{\alpha}$ and $\bar{\cK}_{2}=\cC_{1-\alpha}$ which are defined in \eqref{eq.cC_alpha}, we have
  \begin{align}\label{C_alpha}
  I_{2}^{(1)}=&   \sup\limits_{t}\   \int_{0}^{t}\int_{\RR}r^{q(\theta-1)}|\cC_{1-\alpha}(r,z)|^{q} dzdr \\
      \lesssim&\lk\sup_{t}\ \int_{0}^{t}r^{q\lc\theta-1-\alpha+\frac 1q\rc} dr\rk \nonumber \\
    &\times\int_{0}^{\infty}\bigg[\cos\lc\frac{\alpha\pi}{2}\rc \lk\big|1+|z|\big|^{-\alpha}+\big|1-|z|\big|^{-\alpha}\rk\\
    &\qquad\qquad-2\cos\lc\alpha\tan^{-1}(z)\rc \blk 1+z^2\brk^{-\frac{\alpha}{2}}\bigg]^q dz \nonumber\,.
  \end{align}
By  Lemma \ref{Lem_A.1} in the  Appendix \ref{s.6}, $\cC_{1-\alpha}(r,z)$ can be bounded by
\begin{align*}
|\cC_{1-\alpha} (r,z)|
\lesssim
\begin{cases}  \big|r+|z|\big|^{-\alpha}+\big|r-|z|\big|^{-\alpha}+\blk r^2+z^2\brk^{-\frac{\alpha}{2}} &  \hbox{if $|z|\approx r$}\,, \\
  r\blc |z|^2-r^2\brc^{-\frac{\alpha}{2}-1} &
  \hbox{if $|z|\approx\infty$\,.}
\\
\end{cases}
\end{align*}
Thus, in order to make sure \eqref{C_alpha} is bounded, we need
\begin{equation}\label{Ineq.alpha_I2_2}
    \alpha q<1\,,~(\alpha+2)q>1\,\Leftrightarrow 0< \alpha<\frac 1q\,,
  \end{equation}
and
\begin{equation}\label{Ineq.theta_I2}
    q\lc\theta-\alpha-1+\frac 1q\rc>-1\Leftrightarrow \theta>1-\frac2q+\alpha\,.
  \end{equation}
Therefore, to prove part \textbf{(i)} of the proposition we only need to show
 \begin{equation*}
  \EE\|J^{\cK_{i}}_{\theta}(r,\cdot)\|_{L^p(\RR)}^p \leq C\|v\|_{\cZ^p(T)}^p,\ \ i=1,2,3,4\,.
  \end{equation*}
%In order to make the paper more smooth, we give the estimation of \eqref{Est.CharI2}
This is  objective of  Lemma \ref{Lem.Est_J}, proved in the Appendix \ref{Lemma for 3.3} under the  following condition:
\begin{equation}\label{condi_J_est}
p>\frac 1H, \ 1-\frac2q+\alpha<\theta<H+\alpha-\frac12, \ \ 1-H<\alpha<1-\frac 1p.
\end{equation}
Therefore, when  $p>\frac 1H$, we can choose $\alpha$ such that  $1-H<\alpha<1-\frac 1p$, and then we see \eqref{Ineq.alpha_I1}, \eqref{Ineq.theta_I1}, and \eqref{condi_J_est} are satisfied since  $\frac 1H>\frac{4}{2H+1}$ if $H<\frac 12$. Thus we have proved   \textbf{(i)} of the proposition for $\Phi_1(t,x)$, $\Phi_2(t,x)$.   The cases for  $\Phi_3(t,x)$ and  $\Phi_4(t,x)$ can be proved similarly. Thus, we complete the proof of part  \textbf{(i)} of the proposition.
% and $1-\frac2q+\alpha<\theta<H+\alpha-\frac12$ which implies $p>\frac{4}{2H+1}$,  we see \eqref{Ineq.alpha_I1}, \eqref{Ineq.theta_I1}, and \eqref{condi_J_est} are satisfied by noticing that $\frac 1H>\frac{4}{2H+1}$ if $H<\frac 12$. Thus we have proved the validity of \textbf{(i)}.

%\tcg{ The same estimate can be done by the exact method when $\cK_{2}=\cS_{\alpha}$ and $\bar{\cK}_{2}=\cC_{1-\alpha}$, so we skip the details.}

%  The rest of \eqref{Ineq.Phi}, i.e. $\cK_{3}=\cS$ and $\cK_{4}=\cE$ can be done in the same spirit even simpler.

\medskip
\noindent \textbf{Step 2:}  Let us now consider part \textbf{(ii)} of the proposition. In order to obtain the desired decay    rate of $\cN_{\frac 12-H}\Phi(t,x)$, we still use the equation  \eqref{eq.Fubini} to express $\Phi(t,x)$   by $J_{\theta}^{\cK_i}$.  Recall our notation  $\fD_{h} \Phi(t,x):=\,\Phi(t,x+h) -\Phi(t,x)$ and same notations  for $\fD_{h}\bar{\cK}_{i}(t-r,z)$, $\fD_{h}J_{\theta}^{\cK_{i}}(r,z)$. Then
  \begin{align}\label{eq.Fubini.N}
%       \fD_{h} \Phi(t,x):=\,&\Phi(t,x+h) -\Phi(t,x)\\
      \fD_{h} \Phi(t,x)=\,&\frac{\sin(\theta\pi)}{\pi}\sum_{i}\int_{0}^{t}\int_{\RR}(t-r)^{\theta-1}  \fD_{h}\bar{\cK}_{i}(t-r, x-z)  J_{\theta}^{\cK_{i}}(r,z)dzdr\nonumber\\
      \es\,&\sum_{i}\int_{0}^{t}\int_{\RR}(t-r)^{\theta-1}\bar{\cK}_{i}(t-r,x-z)\fD_{h}J_{\theta}^{\cK_{i}}(r,z) dzdr\,,
  \end{align}
  with the choice of $\cK$ and $\bar{\cK}$ defined by  \eqref{def.cK} and \eqref{def.cK_conj}. Invoking   Minkowski's inequality and then H\"{o}lder's  inequality
  % with $\frac 1p + \frac 1q=1$,
  we get
  \begin{align}
      \sup_{t,x}&\ \lc\int_{\RR}|\fD_{h} \Phi(t,x)|^2|h|^{2H-2}dh\rc^{\frac 12}\nonumber\\
      &\ls\sup_{t,x}\ \sum_{i}\bigg(\int_{\RR}\bigg|\int_{0}^{t}\int_{\RR}(t-r)^{\theta-1}\bar{\cK}_{i}(t-r,x-z) \nonumber\\
      &\qquad\qquad\qquad\qquad\qquad\times \fD_{h}J_{\theta}^{\cK_{i}}(r,z) dzdr\bigg|^2\cdot |h|^{2H-2}dh\bigg)^{\frac 12}\nonumber \\
      &\ls\sup_{t,x}\ \sum_{i} \int_{0}^{t}\int_{\RR}(t-r)^{\theta-1}|\bar{\cK}_{i}(t-r,x-z)| \nonumber\\
      &\qquad\qquad\qquad\qquad\qquad\times\lc\int_{\RR}\Big|\fD_{h}J_{\theta}^{\cK_{i}}(r,z)\Big|^2|h|^{2H-2}dh\rc^{\frac 12}dzdr  \nonumber \\
      &\ls\sum_{i} \left(\sup_{t}\ \int_{0}^{t}\int_{\RR}r^{q(\theta-1)}\lt|\bar{\cK}_{i}(r,z)\rt|^q  dzdr\right)^{\frac 1q}\nonumber\\
      &\qquad\qquad\times  \lc\int_{0}^{T}\int_{\RR}\lk\int_{\RR}\Big|\fD_{h}J_{\theta}^{\cK_{i}}(r,z)\Big|^2|h|^{2H-2}dh\rk^{\frac p2} dz dr\rc^{\frac 1p}\nonumber\\
      &=:(J_{i}^{(1)})^{\frac 1q}\times(J_{i}^{(2)})^{\frac 1p}\,.\label{K_k}
  \end{align}

  The first factor $(J_{i}^{(1)})^{\frac 1q}$ in  \eqref{K_k}  is finite if we require that  $\alpha,\theta, p,q$ satisfy  \eqref{Ineq.alpha_I1} and \eqref{Ineq.theta_I1}. Therefore we only need to focus on the second factor $(J_{i}^{(2)})^{\frac 1p}$ in  \eqref{K_k}. By Lemma \ref{Lem.Est_D},   we see
  \begin{equation*}
   \EE \int_{\RR}\Blk\int_{\RR}\lt|\fD_{h}J_{\theta}^{\cK_{i}}(r,z)\rt|^2|h|^{2H-2}dh\Brk^{\frac p2}dz \leq C_{T,p,\alpha,\theta}  \|v\|^p_{\cZ^p(T)}\,,
  \end{equation*}
  under the conditions
  \begin{equation}\label{condi_D_est}
   p>\frac 1H, \ 1-2/q+\alpha<\theta<2H+\alpha-1, \ \  \frac 32-2H<\alpha<1-\frac 1p.
\end{equation}
If  $p>\frac{2}{4H-1}$,  then we can choose $\alpha$
 such that $\frac 32-2H<\alpha<1-\frac 1p$, and then we see \eqref{Ineq.alpha_I1}, \eqref{Ineq.theta_I1} and \eqref{condi_D_est} are satisfied since  $\frac{2}{4H-1}>\frac 1H$ when $H<\frac 12$. Thus, we complete the proof of part \textbf{(ii)} of the proposition.
%  which implies $p>\frac{1}{H}$, and $\frac 32-2H<\alpha<1-\frac 1p$ and noting that $\frac{2}{4H-1}>\frac 1H$ when $H<\frac 12$, we see \eqref{Ineq.alpha_I1}, \eqref{Ineq.theta_I1}, \eqref{Ineq.alpha_I2}, and \eqref{condi_D_est} are satisfied. Thus, we complete the proof of \textbf{(ii)}.
\end{proof}

\subsection{Uniform moment bounds of the approximate solutions}
We approximate  the noise $ W$  by the following smoothing of the noise  with respect to the space variable.
 That   is, for $\ep>0$ we define
 \begin{equation}\label{Regu}
     \frac{\partial }{\partial x} W_{\ep}(t,x) = \int_{\RR}  \rho_{\ep} (x-y)W(t,dy)\,,
 \end{equation}
where $\rho_{\ep}(x)=\frac{1}{\sqrt{2\pi\ep}}\exp(-\frac{x^2}{2\ep})$. The regulated noise $W_\ep$ induces an approximation of  mild solution
 \begin{equation}\label{MildSolRegu}
   u_\ep(t,x)=I_0(t,x)+\int_{0}^{t}\int_{\RR} G_{t-s}(x-y)\sigma(s,y,u_\ep(s,y))W_{\ep}(ds,dy),
 \end{equation}
 where the stochastic
 integral is understood in the It\^{o} sense.
Due to the regularity in space of the noise, the existence and uniqueness of the solution $u_\ep(t,x)$ to above equation is standard (even the higher dimensional case were known (e.g. \cite{HHN2014,Peszat2002} and references therein).

The lemma below asserts that the approximate solution
 $\left\{ u_\ep(t,x)\,, \varepsilon>0\right\}$  is uniformly bounded in the space $\cZ^p(T)$.  More  precisely, we have

\begin{lem}\label{weak_est}
 Let $H\in(\frac 14,\frac 12)$. Assume  that $\sigma(t,x,u)$ satisfies the hypothesis \ref{H1}  and assume that $I_0(t,x)$ is in $  \cZ^p(T)$. Then the approximate solutions $u_\ep$ satisfy
   \begin{equation}\label{ReguBdd}
     \sup_{\ep>0}\|u_\ep\|_{\cZ^p(T)}:=\sup_{\ep>0}\|u_\ep(t,\cdot)\|_{\cZ_1^p(T)} +\sup_{\ep>0}\|u_\ep(t,\cdot)\|_{\cZ_2^p(T)}<\infty.
   \end{equation}
\end{lem}
%\sup_{t\in[0,T]}\|u_\ep(t,\cdot)\|_{L^p(\Omega\times\RR)}
%\sup_{t\in[0,T]}\cN^*_{\frac 12-H,p}u_\ep(t)

\begin{proof}%[Proof of Theorem \ref{UniBExist}]
For notational simplicity  we   assume $\sigma(t,x,u)=\sigma(u)$ without loss of generality because of hypothesis \ref{H1}.  We shall use some     thoughts  similar to   those  in  \cite{HW2019}.
%\deleted{In the following steps 1 to 2  we will use Picard’s iteration to show that for each $\ep$, $u_\ep\in\cZ_{\lambda,T}^p$ for $p\geq 2$. Then,  in step 3 we   prove that  $u_\ep$ is uniformly     bounded in $\cZ_{\lambda,T}^p$ with respect to   $\ep\in (0, 1]$.}
To this end, we define the Picard iteration  as follows:
   \[
    u_\ep^0(t,x)=I_0(t, x)\,,
   \]
   and recursively  for $n=0, 1, 2, \cdots$,
\begin{equation}
    u_\ep^{n+1}(t,x)=I_0(t, x)+\int_{0}^{t}\int_{\RR} G_{t-s}(x-y)\sigma(u_\ep^{n}(s,y))W_{\ep}(ds,dy)\,. \label{Def.u_n+1}
   \end{equation}
From \cite[Lemma 4.12]{HHLNT2019}  it follows that  for  any fixed $\ep>0$
 when $n$ goes to infinity,  the  sequence $u_{\ep}^n(t,x)$ converges to $u_{\ep}(t,x)$ a.s.
 In the following steps 1 and 2, we shall first  bound $\|u_\ep^n\|_{\cZ^p(T)}$ uniformly in $n$,    and $\ep$.
Then,  in step 3 we
%prove that  $u_\ep$ is uniformly     bounded in $\cZ_{\lambda,T}^p$ with respect to   $\ep\in (0, 1]$. and then we
use Fatou's lemma to show \eqref{ReguBdd}.

In the following, we will continue to  use the notations $\fD_hf(t,x)$ and $\Box_{h,l}f(t,x)$  previously defined in \eqref{space_dif} and \eqref{box_diff}.

%\deleted{In the following, we will use the notations $D_t(x,h)$ and $\Box_{t-s}(x,y,h)$ defined in \eqref{FirDiff} and \eqref{SecDiff} frequently.}
%need to bound the terms $\|u_\ep^n(t,\cdot)\|_{L^p_{\lambda}(\Omega\times\RR)}$ and $\cN^*_{\frac 12-H,p} u^n_\ep(t)$.

   \noindent{\textbf{Step 1. }}\
   In this step, we bound  the $L^p(\Omega\times\RR)$ norm  of $u^{n+1}_\ep(t,x)$ by the $\cZ^p$  norm  of $u^{n}_\ep(t,x)$.
%    are   uniformly bounded in $n$ and $\ep>0$.
    Rewrite \eqref{Def.u_n+1} as
   \[
    u^{n+1}_\ep(t,x)=I_0(t, x)+\int_{0}^{t}\int_{\RR} \Blk\Blc G_{t-s}(x-\cdot)\sigma(u^{n}_\ep(s,\cdot))\Brc\ast G_\ep\Brk(y)W(ds,dy)\,.
   \]
%   Moreover, denote
%   \begin{align*}
%   		D_t(x,h):=&\ G_t(x+h)-G_t(x)\,; \\
%   		\Box_t(x,y,h):=&\ G_t(x+y+h)-G_t(x+y)-G_t(x+h)+G_t(x).
%   \end{align*}
%   $$\Box_t(x,y,h):=G_t(x+y+h)-G_t(x+y)-G_t(x+h)+G_t(x).$$

 Using  $e^{-\varepsilon |\xi|^2}\leq 1$
%   for the first inequality
%   \footnote{Please use the label names. It is not this proposition we should use? It is should be the  BDG inequality? Please make sure anything you write is correct!}  and using inner product \eqref{hinner.3} to replace  \eqref{hinner.1}  for the second inequality,
 and the condition \eqref{LinearG} on $\sigma$,  we have from the Burkholder-Davis-Gundy   inequality \eqref{e.bdg}
   \begin{align}\label{uepLp}
        &\EE\blk|u^{n+1}_\ep(t,x)|^p\brk \nonumber\\
     \leq& C_p |I_0(t, x)|^p
       +C_p\EE\lc\int_{0}^{t}\int_\RR \Big|\cF\blk G_{t-s}(x-\cdot)\sigma(u_\ep(s,\cdot)) \brk (\xi)\Big|^2  e^{-\varepsilon |\xi|^2}| \xi|^{1-2H} d\xi ds\rc^{\frac p2}  \nonumber\\
       \leq& C_p|I_0(t, x)|^p
       + C_p\EE\bigg(\int_{0}^{t}\int_{\RR^2}  \Big|G_{t-s}(x-y-h)\sigma(u^{n}_\ep(s,y+h)) \nonumber\\
       &\qquad\qquad\qquad\qquad\qquad\qquad\qquad-G_{t-s}(x-y)\sigma(u^{n}_\ep(s,y))\Big|^2 |h|^{2H-2}dhdyds\bigg)^{\frac p2} \nonumber\\
       \leq & C_p\lk |I_0(t, x)|^p+\lt| \cD^{\ep,n}_1(t,x) \rt|^{\frac p2}+\lt| \cD^{\ep,n}_2(t,x)\rt|^{\frac p2} \rk\,,%+\cD_3(t,x),
   \end{align}
%  where the constant  $C_p$ is independent  of $\ep$,
%   and may vary from place to place, and
where we have used the notations $\cD^{\ep,n}_1(t,x)$ and $\cD^{\ep,n}_2(t,x)$ similar to \eqref{eq.D1_def} and \eqref{eq.D2_def}, namely,
   \begin{align*}
     \cD^{\ep,n}_1(t,x) &:=\int_{0}^{t}\int_{\RR^2} \big|\fD_{h}G_{t-s}(y)\big|^2 \cdot \|u^{n}_\ep(s,x+y )\|_{L^p(\Omega)}^2  |h|^{2H-2}dhdyds \,,
\end{align*}
   and
   \begin{align*}
     \cD^{\ep,n}_2(t,x) &:= \int_{0}^{t}\int_{\RR^2}|G_{t-s}(y)|^2\|\fD_h u_{\ep}^n(t,x+y)\|_{L^p(\Omega)}^2
             |h|^{2H-2} dhdyds\,.
   \end{align*}
   This means
   \begin{align}
      \|u^{n+1}_{\ep}(t,\cdot)\|_{L^p(\Omega\times\RR)}^2 =& \lc \int_{\RR} \EE\blk|u^{n+1}_\ep(t,x)|^p\brk dx \rc^{\frac 2p} \nonumber\\
     \leq&\ C_p\lk \|I_0(t, x)\|_{L^p(\Omega\times\RR)}^2+ D^{\ep,n}_1(t)+D^{\ep,n}_2(t)\rk\,,
     \label{Est.u_n+1_Z1}
%      \lc\int_\RR \cD_1(t,x) \lambda(x)dx \rc^{\frac 2p} \\
%     & \qquad \qquad + \lc\int_\RR \cD_2(t,x) \lambda(x) dx \rc^{\frac 2p}+\lc\int_\RR \cD_3(t,x) \lambda(x) dx \rc^{\frac 2p}.
   \end{align}
where $D^{\ep,n}_1(t)$ and $D^{\ep,n}_2(t)$ are defined and can be  bounded similar to the argument  used in the proof of  Lemma \ref{Lem.Est_J}:
\begin{equation}\label{D1Bdd}
    \begin{split}
 	D^{\ep,n}_1(t):= \lc\int_{\RR} \lt| \cD^{\ep,n}_1(t,x) \rt|^{\frac p2} dx \rc^{\frac 2p}\leq& C_{p,H}\int_{0}^{t} (t-s)^{2H} \|u^n_{\ep}(s,\cdot)\|^2_{L^p(\Omega\times\RR)} ds\,,
    \end{split}
\end{equation}
and
\begin{equation}\label{D1Bdd2}
    \begin{split}
    D^{\ep,n}_2(t):= &\lc \int_{\RR} \lt| \cD^{\ep,n}_2(t,x)\rt|^{\frac p2} dx\rc^{\frac 2p}\leq C_{p,H} \int_{0}^{t}(t-s)\lt[\cN^*_{\frac 12-H,p}u^n_\ep(s)\rt]^2 ds\,.
%     \lesssim&\int_{0}^{t}\int_{\RR}(t-s)^{-\frac 12} \bigg(\int_{\RR^2}G_{\frac{t-s}{2}}(y)\|u_\ep(s,x+h)-u_\ep(s,x)\|_{L^p(\Omega)}^p \\
%              &\qquad\qquad\qquad\qquad\qquad\times \lambda(x-y) dxdy\bigg)^{\frac 2p}|h|^{2H-2}dh ds \\
    \end{split}
\end{equation}
The  above bounds on $D^{\ep,n}_1(t) , D^{\ep,n}_2(t)$ together with \eqref{Est.u_n+1_Z1} yield
\begin{align}
  \|u^{n+1}_{\ep}(t,\cdot)\|_{L^p(\Omega\times\RR)}^2
 &\leq C_{p,H}\bigg( \|I_0(t, x)\|_{L^p(\Omega\times\RR)}^2+ \int_{0}^{t} (t-s)^{2H}\  \|u^n_{\ep}(s,\cdot)\|^2_{L^p(\Omega\times\RR)}  ds \nonumber\\
 &\qquad\qquad\qquad\qquad\quad +\int_{0}^{t}(t-s) \lt[\cN^*_{\frac 12-H,p}u^n_\ep(s)\rt]^2  ds\bigg)\,.
     \label{Est.u_Z1}
\end{align}
%   \begin{align*}
%     \|u_\ep^{n+1}(t,\cdot)&-u_\ep^n(t,\cdot)\|_{L^p_{\lambda}(\Omega\times\RR)}^p
%     = \int_{\RR}\EE\blk| u_\ep^{n+1}(t,x)-u_\ep^n(t,x)|^p\brk \lambda(x)dx \\
%     \tcb{\leq}& C_{\ep,p}\int_{\RR}\EE\lk\lt|\int_{0}^{t}\int_{\RR}G_{t-s}(x-y)|u_\ep^{n}(s,y)-u_\ep^{n-1}(s,y)|^2dy\cdot t\frac{ds}{t}\rt|^{\frac p2}\rk\lambda(x)dx \\
%     \tcb{\leq}& C_{\ep,p}t^{\frac p2-1} \int_{0}^{t}\int_{\RR}\Blc \int_{\RR} G_{t-s}(x-y)\lambda(x)dx\Brc \EE\blk| u_\ep^{n}(s,y)-u_\ep^{n-1}(s,y)|^p\brk dyds\\
%     \leq& C_{\ep,p,T}\int_{0}^{t}\int_{\RR}\Blc \int_{\RR} G_{t-s}(x-y)\lambda(x)dx\Brc \EE\blk| u_\ep^{n}(s,y)-u_\ep^{n-1}(s,y)|^p\brk dyds \\
%      \leq& C_{\ep,p,T}\int_{0}^{t}\int_{\RR} \EE\blk| u_\ep^{n}(s,y)-u_\ep^{n-1}(s,y)|^p\brk \lambda(y) dyds\\
%     =&C_{\ep,p,T} \int_{0}^{t}\|u_\ep^{n}(s,\cdot)-u_\ep^{n-1}(s,\cdot)\|_{L^p_{\lambda}(\Omega\times\RR)}^p ds\\
%     \le &C_{\ep,p,T}^n  \frac{T^n}{n!} \sup_{0\le s\le T}\|u_\ep^{1}(s,\cdot)-u_\ep^{0}(s,\cdot)\|_{L^p_{\lambda}(\Omega\times\RR)}^p  ,
%   \end{align*}

\noindent\textbf{Step 2.}\
Next, we  bound    $\cN^*_{\frac 12-H,p}u^{n+1}_\ep(t)$ by the $\cZ^p$  norm  of $u^{n}_\ep(t,x)$.     Similar to \eqref{uepLp} we have
   \begin{align*}
%     \EE\blk|\fD_h u^{n+1}_\ep(t,x)|^p\brk =&\EE\blk|u^{n+1}_\ep(t,x)-u^{n+1}_\ep(t,x+h)|^p\brk \\
     \EE\blk|\fD_h u^{n+1}_\ep(t,x)|^p\brk\leq& C_p \big| I_0(t,x)-I_0(t,x+h)\big|^p \\
       &+C_p\EE\bigg( \int_{0}^{t}\int_{\RR^2} \Big|\fD_h G_{t-s}(x-y-z)\sigma(u^{n}_\ep(s,y+z)) \\
       &\qquad\qquad\  -\fD_h G_{t-s}(x-z)\sigma(u^{n}_\ep(s,z))\Big|^2|y|^{2H-2} dzdyds\bigg)^{\frac p2} \\
     \leq& C_p\lk \cI_0(t,x,h)+\cI^{\ep,n}_1(t,x,h)+\cI^{\ep,n}_2(t,x,h)\rk\,,
   \end{align*}
   where
   \[
    \cI_0(t,x,h):=\big|I_0(t,x)-I_0(t,x+h)\big|^p\,,
   \]
   \begin{align*}
     \cI^{\ep,n}_1(t,x,h)&:= \EE \bigg(\int_{0}^{t}\int_{\RR^2}\Big| \fD_h G_{t-s}(x-y-z)\Big|^2 \big|\fD_y \sigma(u_\ep(s,z)) \big|^2 |y|^{2H-2}dzdyds\bigg)^{\frac p2}\,,
   \end{align*}
   and
   \begin{align*}
     \cI^{\ep,n}_2(t,x,h)&:= \EE \bigg(\int_{0}^{t}\int_{\RR^2} \Big|\Box_{y,h}G_{t-s}(x-z)\Big|^2 \cdot\big|\sigma(u_\ep(s,z))\big|^2|y|^{2H-2} dzdyds\bigg)^{\frac p2}\,.
   \end{align*}
Thus, by    Minkowski's  inequality we have
   \begin{align*}
     \lk\cN^*_{\frac 12-H,p}u^{n+1}_\ep(t)\rk^2&= \int_{\RR} \|\fD_h u^{n+1}_\ep(t,x) \|^2_{L^p(\RR\times\Omega)} |h|^{2H-2} dh \nonumber\\
     &\leq C_{p} \sum_{j=0}^{2}\int_{\RR}\lc \int_{\RR} \cI^{\ep,n}_j(t,x,h)dx \rc^{\frac 2p} |h|^{2H-2} dh \nonumber\\
     &=:J_0+J_1+J_2.
   \end{align*}
  For the   term $J_0$, it is clear that
\begin{equation}\label{I0Bdd}
%      &\int_{\RR}\lc \int_\RR \cI_1(t,x,h) \lambda(x) dx\rc^{\frac 2p} |h|^{2H-2} dh \\
%     =&\int_\RR \lc \int_\RR \lt|\int_\RR G_t(x-y)\blk u_0(y)-u_0(y+h)\brk dy\rt|^p \lambda(x)dx \rc^{\frac 2p} |h|^{2H-2} dh \\
   J_0  = C_p\int_\RR \lc\int_\RR \big|\fD_h I_0(t,x)\big|^p  dx\rc^{\frac 2p} |h|^{2H-2} dh= \lk\cN^*_{\frac 12-H,p}I_0(t)\rk^2\,.
%     \leq& C_p t^{\frac 2p} \int_\RR \lc \int_{\RR} \lt|\fD_h u_0(y)\rt|^p dy  \rc^{\frac 2p} |h|^{2H-2} dh=C_p \cdot t^{\frac 2p}\lk\cN^*_{\frac 12-H,p}u_0\rk^2\,.
   \end{equation}

 We can deal with the    term $J_1$ in the similar    manner as that when  we deal  with  \eqref{Est.Char1J1} in the  proof of Lemma \ref{Lem.Est_D}.
%   \footnote{what change of variables? Mentioned before \eqref{CharEst2I1.0} and \eqref{CharEst2I1}}
An application  of      Minkowski's inequality and then an application of  Parseval's formula yield   %\footnote{label}
    %\footnote{write each inequality for each operation?}
   \begin{equation}\label{I1Bdd}
     \begin{split}
%      &\int_{\RR}\lt| \int_\RR \cI^{\ep,n}_2(t,x,h) \lambda(x) dx\rt|^{\frac 2p} |h|^{2H-2} dh \\
    J_1   \leq& C_{p,H}\int_{0}^{t}  \int_{\RR^2} \Big| \fD_h G_{t-s}(z) \Big|^2 |h|^{2H-2} dhdz \\
      &\qquad\qquad \times \int_\RR\bigg(\int_\RR \EE\Blk\big|\fD_y u_{\ep}^n(t,x)\big|^p\Brk dx\bigg)^{\frac 2p} |y|^{2H-2} dyds \\
%      \lesssim& \int_{0}^{t}\int_\RR (t-s)^{H-1}\|u_\ep(s,\cdot)-u_\ep(s,\cdot+y)\|_{L^p_{\lambda}(\Omega\times\RR)}^{2} |y|^{2H-2} dyds \\
      \leq& C_{p,H}\int_{0}^{t} (t-s)^{2H}\lk\cN^*_{\frac 12-H,p}u^{n}_{\ep}(s)\rk^2 ds\,.
     \end{split}
   \end{equation}
Next,  we   bound  $J_2$. By the condition \eqref{LinearG}  ($|\sigma(u)|\lesssim |u|$)
and by a change  of variable $z\to x-z$, we obtain
%\footnote{detail,  Mentioned before \eqref{CharEst2I2.0}, \eqref{CharEst2I21} and \eqref{CharEst2I22}}
%and  then split it to two terms to obtain
   \begin{align*}
%     &\cI^{\ep,n}_2(t,x,h)\leq C_{p}\lc \cI^{\ep,n}_{21}(t,x,h)+\cI^{\ep,n}_{22}(t,x,h)\rc \\
    \cI^{\ep,n}_2(t,x,h) \leq C_{p}\EE\lc \int_{0}^{t}\int_{\RR^2} |\Box_{t-s}(z,y,h)|^2 |u^{n}_\ep(s,x-z)|^2 |y|^{2H-2}dydzds\rc^{\frac p2} \,.
%     +& C_{p}\EE\lc \int_{0}^{t}\int_{\RR^2} |\Box_{t-s}(y,z,h)|^2|\sigma(u^n_\ep(s,x+z))-\sigma(u^n_\ep(s,x-z))|^2 |y|^{2H-2}dydzds\rc^{\frac p2}  \,.
   \end{align*}
 In a similar way to that when we deal  with  \eqref{Est.Char1J2} in the  proof of Lemma \ref{Lem.Est_D},   we have
   \begin{equation}\label{I21Bdd}
     \begin{split}
   J_{2}:=& \int_{\RR}\lt| \int_\RR \cI^{\ep,n}_{2}(t,x,h) dx\rt|^{\frac 2p} |h|^{2H-2} dh \\
   \leq&\ C_{p,H} \int_{0}^{t}\int_{\RR^3} |\Box_{y,h}G_{t-s}(z)|^2 |y|^{2H-2}dy |h|^{2H-2} dhdz \\
     &\qquad\qquad\qquad\qquad\times \lc\int_{\RR}\EE |u^{n}_\ep(s,x-z)|^p dx\rc^{\frac 2p} ds\\
%     \lesssim& \int_{0}^{t} \int_{\RR^3}|\Box_{t-s}(z,y,h)|^2 |y|^{2H-2}|h|^{2H-2} dydzdhds \\
%      &+\int_{0}^{t} \int_{\RR^3}|\Box_{t-s}(z,y,h)|^2 \|u_\ep(s,\cdot)\|_{L^p_{\lambda}(\Omega\times\RR)}^{2}|y|^{2H-2}|h|^{2H-2} dydzdhds \\
     \leq&\ C_{p,H} \int_{0}^{t} (t-s)^{4H-1} \|u^n_\ep(s,\cdot)\|_{L^p(\Omega\times\RR)}^{2}  ds\,.
     \end{split}
   \end{equation}
%   Again by Minkowski's inequality,  the Lipschitz condition \eqref{LipSigam} on $\sigma$,
%%   \footnote{which condition?}
%and  Lemma 2.11 in \cite{HW2019} %\footnote{use label}
%we obtain
%   \begin{equation}\label{I22Bdd}
%     \begin{split}
%    J_{22}:=  &\int_{\RR}\lt| \int_\RR \cI^{\ep,n}_{22}(t,x,h) dx\rt|^{\frac 2p} |h|^{2H-2} dh \\
%    \leq & C_{p,H} \int_{0}^{t} \tcb{(t-s)^{2H}}\lk\cN^*_{\frac 12-H,p}u^n_{\ep}(s)\rk^2 ds\,.
%%     \lesssim& \int_{0}^{t}\int_\RR \lc\int_{\RR^2}|\Box_{t-s}(z,y,h)|^2 |y|^{2H-2}|h|^{2H-2} dydh\rc \\ &\qquad\qquad\qquad\qquad\times\|u_\ep(s,\cdot+z)-u_\ep(s,\cdot)\|_{L^p_{\lambda}(\Omega\times\RR)}^{2} dzds \\
%%     \lesssim& \int_{0}^{t}\int_\RR (t-s)^{H-1} \|u_\ep(s,\cdot+z)-u_\ep(s,\cdot)\|_{L^p_{\lambda}(\Omega\times\RR)}^{2} |z|^{2H-2} dzds \\
%%     \lesssim& \int_{0}^{t} (t-s)^{H-1}\lk\cN^*_{\frac 12-H,p}u(s)\rk^2 ds\,.
%     \end{split}
%   \end{equation}
Thus, we obtain
\begin{align}\label{Est.u_n+1_Z2}
     \lk\cN^*_{\frac 12-H,p}u^{n+1}_\ep(t)\rk^2  \leq&\ C_{p,H} \lk\cN^*_{\frac 12-H,p}I_0(t)\rk^2+C_{p,H}\int_{0}^{t} (t-s)^{2H}\lk\cN^*_{\frac 12-H,p}u^n_{\ep}(s)\rk^2 ds\nonumber \\
     +&\ C_{p,H}\int_{0}^{t} (t-s)^{4H-1} \|u^n_\ep(s,\cdot)\|_{L^p(\Omega\times\RR)}^{2} ds\,.
\end{align}

\noindent\textbf{Step 3.}\
Set
   \[
    \Psi_{\ep}^{n}(t):=\|u^{n}_\ep(t,\cdot)\|_{L^p(\Omega\times\RR)}^{2}+\lk\cN^*_{\frac 12-H,p}u^{n}_\ep(t)\rk^2.
   \]
   Then  combining the estimates \eqref{Est.u_Z1} and \eqref{Est.u_n+1_Z2} yields
   \begin{align*}
     \Psi_{\ep}^{n+1}(t)&\leq C_{p,H,T}\lc \|I_0\|_{\cZ^p(T)}^2 +\int_{0}^{t}  (t-s)^{4H-1}  \Psi_{\ep}^{n}(s) ds\rc\,.
   \end{align*}
 Now  it is relatively easy to see by fractional Gronwall lemma (similar to \cite[Lemma A.2]{CHN2016})
   \[
   \sup_{\varepsilon>0}
    \sup_{n\geq 1}\sup_{t\in[0,T]}\Psi_{\ep}^{n}(t)\leq C_{T,p,H}<\infty\,.
   \]
   Thus,  by the same argument as in the proof of \cite[Lemma 4.5]{HW2019}, we have  that $\sup\limits_{\ep>0}\sup\limits_{t\in[0,T]} \|u_{\ep}(t,\cdot)\|_{L^p(\Omega\times\RR)}$  and $\sup\limits_{\ep>0}\sup\limits_{t\in[0,T]} \cN^*_{\frac 12-H,p}u_{\ep}(t)$ are finite.
%   \begin{align}
%   \sup_{t\in[0,T]} \cN^*_{\frac 12-H,p}u_{\ep}(t)=&\sup_{t\in[0,T]}\lc\int_{\RR} \|\fD_hu_{\ep}(t,\cdot)\|^2_{L^p(\Omega\times\RR)} |h|^{2H-2}dh\rc^{\frac 12}\nonumber \\
%   	\leq& C_H\sup_{n\geq 1}\sup_{t\in[0,T]}\Psi_{\ep}^{n}(t)<\infty\,.
%   	\label{e.4.53a}
%   \end{align}
%   Therefore,    $\sup\limits_{\ep>0}\sup\limits_{t\in[0,T]} \cN^*_{\frac 12-H,p}u_{\ep}(t)$ is finite.
 %     \footnote{don't know what you want to do here}

   In conclusion, {   we have proved $\sup_{\ep>0}\|u_\ep\|_{\cZ^p(T)}:=\sup\limits_{\ep>0}\sup\limits_{t\in[0,T]} \|u_{\ep}(t,\cdot)\|_{L^p(\Omega\times\RR)}+\sup\limits_{\ep>0}\sup\limits_{t\in[0,T]}\cN^*_{\frac 12-H,p}u_{\ep}(t)$ is finite.   }
\end{proof}

\section{H\"older continuity and well-posedness}\label{s.4}
In this section, we obtain some estimations which imply  the H\"older regularity of the stochastic convolution with respect to our noise $\dot{W}$. Then the similar estimations   of the solution to SWE \eqref{eq.SWE} follow  in a routine way. These estimations   are devoted to prove the tightness   associated with the solution to \eqref{eq.SWE}.
%Next, we show that $H>\frac 14$ is a necessary condition for the solvability of Hyperbolic Anderson equation \eqref{eq.HAE}.

\subsection{H\"older continuity of stochastic convolution }
We have the following regularity results for stochastic convolution $\Phi(t,x)$ defined by  \eqref{eq.StoConv} and the approximated solution $u_\ep$  defined  by  \eqref{MildSolRegu}.
\begin{prop}\label{p.4.1}  Let $v(\cdot, \cdot)\in \cZ^p(T)$   and let  the stochastic convolution $\Phi(t,x)$ be defined by \eqref{eq.StoConv}.
We have the following   H\"older regularity in the space and time variables for $\Phi(t,x)$:
\begin{enumerate}[leftmargin=*]
  \item[\textbf{(i)}] If  $p>\frac{1}{H}$ and ~$0<\gamma<H-\frac 1p$, %(and $1-H<\alpha<1-\frac 1p$)
   then
  \begin{equation}\label{Est.t}
  \big\|\sup\limits_{t,t+h\in[0,T], x\in \RR}\
  | \Phi(t+h,x)-\Phi(t,x)|\big\|_{L^p(\Omega)}\leq C_{T,p,H,\gamma}|h|^{\gamma}\|v\|_{\cZ^p(T)}\,.
  \end{equation}
  \item[\textbf{(ii)}] If  $p>\frac{1}{H}$ and ~$0<\gamma<H-\frac 1p$,
  % (and $1-H<\alpha<1-\frac 1p$)},
  then
  \begin{equation}\label{Est.x}
  \big\|\sup\limits_{t\in[0,T], x,y\in \RR}\ |\Phi(t,x)-\Phi(t,y)|\big\|_{L^p(\Omega)}\leq C_{T,p,H,\gamma}|x-y|^{\gamma}\|v\|_{\cZ^p(T)}\,.
  \end{equation}
  \end{enumerate}
\end{prop}

\begin{proof}
\textbf{Step 1:} In this step, we concentrate on the analysis of the following quantity (we denote $\sup_{t,t+h\in[0,T], x\in \RR} $ by $\sup_{t,x}$)
$$\sup\limits_{t,x}\ |\Delta_h \Phi(t,x)|:= \sup\limits_{t,x}\ \lt| \Phi(t+h,x)-\Phi(t,x)\rt|\,.$$
%which is related with the H\"older index of $\Phi(t,x)$ in the time variable.
Assuming $h\in(0,1)$ and $t\in[0,T]$ such that $t+h\leq T$, then by  the  representation formula
 \eqref{eq.Fubini}  and the triangle  inequality we have
\begin{equation*}
       \begin{split}
          \Delta_h \Phi(t,x)
          =\,&\sum_i\frac{\sin(\pi\theta)}{\pi} \bigg[\int_{0}^{t+h}\int_{\RR}  (t+h-r)^{\theta-1}\overline{\cK}_i(t+h-r,x-z) J_{\theta}^{\cK_i}(r,z)drdz\\
          &\qquad\qquad\qquad\qquad\quad-\int_{0}^{t}\int_{\RR}  (t-r)^{\theta-1}\overline{\cK}_i(t-r,x-z) J_{\theta}^{\cK_i}(r,z)drdz\bigg]\\
          \lesssim\,&\sum_{j=1}^{3} \cI_{j}(t,h,x)\,,
       \end{split}
\end{equation*}
where % we have used the following notations:
    \begin{align*}
      \cI_1(t,h,x):=&\sum_i\cI_1^{(i)}(t,h,x)\\
      :=&\sum_i\int_{0}^{t}\int_{\RR}\Delta_h (t-r)^{\theta-1} \overline{\cK}_i(t-r,x-z) J_{\theta}^{\cK_i}(r,z)drdz
    \end{align*}
    with $\Delta_h (t-r)^{\theta-1}:=  (t+h-r)^{\theta-1}-(t-r)^{\theta-1}$;
    \begin{align*}
      \cI_2(t,h,x)&:=\sum_i\cI_2^{(i)}(t,h,x)\\
      &:=\sum_i\int_{0}^{t}\int_{\RR}(t+h-r)^{\theta-1} \Delta_h \overline{\cK}_i(t-r,x-z) J_{\theta}^{\cK_i}(r,z)drdz,
   \end{align*}
   with $\Delta_h \overline{\cK}_i(t-r,x-z):=  \overline{\cK}_i(t+h-r,x-z)-\overline{\cK}_i(t-r,x-z)$; and
     \begin{eqnarray*}
      \cI_3(t,h,x)&:=&\sum_i\cI_3^{(i)}(t,h,x)\\
      &:=&\sum_i\int_{t}^{t+h}\int_{\RR}(t+h-r)^{\theta-1}  \overline{\cK}_i(t+h-r,x-z) J_{\theta}^{\cK_i}(r,z)drdz.
    \end{eqnarray*}
Our goal is to show that
\begin{equation}\label{Ineq_main_I_j}
\big\|\sup\limits_{t,x}\ \cI_j(t,h,x) \big\|_{L^p(\Omega)} \leq C_{T,p,H,\gamma}|h|^{\gamma}\| v\|_{\cZ^p(T)}\,,\qquad j=1,2,3\,,
\end{equation}
under the conditions
\begin{equation}\label{Condi_iii.main}
	p>\frac 1H\,,\quad 1-H<\alpha<1-\frac 1p\,,\quad \gamma<H-\frac 1p\,.
\end{equation}
We shall first treat   $\cI_1(t,h,x)$ and $\cI_3(t,h,x)$. The term $\cI_2(t,h,x)$ is more complicated and  shall  be  handled lastly.

For the term $\cI_1(t,h,x)$, it is easy to see that for any fixed $\gamma\in(0,1)$,
\begin{equation}\label{h_regular}
\Delta_h (t-r)^{\theta-1}=|(t+h-r)^{\theta-1}-(t-r)^{\theta-1}|\lesssim|t-r|^{\theta-1-\gamma}h^{\gamma}.
\end{equation}
Then by H\"{o}lder's  inequality with $1/p+1/q=1$ and Lemma \ref{Lem.Est_J}, under conditions \eqref{condi_J_est} we have for $i=1,\cdots,4$
 \begin{align}
 \big\|\sup\limits_{t,x}\ &\cI_1^{(i)}(t,h,x) \big\|_{L^p(\Omega)} \nonumber\\
 &\leq \left(\sup\limits_{t,x }\ \int_{0}^{t}\int_{\RR}\lt|\Delta_h (t-r)^{\theta-1}\rt|^q |\overline{\cK}_i(t-r,x-z)|^q dzdr\right)^{1/q} \times \|v\|_{\cZ^p(T)}\nonumber\\
% & \times\left(\int_{0}^{T}\parallel J_{\theta}^{\cK_i}(r,\cdot)\parallel ^p_{L^p(\RR)}dr\right)^{1/p} \nonumber\\
 &\leq \left(\sup\limits_{t}\ \int_{0}^{t}\int_{\RR}|r|^{(\theta-1-\gamma)q} |\overline{\cK}_i(r,z)|^q dzdr\right)^{1/q} \times \|v\|_{\cZ^p(T)}\cdot |h|^{\gamma} \,,\label{J_1}
% &\times \left(\int_{0}^{T}\parallel J_{\theta}^{\cK_i}(r,\cdot)\parallel ^p_{L^p(\RR)}dr\right)^{1/p}\,, \label{J_1}
\end{align}
where in  the last inequality of \eqref{J_1} we have used the change of variables $r\to t-r$ and $z\to z+x$. Now we only need to show
$$\sup\limits_{t}\ \int_{0}^{t}\int_{\RR}|r|^{(\theta-1-\gamma)q} |\overline{\cK}_i(r,z)|^q dzdr<+\infty.$$
We shall only discuss the situation $i=1$. Other cases $i=2,3,4$ can be treated similarly. For $i=1$, we have $\cK_1=\cC_{\alpha},\ \overline{\cK}_1=\cS_{1-\alpha}$ as defined in \eqref{eq.cC_alpha}. Hence, by changing variable $r\to rz$ we have
\begin{align*}
	\sup\limits_{t}&\  \int_{0}^{t}\int_{\RR}|r|^{(\theta-1-\gamma)q} |\cS_{1-\alpha}(r,z)|^q dzdr \\
	\leq& \sup\limits_{t}\ \int_{0}^{t}|r|^{(\theta-1-\gamma)q+1-\alpha q}dr  \cdot \int_0^{\infty} \Big|(1+|z|)^{-\alpha} +\textrm{sgn}(1-|z|) |1-|z||^{-\alpha}\Big|^q dz\,.
\end{align*}
%\begin{align}
%\cI_1^{(1)}(t,h,x)&\leq\left(\int_{0}^{t}\int_{\RR}|(t-r)|^{(\theta-1-\gamma)q}|h|^{\gamma q} |\cS_{1-\alpha}(t-r,x-z)|^q dzdr\right)^{1/q} \nonumber\\
% &\qquad\ \times \left(\int_{0}^{T}\parallel J_{\theta}^{\cC_{\alpha}}(r,\cdot)\parallel ^p_{L^p(\RR)}dr\right)^{1/p}\nonumber\\
% &\leq|h|^{\gamma}\bigg(\sup_t\int_{0}^{t}|(t-r)|^{(\theta-1-\gamma)q+1-\alpha q}dr\cdot \int_0^{\infty} \Big|(1+|z|)^{-\alpha}  \nonumber\\
% &\qquad\ +\textrm{sgn}(1-|z|) |1-|z||^{-\alpha}\Big|^q dz\bigg)^{1/q} \cdot \left(\int_{0}^{T}\parallel J_{\theta}^{\cC_{\alpha}}(r,\cdot)\parallel ^p_{L^p(\RR)}dr\right)^{1/p}.\nonumber
%\end{align}
Then by the same argument as in the proof of part $\textbf{(i)}$ of  Proposition \ref{prop_est}, we have
$$\big\|\sup\limits_{t,x}\ \cI_1(t,h,x) \big\|_{L^p(\Omega)} \leq C_{T,p,H,\gamma}|h|^{\gamma}\| v\|_{\cZ^p(T)}$$
under the conditions \eqref{condi_J_est} and $(\theta-1-\gamma)q+1-\alpha q>-1$, which can be summarized as the following conditions
\begin{align}
	p>\frac 1H\,,\quad 1-H<\alpha<1-\frac 1p\,,\quad 1+\alpha-\frac{2}{q}+\gamma<\theta<H+\alpha-\frac 12\,.\label{condi_J_I1}
\end{align}
Since  $p>\frac 1H>2$ implies $\gamma<H-1/p<H+2/q-3/2$  it is clear that one can choose $\alpha$ and $\theta$ satisfying \eqref{condi_J_I1} under conditions \eqref{Condi_iii.main}.
%\begin{align}
% 1-H<\alpha<1-\frac 1p\,,&\quad\theta<H+\alpha-\frac 12\,,\label{condi_J_I1}\\
% (\theta-1-\gamma)q+1-\alpha q>-1 \, &\Leftrightarrow\, \theta-\gamma>1+\alpha-\frac{2}{q}\,.\label{Ineq_r_I1_1}
%\end{align}
%Now let us deal with the second term, i.e. $\cK_{2}=\cS_{\alpha}$ and $\bar{\cK}_{2}=\cC_{1-\alpha}$.
%
%\begin{align}
%\cI_1^{(2)}(t,h,x)&\leq\left(\int_{0}^{t}\int_{\RR}|(t-r)|^{(\theta-1-\gamma)q}|h|^{\gamma q} |\cC_{1-\alpha}(t-r,x-z)|^q dzdr\right)^{1/q} \nonumber\\
% &\qquad\quad\times \left(\int_{0}^{T}\parallel J_{\theta}^{\cS_{\alpha}}(r,\cdot)\parallel ^p_{L^p(\RR)}dr\right)^{1/p}\nonumber\\
% &\leq|h|^{\gamma}\bigg(\sup_t\int_{0}^{t}|(t-r)|^{(\theta-1-\gamma)q+1-\alpha q}dr\cdot \int_{0}^{\infty}\bigg[\cos\lc\frac{\alpha\pi}{2}\rc \lk\big|1+|z|\big|^{-\alpha}+\big|1-|z|\big|^{-\alpha}\rk  \nonumber\\
% &\qquad\quad-2\cos\lc\alpha\tan^{-1}(z)\rc \blk 1+z^2\brk^{-\frac{\alpha}{2}}\bigg]^q dz\bigg)^{1/q}\nonumber\\
% &\qquad\quad \times \left(\int_{0}^{T}\parallel J_{\theta}^{\cC_{\alpha}}(r,\cdot)\parallel ^p_{L^p(\RR)}dr\right)^{1/p}.\nonumber
%\end{align}
%By Appendix Lemma \ref{Lem_A.1}, $\cC_{1-\alpha}(t,x)$ can be bounded by $$\big|t+|x|\big|^{-\alpha}+\big|t-|x|\big|^{-\alpha}+\blk t^2+x^2\brk^{-\frac{\alpha}{2}}$$
%when $|x|$ is close to $t$, and be bounded by $t\blc |x|^2-t^2\brc^{-\frac{\alpha}{2}-1}$ when $|x|$ is large enough.
%
%Then by the same proof of $\textbf{(i)},$ $\cI_1^{(2)}(t,h,x)\leq C_{T,p,\alpha,\gamma}|h|^{\gamma}\| v\|_{\cZ^p(T)}$ under \eqref{condi_J_I1} and \eqref{Ineq_r_I1_1}.
%It is similar when $i=3,4.$

Now let us deal with the term $\cI_3$.
Using H\"{o}lder's inequality, Lemma \ref{Lem.Est_J} and the change  of variables $z\to z+x$ and $r\to r-t-h$, we have
\begin{align}
	\big\|\sup\limits_{t,x}\ &\cI_3(t,h,x) \big\|_{L^p(\Omega)} \nonumber\\
	\ls\ & \sum_i  \left( \int_{0}^{h}\int_{\RR}r^{q(\theta-1)} | \overline{\cK}_i(r,z)|^q dzdr\right)^{1/q} \times \|v\|_{\cZ^p(T)}\nonumber\\
	=:&\sum_i \left(\cI^{(i)}_3(h) \right)^{1/q}\times \|v\|_{\cZ^p(T)}\,. \label{J_3}
\end{align}
We want to show that $\left( \cI^{(i)}_3(h) \right)^{1/q}\lesssim h^\gamma$ for $i=1, \cdots, 4$ with $p,\ \alpha,\ \gamma$ satisfying  \eqref{Condi_iii.main}.
As before, we only consider the case  $i=1$, i.e. $\cK_1=\cC_{\alpha},\ \overline{\cK}_1=\cS_{1-\alpha}$. The other cases can be handled in similar way.  In this case we have
\begin{align}
	\left(\sup\limits_{t,x}\cI^{(1)}_3(h)\right)^{1/q} =&  \left(\int_{0}^{h}\int_{\RR}r^{q(\theta-1)} | \cS_{1-\alpha}(r,z)|^q dzdr\right)^{1/q} \nonumber\\
	\leq& \left(\int_{0}^{h} r^{q(\theta-1)+1-q\alpha}dr  \right.\nonumber\\
	&\qquad\ \times\left.\int_0^{+\infty}\blk (1+|z|)^{-\alpha}+\textrm{sgn}(1-|z|)|1-|z||^{-\alpha}\brk^q dz\right)^{1/q}\nonumber\\
	\lesssim&\left[|h|^{q(\theta-1)+2  -q\alpha}\right]^{1/q} \le  h^\gamma   \label{Est.I_3(1)}
\end{align}
if \eqref{condi_J_I1} is satisfied (and hence so does \eqref{Condi_iii.main}).
%, then for any   small $h$ $\in(0,1)$
%\[
%\left(\sup\limits_{t,x} \cI^{(1)}_3(h) \right)^{1/q}\lesssim h^\gamma\,.
%\]
We have similar bound for $ \cI^{(i)}_3(h)$   for $i=2,3,4$. Combing these bounds with \eqref{J_3} we have
$$
 \big\|\sup\limits_{t,x}\ \cI_3(t,h,x) \big\|_{L^p(\Omega)} \leq C_{T,p,H,\gamma} |h|^{\gamma} \| v\|_{\cZ^p(T)}\,,
% \leq C_{T,p,\theta,\alpha}|h|^{\theta-1+2/q-\alpha}\| v\|_{\cZ^p(T)} \lesssim |h|^{\gamma} \| v\|_{\cZ^p(T)}\,.
$$
if $p,\ \alpha,\ \gamma$ satisfy \eqref{Condi_iii.main}.
%under \eqref{Ineq.alpha_I1}, \eqref{Ineq.theta_I1} \eqref{Ineq.alpha_I2}, \eqref{Ineq.alpha_I3}.
%It is similar when $i=2,3,4.$

Lastly, we are going to deal  with $\cI_2$, which is much more complicated. By H\"{o}lder's  inequality,
\begin{eqnarray}
\cI_2(t,h,x)&\leq&\sum_i\left(\int_{0}^{t}\int_{\RR}(t+h-r)^{q(\theta-1)} |\Delta_h \overline{\cK}_i(t-r,x-z)|^q dzdr\right)^{1/q} \nonumber\\
&&\qquad \times \left(\int_{0}^{T}\parallel J_{\theta}^{\cK_i}(r,z)\parallel ^p_{L^p(\RR)}dr\right)^{1/p}\,. \label{J_2}
\end{eqnarray}
The second factor inside the summation  in \eqref{J_2} can be bounded by a multiple of $\|v\|_{\cZ^p(T)}$   via Lemma \ref{Lem.Est_J} under the condition  \eqref{condi_J_est}.
% Thus, we can see from \eqref{J_2} that the regularity of $\cI_2$ only depends on  $\Delta_h \overline{\cK}_i(t,x)$.
 By the change of variables $r\rightarrow t-r,\ z\rightarrow x-z$,   we see
\begin{align*}
\big\|\sup\limits_{t, x}\ & \cI_2(t,h,x) \big\|_{L^p(\Omega)}\\
&\leq \sum_i\left(\sup\limits_{t}\ \int_{0}^{t}\int_{\RR}(r+h)^{q(\theta-1)} | \Delta_h \overline{\cK}_i(r,z)|^q dzdr\right)^{1/q} \times \|v\|_{\cZ^p(T)} \nonumber\\
%&\qquad\qquad\ \times \left(\int_{0}^{T}\parallel J_{\theta}^{\cK_i}(r,z)\parallel ^p_{L^p(\RR)}dr\right)^{1/p}\\
&=:\sum_i\left(\sup\limits_{t}\cI^{(i)}_2(t,h)\right)^{1/q} \times \|v\|_{\cZ^p(T)} \,.
\end{align*}
Thus, we shall need to show  that for $i=1,2,3,4$
\begin{equation}\label{Ineq_I_2^i}
\sup\limits_{t}\cI_2^{(i)}(t, h)=\sup\limits_{t}\ \int_{0}^{t}\int_{\RR}(r+h)^{q(\theta-1)} | \Delta_h \overline{\cK}_i(r,z)|^q dzdr \leq C_{T,p,H,\gamma}|h|^{\gamma q},
\end{equation}
to obtain
\begin{equation}\label{Ineq_I_2}
 \big\|\sup\limits_{t,x}\ \cI_2(t,h,x) \big\|_{L^p(\Omega)} \leq C_{T,p,H,\gamma} |h|^{\gamma} \| v\|_{\cZ^p(T)}\,.
% \leq C_{T,p,\theta,\alpha}|h|^{\theta-1+2/q-\alpha}\| v\|_{\cZ^p(T)} \lesssim |h|^{\gamma} \| v\|_{\cZ^p(T)}\,.
\end{equation}
Now,  we shall deal with $\cI_2^{(i)}( t,h)$ for $i=1,2,3,4$ term by term.

\medskip
\noindent $\textbf{Case i=1}$.  Recall   that $\overline{\cK}_1(r,z)=\cS_{1-\alpha}(r,z)$ and $\cK_1(r,z)=\cC_{\alpha}(r,z)$   are defined by \eqref{eq.cC_alpha}.  We shall  show
\begin{equation}
\left\{\begin{split}
\sup\limits_{t}\cI_2^{(1)}(t, h)& \leq C_{T,p,\theta,\alpha} |h|^{\gamma q}\,,\quad \hbox{where}   \\
\cI_2^{(1)}(t, h)=& \int_{0}^{t}\int_{\RR}(r+h)^{q(\theta-1)} |  \Delta_h\cS_{1-\alpha}(r,z)|^q dzdr
\end{split}\right. \label{Est.I_2^(1)}
\end{equation}
for $p,~\gamma$ and $\alpha$ satisfying \eqref{Condi_iii.main}.
Set $A_1:=[|z|<r]$, $A_2:=[|z|>r+2h]$ and $A_3:=[r<|z|<r+2h]$. For fixed $\eta\in(0,1)$,   we see
 \begin{equation}\label{Ineq_important}
 	\begin{cases}
 		\Delta_h |r+|z||^{-\alpha}=\lt||r+|z|+h|^{-\alpha}-|r+|z||^{-\alpha}\rt|\lesssim |r+|z||^{-\alpha-\eta}|h|^{\eta}\,,~&\text{on } \RR\,; \\
 		\Delta_h |r-|z||^{-\alpha}=\lt||r-|z|+h|^{-\alpha}-|r-|z||^{-\alpha}\rt|\lesssim |r-|z||^{-\alpha-\eta}|h|^{\eta}\,,~&\text{on } A_1\,;\\
 		\Delta_h |r-|z||^{-\alpha}=\lt||r-|z|+h|^{-\alpha}-|r-|z||^{-\alpha}\rt|\lesssim ||z|-r-h|^{-\alpha-\eta}|h|^{\eta}\,,~&\text{on } A_2\,.
 	\end{cases}
 \end{equation}
%\[
% \tcr{\Delta_h |r\pm|z||^{-\alpha}=\lt|r\pm|z|+h|^{-\alpha}-|r\pm|z||^{-\alpha}\rt|\lesssim |r\pm|z||^{-\alpha-\eta}|h|^{\eta}\,,}
%\]
Then we have
%$\footnote{We only take into account the first variation. It is in fact natural to consider second variation.}$
\begin{eqnarray*}
|\Delta_h\cS_{1-\alpha}(r,z)|^q
&\leq&\left|\Delta_h \lt|r+|z|\rt|^{-\alpha}\right|^{q} +\Big|\Delta_h |r-|z||^{-\alpha}\cdot\left[\1_{A_1}+\1_{A_2}\right]\Big|^q\\
%&&+\Big|\Delta_h |r-|z||^{-\alpha}\cdot\left[1_{(|z|<r)}+1_{(|z|>r+h)}\right]\Big|^q\\
&&+\Big||r+h-|z||^{-\alpha}+|r-|z||^{-\alpha}\Big|^q \1_{A_3}\\
&\leq&|r+|z||^{(-\alpha-\eta_1)q}h^{\eta_1 q}+|r-|z||^{(-\alpha-\eta_2)q}h^{\eta_2 q}\cdot \1_{A_1}\\
&&+ ||z|-r-h|^{(-\alpha-\eta_3)q} h^{\eta_3 q}\cdot \1_{A_2} \\
&&+\lt||r+h-|z||^{-\alpha}+|r-|z||^{-\alpha}\rt|^q \1_{A_3}\,,
\end{eqnarray*}
for some $\eta_1,\eta_2,\eta_3\in(0,1)$. Therefore,
 \begin{align}\label{Est.I2(1)}
%	\int_0^t\int_{\mathbb{R}} &(r+h)^{q(\theta-1)}\lt|\Delta_h\cS_{1-\alpha}(r,z)\rt|^qdzdr  \nonumber\\
\cI_2^{(1)}(t, h)	\ls\,&\int_0^t\int_{\mathbb{R}} (r+h)^{q(\theta-1)} |r+|z||^{(-\alpha-\eta_1)q}h^{\eta_1 q} dzdr \nonumber\\
	&+\int_0^t\int_{\mathbb{R}} (r+h)^{q(\theta-1)} |r-|z||^{(-\alpha-\eta_2)q}h^{\eta_2 q}\cdot \1_{A_1}dzdr \nonumber\\
	&+\int_0^t\int_{\mathbb{R}} (r+h)^{q(\theta-1)}||z|-r-h|^{(-\alpha-\eta_3)q}h^{\eta_3 q}\cdot \1_{A_2} dzdr \nonumber\\
	&+\int_0^t\int_{\mathbb{R}} (r+h)^{q(\theta-1)} \lt||r+h-|z||^{-\alpha}+|r-|z||^{-\alpha}\rt|^q \cdot \1_{A_3}dzdr\nonumber\\
=:& \sum_{k=1}^4\cI_{2,k}^{(1)}(t,h)\,.
\end{align}
%Now we are going to handle $\cI_{2,k}^{(1)}(t, h)$, $k=1, 2, 3, 4$  separately. By Lemma \ref{I_{2,k}^{1}}, $\cI_{2,k}^{(1)}(t, h)\ (k=1,2,3,4)$ can be bounded by $h^{\gamma q}$ if we require conditions \eqref{condi_eta_1}, \eqref{condi_eta_2}, \eqref{condi_eta_3} and \eqref{condi_eta_4}.
The procedures of dealing terms $\cI_{2,k}^{(1)}(t, h)$, $k=1, 2, 3, 4$ require standard but careful computations which are  included in Appendix \ref{Appen.C}. By Lemma \ref{I_{2,k}^{1}}, for any $p>\frac1H$, $\gamma<H-\frac 1p$, $\cI_{2,k}^{(1)}(t, h)\ (k=1,2,3,4)$ can be bounded by $h^{\gamma q}$ if $\alpha$, $\theta$ satisfy \eqref{list1} and $\eta_k$, $k=1,2,3$ satisfy \eqref{Cond_AppeC1}.

\medskip
\noindent{$\textbf{Case i=2}$}.  \
In this case, we have $\overline{\cK}_2(r,z)=\cC_{1-\alpha}(r,z)$ defined  by  \eqref{eq.cC_alpha}. We want to show when $i=2$, i.e.
%\begin{eqnarray}
% \sup\limits_{t} \cI_2^{(2)}( t, h)=\sup\limits_{t}\ \int_{0}^{t}\int_{\RR}(r+h)^{q(\theta-1)} |  \Delta_h\cC_{1-\alpha}(r,z)|^q dzdr \leq C_{T,p,\theta,\alpha}|h|^{\gamma q}\, ,\label{Est.I_2^(2)}
%%&& \times \left(\int_{0}^{T}\parallel J_{\theta}^{\cC_{\alpha}}(r,z)\parallel ^p_{L^p(\RR)}dr\right)^{1/p}\,.
%\end{eqnarray}
\begin{equation}
\left\{\begin{split}
\sup\limits_{t}\cI_2^{(2)}(t, h)& \leq C_{T,p,\theta,\alpha} |h|^{\gamma q}\,,\quad \hbox{where}   \\
\cI_2^{(2)}(t, h)=& \int_{0}^{t}\int_{\RR}(r+h)^{q(\theta-1)} |  \Delta_h\cC_{1-\alpha}(r,z)|^q dzdr
\end{split}\right. \label{Est.I_2^(2)}
\end{equation}
with parameters $p$, $\gamma$ and $\alpha$ satisfying \eqref{Condi_iii.main}.

For fixed $\eta\in(0,1)$, it is not hard to verify
\begin{align}\label{r2_regular}
\lt|\Delta_h\lt(r^2+|z|^2\rt)^{-\frac{\alpha}{2}}\rt|\ls\, &\lt(r^2+|z|^2\rt)^{-\frac{\alpha}{2}(1-\eta)} \lt|\frac{(r+\xi h)\cdot h}{[(r+\xi h)^2+|z|^2]^{\al/2+1}} \rt|^{\eta} \nonumber \\
\ls\, &\lt(r^2+|z|^2\rt)^{-\frac{\alpha}{2}(1-\eta)} \lt|r^2+|z|^2\rt|^{-(\frac{\alpha}{2}+1)\eta}(r+h)^{\eta}|h|^{\eta} \nonumber \\
\ls\, &\lt|r^2+|z|^2\rt|^{-\frac{\alpha}{2}-\eta}\cdot (r+h)^{\eta}|h|^{\eta}\,,
\end{align}
and
\begin{equation}\label{Ineq_imp_2}
\lt|\Delta_h\cos\lc\alpha\tan^{-1}\lc\frac{|z|}{r}\rc\rc\rt|\ls\frac{|z|^{\eta}|h|^{\eta}}{(r^2+z^2)^{\eta}}.
\end{equation}
Then by the above two inequalities    \eqref{r2_regular} and \eqref{Ineq_imp_2},  and the inequalities in \eqref{Ineq_important}, we have
\begin{align}\label{Ineq_C_{1-alpha}}
\Big|\Delta_h\cC_{1-\alpha}(r,z)\Big|^q
\lesssim&|r+|z||^{(-\alpha-\eta_1)q}h^{\eta_1 q}+|r-|z||^{(-\alpha-\eta_2)q}h^{\eta_2 q}\cdot \1_{A_1}\nonumber\\
&+ ||z|-r-h|^{(-\alpha-\eta_3)q} h^{\eta_3 q}\cdot \1_{A_2} \nonumber\\
&+\lt||r+h-|z||^{-\alpha}+|r-|z||^{-\alpha}\rt|^q \1_{A_3}\nonumber\\
&+\lt|r^2+|z|^2\rt|^{(-\frac{\alpha}{2}-\eta_4)q}\cdot (r+h)^{\eta_4q}|h|^{\eta_4q}+\frac{|z|^{\eta_5q}|h|^{\eta_5q}}{(r^2+z^2)^{\eta_5q}}\nonumber\\
=:&\sum_{k=1}^6M_k^{(2)}(r,z)\,.
\end{align}
Substituting this bound into \eqref{Est.I_2^(2)},  we see that,
\begin{eqnarray*}
\sup\limits_{t}\cI_2^{(2)}( t, h)\le \sup\limits_{t}\ \sum_{k=1}^6\cI_{2,k}^{(2)}(t, h)\,,
\end{eqnarray*}
where
\begin{equation}\label{Ineq_I_2,k^2}
\cI_{2,k}^{(2)}(t, h)=\int_{0}^{t}\int_{\RR}(r+h)^{q(\theta-1)}M_k^{(2)}(r,z) dzdr, \ k=1,\cdots,6.
\end{equation}
The first four terms $\cI_{2,k}^{(2)}(t, h),\ k=1,\cdots,4$ are treated in the same  way   as \textbf{Case i=1} and require conditions \eqref{list1} and \eqref{Cond_AppeC1} to guarantee
$$\sup\limits_{t}\cI_{2,k}^{(2)}(t, h)\ls |h|^{\gamma q},\ \ k=1,\cdots,4.$$
We shall deal with the $\cI_{2,5}^{(2)}(t, h)$ and $\cI_{2,6}^{(2)}(t, h)$ in Appendix \ref{Appen.C}. By Lemma \ref{Lem.I_{2,5+6}^2}, $\sup\limits_{t}\cI_{2,k}^{(2)}(t, h)\ls |h|^{\gamma q}$ for $ k=5,6$ under conditions \eqref{list1} and \eqref{Cond_AppeC2}.

As a result, for any $p>\frac 1H,$ $\gamma<H-\frac 1p$, we have $\sup\limits_{t}\cI_{2,k}^{(2)}(t, h)\ls |h|^{\gamma q}$ for $ k=1,\cdots,6$, if $\alpha$, $\theta$ satisfy \eqref{list1} and $\eta_k$ $(k=1,\cdots,6)$ satisfy \eqref{Cond_AppeC1} and \eqref{Cond_AppeC2}.

% we can select associated $\alpha$, $\theta$ and $\eta_k$, $k=1,2,3,4$ satisfy the restrictions for any $p>\frac1H$, $\gamma<H-\frac 1p$ .

%\tcb{For instance, we may select $p$ large, $\alpha=1-H+\ep_1$, $\gamma=H-1/p-\ep_2$ and $\theta=H-\ep_3$ to be fixed. And we can choose $\eta_1=\eta_3=\eta_4=H-1/p$ and $\eta_2=H-1/p-\ep_2/2$. So we can let $\ep_1<\ep_2/2$, $\ep_3<\ep_2$ and $\ep_2/2+\ep_3<H-1/p$ ($\ep_1,\ep_2,\ep_3>0$ be small enough).}
%\begin{align*}
%	\tcr{p>\frac{1}{H}\,,~1-H<\alpha<1-\frac 1p\,,~0<\gamma<H-\frac 1p}\,,
%\end{align*}
%which is the same as $i=1.$

\medskip
\noindent $\textbf{Case  i=3}$.
In this case we have  $\overline{\cK}_3(r,z)=\cE(r,z)=\frac{1}{\pi}\frac{r}{r^2+z^2}$  and
\begin{align}\label{Est.III3}
\lt|\Delta_h\cE(r,z)\rt|^q&\simeq \lt|\frac{r+h}{(r+h)^2+z^2}-\frac{r}{r^2+z^2}\rt|^q\nonumber\\
&=\lt|\frac{h}{(r+h)^2+z^2}+r\lt[\frac{1}{(r+h)^2+z^2}-\frac{1}{r^2+z^2}\rt]\rt|^q\nonumber\\
&\leq\lt|\frac{h}{(r+h)^2+z^2}\rt|^q+\frac{r^q\cdot|h|^q\cdot|2r+h|^q}{|(r+h)^2+z^2|^{q}\cdot|r^2+z^2|^{q}}.
\end{align}
By H\"{o}lder's  inequality with $\frac{1}{m}+\frac{1}{n}=1$ and  $|2r+h|^q\leq 2^q|r+h|^q$,  we obtain
\begin{align}
\cI_2^{(3)}&(t,h)=\int_0^t\int_{\mathbb{R}}(r+h)^{q(\theta-1)}\lt|\Delta_h\cE(r,z)\rt|^qdzdr\nonumber\\
\lesssim&\int_0^t\int_{\mathbb{R}}(r+h)^{q(\theta-1)}\lt|\frac{h}{(r+h)^2+z^2}\rt|^qdzdr\nonumber\\
&\ \ \ \ \ \ \ \ \ \ \ \ \ \ \ \ \ \ \ \ \ +|h|^q\cdot\int_0^t\int_{\mathbb{R}}\frac{|r+h|^{q\theta}}{|(r+h)^2+z^2|^q}\cdot \frac{|r|^q}{|r^2+z^2|^q}dzdr\nonumber\\
\lesssim&\int_0^t\int_{\mathbb{R}}(r+h)^{q(\theta-1)}\frac{|h|^q}{\lt|(r+h)^2+z^2\rt|^q}dzdr\nonumber\\
&+|h|^q\cdot\lt[\int_0^t\int_{\mathbb{R}}\lt(\frac{|r+h|^{q\theta}}{\lt|(r+h)^2+z^2\rt|^q}\rt)^mdzdr\rt]^{\frac{1}{m}}\cdot \lt[\int_0^t\int_{\mathbb{R}}\lt(\frac{|r|^q}{|r^2+z^2|^q}\rt)^ndzdr\rt]^{\frac{1}{n}}\nonumber\\
=:&\cI_{2,1}^{(3)}(t,h)+|h|^q \left[
\cI_{2,2}^{(3)}(t,h)\right]^{1/m}\left[
\cI_{2,3}^{(3)}(t,h)\right]^{1/n} \label{Inq_K3_h}.
\end{align}
By the change of  variable $z\rightarrow (r+h)z$ in $\cI_{2,1}^{(3)}(t,h)$ and  $\cI_{2,2}^{(3)}(t,h)$, and by the change of variable $z\rightarrow rz$ in   $\cI_{2,3}^{(3)}(t,h)$, we have
\begin{align*}
%\int_0^t&\int_{\mathbb{R}}(r+h)^{q(\theta-1)}\lt|\Delta_h\cE(r,z)\rt|^qdzdr\\
\cI_2^{(3)} \lesssim&|h|^q\cdot\int_0^t\int_{\mathbb{R}}|r+h|^{q(\theta-1)+1-2q}\frac{1}{(1+z^2)^q}dzdr\\
&+|h|^q\cdot\lt[\int_0^t\int_{\mathbb{R}}\lt(\frac{|r+h|^{1+mq\theta-2qm}}{\lt|1+z^2\rt|^{mq}}\rt)dzdr\rt]^{\frac{1}{m}}\cdot \lt[\int_0^t\int_{\mathbb{R}}\lt(\frac{|r|^{1-nq}}{|1+z^2|^{nq}}\rt)dzdr\rt]^{\frac{1}{n}}\\
\lesssim&|h|^q\cdot\int_0^t|r+h|^{q\theta+1-3q}dr+|h|^q\cdot\lt[\int_0^t|r+h|^{1+mq\theta-2qm}dr\rt]^{\frac{1}{m}}\cdot \lt[\int_0^t|r|^{1-nq}dr\rt]^{\frac{1}{n}}\\
%\lesssim&|h|^{q}\cdot\lk |h|^{q(\theta-1)+2-2q}-|t+h|^{q(\theta-1)+2-2q} \rk \\
%&\qquad\qquad+|h|^{q}\cdot\lt[ |h|^{2+mq\theta-2qm}-|t+h|^{2+mq\theta-2qm}\rt]^{\frac{1}{m}}\cdot \lt[\int_0^t|r|^{1-nq}dr\rt]^{\frac{1}{n}}\\
\lesssim& |h|^{q(\theta-1)+2-q}+|h|^{q(\theta-1)+2/m}=|h|^{q(\theta-1)+2-q}+|h|^{q(\theta-1)+2-2/n}%<|h|^{q(\theta-1)+2-q}
\lesssim |h|^{q\gamma}\,,
\end{align*}
under condition
\begin{equation}\label{gamma_4}
\frac2n>q,\ \theta-\gamma> 2-\frac 2q\,.
\end{equation}
Then $\cI_2^{(3)}(t,h)\leq C_{T,p,H,\gamma}|h|^{\gamma q}$ under \eqref{gamma_4}.

%\tcb{Notice that the factor $2$ in \eqref{gamma_4} can be compared with the factor $1+\alpha$ in \eqref{Ineq_r_I1}.} The second inequality holds because $1<q<2$ such that $-2q<-1$ and the last inequality requires
%\begin{equation}\label{r_1_K3}
%q(\theta-1)+1-2q>-1\,,
%\end{equation}
%\begin{equation}\label{r_2_K3}
%1+mq\theta-2qm>-1,\ \ 1-nq>-1\,\Leftrightarrow\,2/n>q\,.
%\end{equation}
%Then by the proof of $\textbf{(i)}$, we have $\cI_2(t,h,x)\leq C_{T,p,\gamma,\alpha}|h|^{\gamma}\| v\|_{\cZ^p(T)}$ under \eqref{Ineq.alpha_I2}, \eqref{Ineq.alpha_I3}, \eqref{r_1_K3} and \eqref{r_2_K3}.

\medskip
\noindent $\textbf{Case i=4}$.
In this case we use  $\overline{\cK}_4(r,z)=\cS(r,z)=\frac{1}{2}\1_{\{|z|<r\}}$. Since $(r+h)^{q(\theta-1)}\leq r^{q(\theta-1)}$, we see
\begin{align*}
%\int_0^t\int_{\mathbb{R}}&(r+h)^{q(\theta-1)}|\Delta_h\cS(r,z)|^qdzdr\\
\cI_2^{(4)}(t,h) \simeq&\int_0^t\int_{\mathbb{R}}(r+h)^{q(\theta-1)}\lt|\1_{\{z|<r+h\}}-\1_{\{|z|<r+h\}}\rt|^qdzdr\\
\lesssim&\int_0^t\int_{-(r+h)}^{-r}(r+h)^{q(\theta-1)}dzdr+\int_0^t\int_{r}^{r+h}(r+h)^{q(\theta-1)}dzdr\\
=&\int_0^t2h(r+h)^{q(\theta-1)}dr\leq h\int_0^t2r^{q(\theta-1)}dr
\lesssim |h|^{\gamma q},
\end{align*}
where the last inequality requires
\begin{equation}\label{gamma_5}
q(\theta-1)>-1, \ \gamma<\frac 1q.
\end{equation}
Then under   \eqref{gamma_5}, we have
$$\sup_{t,x}
\cI_2^{(4)}(t,h)\ls |h|^{\gamma q}.$$
%\begin{eqnarray*}
%\cI_2^{(4)}(t,h,x)&=&\lt(\int_0^t\int_{\mathbb{R}}(r+h)^{q(\theta-1)}|\cS(r+h,z)-\cS(r,z)|^qdzdr\rt)^{1/q}\\
%\ \ \ &&\times \left(\int_{0}^{T}\parallel J_{\theta}^{\cE}(r,z)\parallel ^p_{L^p(\RR)}dr\right)^{1/p}\\
%&\leq& C_{T,p,\theta,\alpha}|h|^{\gamma}\| v\|_{\cZ^p(T)},
%\end{eqnarray*}
%if we set
%\begin{equation*}
%1-H<\alpha<1-\frac 1p,\ \ p>\frac1H,\ \ \gamma<\frac1q.
%\end{equation*}

To conclude, with the choice of $1-H<\alpha<1-\frac 1p\,,~p>\frac{1}{H}\,,~0<\gamma<H-\frac 1p$, we see that the condition \eqref{condi_J_est} to guarantee
$$\left(\int_{0}^{T}\parallel J_{\theta}^{\cK_i}(r,z)\parallel ^p_{L^p(\RR)}dr\right)^{1/p}\lesssim \|v\|_{\cZ^p(T)}\ (i=1,2,3,4)$$
and the conditions listed in \textbf{Case i=1,2,3,4} to guarantee \eqref{Ineq_I_2^i} are all satisfied, so we have \[
\|\underset{t,x}{\sup}\ \cI_2(t,h,x)\|_{L^p(\Omega)}\leq C_{T,p,H,\gamma}|h|^{\gamma}\|v\|_{\cZ^p(T)}\,.
\]
This  finishes  the proof of \textbf{(i)}.
%\newpage

%\begin{align}
%       & \Phi(t,x)-\Phi(t,y)\\
%      =\,&\frac{\sin(\theta\pi)}{\pi}\sum_{i}\int_{0}^{t}\int_{\RR}(t-r)^{\theta-1}\blk \bar{\cK}_{i}(t-r,x-z)\nonumber\\
%      &\qquad\qquad\qquad\qquad\qquad\qquad\qquad-\bar{\cK}_{i}(t-r,y-z)\brk J_{\theta}^{\cK_{i}}(r,z)dzdr\nonumber\\
%      \es\,&\sum_{i}\int_{0}^{t}\int_{\RR}(t-r)^{\theta-1}\bar{\cK}_{i}(t-r,z)\lk J_{\theta}^{\cK_{i}}(r,z+x)-J_{\theta}^{\cK_{i}}(r,z+y)\rk dzdr\,, \nonumber
%  \end{align}
  %\begin{equation}\label{Est.Char2_lem}
%   \EE \Blk\int_{\RR}\lt|J^{\cK_{i}}_{\theta}(r,z+h)-J^{\cK_{i}}_{\theta}(r,z)\rt|^2|h|^{2H-2}dh\Brk^{\frac p2}\leq C  \|v\|^p_{\cZ^p(T)}\,.
%  \end{equation}

\textbf{Step 2:} In this  step, we deal with
$$\sup\limits_{t\in[0,T], x,y\in \RR}\ \left| \Phi(t,x)-\Phi(t,y)\right| \,. $$
%which reflects the H\"older index of $\Phi(t,x)$ in the space variable.
By \eqref{eq.Fubini.N} in the proof of part \textbf{(ii)}  of  Proposition \ref{prop_est}, we have
\begin{align}\label{iv}
\left|\Phi(t,x) -\Phi(t,y)
\right| =&\Bigg|\sum_{i=1}^{4}\frac{\sin(\theta\pi)}{\pi} \int_{0}^{t}\int_{\RR}(t-r)^{\theta-1} [\bar{\cK}_{i}(t-r,x-z) \nonumber\\
&\qquad\qquad\qquad  -\bar{\cK}_{i}(t-r,y-z)]J^{\cK_{i}}_{\theta}(r,z)dzdr\Bigg|  \nonumber\\
&\lesssim\sum_{i=1}^{4}\lt(\int_{0}^{t}\int_{\RR}(t-r)^{q(\theta-1)}\lt|\fD_{\hbar}\bar{\cK}_{i}(t-r,z)\rt|^qdzdr\rt)^{1/q}\nonumber\\
&\qquad\qquad\qquad \ \ \times\lt(\int_0^T\|J^{\cK_{i}}_{\theta}(r,\cdot)\|_{L^p(\mathbb{R})}^pdr\rt)^{1/p}\,,
%=:\sum_{i=1}^{4} \cJ^{(i)}(t,x,y)\,,
\end{align}
where $\hbar:=|x-y|$ and $\fD_{\hbar}\bar{\cK}_{i}(t-r,z):=\bar{\cK}_{i}(t-r,z+\hbar)-\bar{\cK}_{i}(t-r,z)$. Without loss of generality, we can suppose that $x>y$ and $\hbar=|x-y|<1$ is sufficiently small. The term  $\lt(\int_0^T\|J^{\cK_{i}}_{\theta}(r,\cdot)\|_{L^p(\mathbb{R})}^pdr\rt)^{1/p}$ in \eqref{iv} can be estimated via Lemma \ref{Lem.Est_J} which requires \eqref{condi_J_est}.
Thus, we need to show for $i=1,\cdots,4$
\begin{align}
	 \sup\limits_{t,x,y}\ \cJ^{(i)}(t,x,y)\leq C_{T,p,H,\gamma}|\hbar|^{\gamma q}\,,\label{iv_main}
\end{align}
where   $\sup\limits_{t,x,y}$ is the is abbreviation for    $\sup\limits_{t\in[0,T], x,y\in \RR}$  and
\begin{align}
	\cJ^{(i)}(t,x,y):= \ \int_{0}^{t}\int_{\RR}& r^{q(\theta-1)}\lt|\fD_{\hbar}\bar{\cK}_{i}(r,z)\rt|^qdzdr\,.\label{iv_main}
\end{align} We are going to bound $\cJ^{(i)}$ for $i=1,2,3,4$ separately.

%By changing variable $r\rightarrow t-r$,  $z\rightarrow y-z$, then we only need to prove
%\begin{equation}
%\int_{0}^{t}\int_{\RR}r^{q(\theta-1)}\lt|\bar{\cK}_{i}(r,z+\Delta)-\bar{\cK}_{i}(r,z)\rt|^qdzdr\lesssim|\Delta|^{\gamma q}.
%\end{equation}
\medskip
\noindent $\textbf{Case i=1}$. In this case  $\bar{\cK}_1(r,z)=\cS_{1-\alpha}(r,z)$ which is defined by  \eqref{eq.cC_alpha}. We shall show that
\begin{equation}\label{Ineq_J^(1)}
\sup\limits_{t,x,y}\cJ^{(1)}(t,x,y)=\sup\limits_{t,x,y} \int_{0}^{t}\int_{\RR} r^{q(\theta-1)}\lt|\fD_{\hbar}\cS_{1-\alpha}(r,z)\rt|^qdzdr\leq C_{T,p,H,\gamma}|\hbar|^{\gamma q},
\end{equation}
with $\alpha,\ p$ and $\gamma$ satisfying  \eqref{Condi_iii.main}. We  split $\cJ^{(1)}(t,x,y)$ into two parts:
 \begin{align*}
	\cJ^{(1)}(t,x,y)
	=&     \int_{\hbar}^{t}\int_{\RR}r^{q(\theta-1)}\lt|\fD_{\hbar}\bar{\cK}_{1}(r,z)\rt|^qdzdr\\
	&+ \int_{0}^{\hbar}\int_{\RR}r^{q(\theta-1)}\lt|\fD_{\hbar}\bar{\cK}_{1}(r,z)\rt|^qdzdr  \\
	=:& \cJ_1^{(1)}(t,x,y)+\cJ_2^{(1)}(t,x,y)\,.
\end{align*}
Let us treat   the term $\cJ_1^{(1)}(t,x,y)$   first. In this case, $-r+\hbar<r-\hbar.$  Set
\begin{equation}\label{Def.B_set}
	\begin{cases}
		B_1=[z<-\hbar-r]\,,\quad B_2=[z>r+\hbar]\,,\quad B_3=[-r+\hbar<z<r-\hbar]\,;\\
		B_4=[-r-\hbar<z<-r+\hbar]\,,\quad B_5=[r-\hbar<z<r+\hbar]\,.
	\end{cases}
\end{equation}
%\[
%B_1=[z<-\hbar-r]\,,\quad B_2=[z>r+\hbar]\,,
%\quad B_3=[-r+\hbar<z<r-\hbar]
%\]
% and
% \[
% B_4=[-r-\hbar<z<-r+\hbar]\,,\quad B_5=[r-\hbar<z<r+\hbar]\,.
% \]
By the triangle inequality  and   the inequalities \eqref{Ineq_important},    we have
\begin{align}\label{Est.iv1}
\lt| \fD_{\hbar}\cS_{1-\alpha}(r,z)\rt|^q
\simeq\ & \bigg|(r+|z+\hbar|)^{-\alpha}+\textrm{sgn}(r-|z+\hbar|)\big|r-|z+\hbar|\big|^{-\alpha}\nonumber\\
&\qquad\ \ \, -(r+|z|)^{-\alpha}-\textrm{sgn}(r-|z|)\big|r-|z|\big|^{-\alpha}\bigg|^q\nonumber\\
\lesssim\,&\lt|\fD_{\hbar}(r+|z|)^{-\alpha} \rt|^q +\lt|\fD_{\hbar}(r-|z|)^{-\alpha}\rt|^q\cdot\lt(\1_{B_1}+\1_{B_2}+\1_{B_3}\rt)\nonumber\\
&\,+\lt|(r-|z+\hbar|)^{-\alpha}+(r-|z|)^{-\alpha}\rt|^q\cdot\lt(\1_{B_4}+\1_{B_5}\rt)\nonumber\\
\lesssim\,&|r+|z||^{-(\alpha+\eta_1)q}\nonumber\\
&\,+\lt[|r-|z||^{-(\alpha+\eta_2)q}\cdot(\1_{B_1}+\1_{B_2})+|r-\hbar-|z||^{-(\alpha+\eta_3)q}\cdot\1_{B_3}\rt]\nonumber\\
&\,+\lt|(r-|z+\hbar|)^{-\alpha}+(r-|z|)^{-\alpha}\rt|^q\cdot\lt(\1_{B_4}+\1_{B_5}\rt)\nonumber\\
=:&\sum_{k=1}^3N_{1,k}^{(1)}(t,x,y)\,,
%\nonumber\\
%&\lesssim&(r+|z|)^{(-\alpha-\eta_1)q}\hbar^{\eta_1q}+\lt|\fD_{\hbar}(r-|z|)^{-\alpha}\rt|^q\cdot\lt(1_{[z<-\hbar-r]}+1_{[z>r]}+1_{[-r<z<-\hbar+r]}\rt)\nonumber\\
%&&+\lt|(r-|z+\hbar|)^{-\alpha}+(r-|z|)^{-\alpha}\rt|^q\cdot\lt(1_{[-\hbar-r<z<-r]}+1_{[-\hbar+r<z<r]}\rt).
\end{align}
Then
\begin{equation}\label{Ineq_J_1,k^1}
 \cJ_1^{(1)}(t,x,y)\lesssim \sum_{k=1}^3\cJ_{1,k}^{(1)}(t,x,y):=\sum_{k=1}^3\int_{\hbar}^{t}\int_{\RR}r^{q(\theta-1)}N_{1,k}^{(1)}(t,x,y)dzdr.
\end{equation}
By Lemma \ref{J_{1,k}^{(1)}}, $\sup\limits_{t,x,y}\cJ_{1,k}^{(1)}(t,x,y)\ls |\hbar|^{\gamma q}$ for $k=1,2,3$ if we require \eqref{list1} and \eqref{Cond_AppeC3}.

 Next, we shall deal with $\cJ_2^{(1)}(t,x,y)$.  In this case, $-r+\hbar\geq r-\hbar.$ Setting
 \begin{equation}\label{Def.C_set}
 	C_1=[z<-r-\hbar],\ C_2=[z>r+\hbar],\ C_3=[-r-\hbar<z<r+\hbar]\,,
 \end{equation}
% $C_1=[z<-r-\hbar],\ C_2=[z>r+\hbar],\ C_3=[-r-\hbar<z<r+\hbar]$,
 then by the inequalities \eqref{Ineq_important},
\begin{align}\label{Est.iv1_2}
\lt| \fD_{\hbar}\cS_{1-\alpha}(r,z)\rt|^q
\simeq\ & \bigg|(r+|z+\hbar|)^{-\alpha}+\textrm{sgn}(r-|z+\hbar|)\big|r-|z+\hbar|\big|^{-\alpha}\nonumber\\
&\qquad\ \ \, -(r+|z|)^{-\alpha}-\textrm{sgn}(r-|z|)\big|r-|z|\big|^{-\alpha}\bigg|^q\nonumber\\
\lesssim\,&\lt|\fD_{\hbar}(r+|z|)^{-\alpha} \rt|^q +\lt|\fD_{\hbar}(r-|z|)^{-\alpha}\rt|^q\cdot\lt(\1_{C_1}+\1_{C_2}\rt)\nonumber\\
&\,+\lt|(r-|z+\hbar|)^{-\alpha}+(r-|z|)^{-\alpha}\rt|^q\cdot \1_{C_3}\nonumber\\
\lesssim\,&|r+|z||^{-(\alpha-\eta_1)q}+|r-|z||^{-(\alpha-\eta_4)q}\cdot(\1_{C_1}+\1_{C_2})\nonumber\\
&\,+\lt|(r-|z+\hbar|)^{-\alpha}+(r-|z|)^{-\alpha}\rt|^q\cdot \1_{C_3}\nonumber\\
=:&\sum_{k=1}^3N_{2,k}^{(1)}(t,x,y)\,.
%\nonumber\\
%&\lesssim&(r+|z|)^{(-\alpha-\eta_1)q}\hbar^{\eta_1q}+\lt|\fD_{\hbar}(r-|z|)^{-\alpha}\rt|^q\cdot\lt(1_{[z<-\hbar-r]}+1_{[z>r]}+1_{[-r<z<-\hbar+r]}\rt)\nonumber\\
%&&+\lt|(r-|z+\hbar|)^{-\alpha}+(r-|z|)^{-\alpha}\rt|^q\cdot\lt(1_{[-\hbar-r<z<-r]}+1_{[-\hbar+r<z<r]}\rt).
\end{align}
Thus
\begin{equation}\label{J_{2,K}^1}
\cJ_2^{(1)}(t,x,y)\lesssim \sum_{k=1}^3\cJ_{2,k}^{(1)}(t,x,y):=\sum_{k=1}^3\int_{0}^{\hbar}\int_{\RR}r^{q(\theta-1)}N_{2,k}^{(1)}(t,x,y)dzdr.
\end{equation}
By Lemma \ref{J_{2,k}^{(1)}}, $\sup\limits_{t,x,y}\cJ_{2,k}^{(1)}(t,x,y)\ls |\hbar|^{\gamma q}$ for $k=1,2,3$ under conditions  \eqref{list1} and \eqref{Cond_AppeC4}.

As a result, for any $p>\frac 1H,$ $\gamma<H-\frac 1p$, we know that \eqref{Ineq_J^(1)} holds  if $\alpha$, $\theta$ satisfy \eqref{list1} and $\eta_k,k=1,\cdots,4 $ satisfy   \eqref{Cond_AppeC3} and \eqref{Cond_AppeC4}.

%To summarize, we obtain that conditions  \eqref{condi_eta_1}, \eqref{condi_eta3_iv} , \eqref{condi_gamma_iv} and \eqref{condi_J_{2,3}} guarantee \eqref{Ineq_J^(1)}.  These conditions are listed as follows.
%\begin{enumerate}
%%\item $1-H<\alpha<1-\frac 1p$\,,\ $1+\alpha-\frac 2q+\gamma<\theta<H+\alpha-\frac 12;$
%\item $\eta_k>\gamma$\,,\ $\alpha+\eta_k>\frac 1q\,,\ \theta-\eta_k>1+\alpha-\frac 2q,\ k=1,4$\,;
%\item $\eta_2>\gamma$\,,\ $\alpha+\eta_2>\frac 1q$\,,\ $\alpha+\gamma<\frac 1q$\,,\ $\theta>\frac 1p$\,;
%\item $\eta_3>\gamma$\,,\ $\alpha+\eta_3<\frac1q$\,,\ $\theta-\eta_3>1+\alpha-\frac 2q$\,;
%\item $\alpha<\frac 1q=1-\frac 1p$ ,\  $\alpha+\gamma<\frac1q$,\ $\theta>\frac 1p$\,;
%\item  $\alpha<\frac 1q=1-\frac 1p$,\  $\theta-\gamma>1+\alpha-\frac2q$,\ $\theta>\frac 1p$\,.
%
%\end{enumerate}
%\tcb{For example, we may select $p$ large, $\alpha=1-H+\ep_1$, $\gamma=H-1/p-\ep_2$ and $\theta=H-\ep_3$ to be fixed. And we can choose $\eta_1=\eta_2=H-1/p$ and $\eta_3=H-1/p-\ep_2/2$. So we can let $\ep_1<\ep_2/2$ and $\ep_2/2+\ep_3<H-1/p$ ($\ep_1,\ep_2,\ep_3>0$ be small enough).} Therefore,

% to obtain \eqref{iv_main} for $i=1$.
%%$$\sup\limits_{t,x,y}\ \cJ^{(1)}(t,x,y) \leq C_{T,p,H,\gamma}|x-y|^{\gamma}\|v\|_{\cZ^p(T)}.$$
%%\begin{equation}\label{iv1}
%%1-H<\alpha<1-\frac1p,\ p>\frac1H,\ 0<\gamma<H-\frac1p,
%%\end{equation}
%%to obtain

\medskip
\noindent $\textbf{Case i=2}$. We   consider $\bar{\cK}_2(r,z)=\cC_{1-\alpha}(r,z)$ defined by  \eqref{eq.cC_alpha}.
We want to obtain
\begin{align}
	\sup\limits_{t,x,y}\cJ^{(2)}(t,x,y)= \sup\limits_{t,x,y} \int_{0}^{t}\int_{\RR}& r^{q(\theta-1)}\lt|\fD_{\hbar}\cC_{1-\alpha}(r,z)\rt|^qdzdr\leq C_{T,p,H,\gamma}|\hbar|^{\gamma q}\,,\label{iv_casei=2}
\end{align}
with parameters $p,\ \alpha,\ \gamma$ satisfy \eqref{Condi_iii.main}. By the triangle inequality,
\begin{align}\label{Est.ivcase2}
\lt| \fD_{\hbar}\cC_{1-\alpha}(r,z)\rt|^q\ls&|\fD_{\hbar}|r+|z||^{-\alpha}|^q+|\fD_{\hbar}|r-|z||^{-\alpha}|^q+\lt|\fD_{\hbar}(r^2+z^2)^{-\frac{\alpha}{2}}\rt|^q\nonumber\\
&+\lt[2\ \fD_{\hbar}\cos\lc\alpha\tan^{-1}\lc\frac{|z|}{r}\rc\rc\rt]^q\lt(r^2+z^2\rt)^{-\frac{\alpha}{2}q}\nonumber\\
=:&\sum_{k=1}^4N_k^{(2)}(r,z)\,.
\end{align}
Substituting \eqref{Est.ivcase2} into \eqref{iv_casei=2}, we have
\begin{eqnarray*}
\sup\limits_{t,x,y}\cJ^{(2)}(t,x,y)\le \sup\limits_{t,x,y} \sum_{k=1}^4\cJ_k^{(2)}(t,x,y)\,,
\end{eqnarray*}
where
$$\cJ_k^{(2)}(t,x,y)=\int_{0}^{t}\int_{\RR}r^{q(\theta-1)}N_k^{(2)}(r,z) dzdr, \ k=1,\cdots,4.$$

For the term $\cJ_1^{(2)}(t,x,y)$, since for fixed $\eta_1\in(0,1)$ , $$|\fD_{\hbar}|r+|z||^{-\alpha}|^q\leq|r+|z||^{-(\alpha+\eta_1)q}|\hbar|^{\eta_1 q},$$
similar to the estimation of $\cI_{2,1}^{(1)}(t,h)$ in \eqref{Est.I2(1)}, we have $\sup\limits_{t,x,y}\cJ_1^{(2)}(t,x,y)\ls |\hbar|^{\gamma q}$
under the  condition    \eqref{condi_eta_1}.

It is more complicated to deal with the term $\cJ_2^{(2)}(t,x,y)$ since $|\fD_{\hbar}|r-|z||^{-\alpha}|^q$ has different upper bounds on different domains  of $|z|$. Similar to $\textbf{Case i=1}$,  we  split $\cJ^{(1)}(t,x,y)$ into two parts
\begin{align}\label{Ineq_J_2^2}
\cJ_2^{(2)}(t,x,y)=&\int_{\frac\hbar2}^{t}\int_{\RR}r^{q(\theta-1)}|\fD_{\hbar}|r-|z||^{-\alpha}|^qdzdr\nonumber\\
&+\int_{0}^{\frac{\hbar}{2} }\int_{\RR}r^{q(\theta-1)}|\fD_{\hbar}|r-|z||^{-\alpha}|^qdzdr\nonumber\\
=:&\cJ_{2,1}^{(2)}(t,x,y)+\cJ_{2,2} ^{(2)}(t,x,y)\,.
\end{align}
We first deal with $\cJ_{2,1}^{(2)}(t,x,y)$ when      $r>\frac\hbar2 $, namely  $-r<r-\hbar$. Let us set
\begin{equation}\label{Def.D_set}
	\begin{cases}
		D_1=[z<-r-\hbar],\quad D_2=[-r-\hbar<z<-r],~D_3=[-r<z<r-\hbar]\,,  \\
		 D_4=[r-\hbar<z<r],~ D_5=[r<z<r+\hbar],~D_6=[r>z+\hbar]\,.
	\end{cases}
\end{equation}
% $D_1=[z<-r-\hbar],\ D_2=[-r-\hbar<z<-r],\ D_3=[-r<z<r-\hbar],\ D_4=[r-\hbar<z<r],\ D_5=[z>r],$
The first integral of \eqref{Ineq_J_2^2} can be bounded by
\begin{align}\label{Ineq_J_21^2}
%\int_{\frac\hbar2}^{t}\int_{\RR}&r^{q(\theta-1)}|\fD_{\hbar}|r-|z||^{-\alpha}|^qdzdr\nonumber\\
\cJ_{2,1}^{(2)}(t,x,y)
\ls&\sum_{j=1}^6\int_{0}^{t}\int_{D_j}r^{q(\theta-1)}|\fD_{\hbar}|r-|z||^{-\alpha}|^qdzdr=:\sum_{j=1}^6\cJ_{2,1,j}^{(2)}(t,x,y)\,.
\end{align}
It is not hard to derive that for some  $\eta\in(0,1)$
\begin{equation*}
	\lt|\fD_{\hbar}|r-|z||^{-\alpha}\rt| \ls
	\begin{cases}
		|r-|z||^{-\alpha-\eta}\hbar^{\eta}\,, &\text{on } D_1\cup D_5\cup D_6\,;\\
		|r-|z+\hbar||^{-\alpha-\eta}\hbar^{\eta}\,, &\text{on } D_3\,;\\
		|r-|z+\hbar||^{-\alpha}+|r-|z||^{-\alpha}\,, &\text{on } D_2\cup D_4\,.
	\end{cases}
\end{equation*}
Substituting this into \eqref{Ineq_J_21^2} we obtain
\begin{align*}
\sum_{j=1}^6\cJ_{2,1,j}^{(2)}(t,x,y)\leq&\int_{0}^{t}\int_{D_1\cup D_5}r^{q(\theta-1)}|r-|z||^{-(\alpha+\eta_2)q}\hbar^{\eta_2q}dzdr\\
&+\int_{0}^{t}\int_{D_6}r^{q(\theta-1)}|r-|z||^{-(\alpha+\eta_3)q}\hbar^{\eta_3q}dzdr\\
&+\int_{0}^{t}\int_{D_3}r^{q(\theta-1)}|r-|z+\hbar||^{-(\alpha+\eta_4)q}\hbar^{\eta_4q}dzdr\\
&+\int_{0}^{t}\int_{D_2\cup D_4}r^{q(\theta-1)}\lt(|r-|z+\hbar||^{-\alpha q}+|r-|z||^{-\alpha q}\rt)dzdr\,.
\end{align*}
By Lemma \ref{J_{2,j}^2} in Appendix \ref{Appen.C}, we have $$\sup\limits_{t,x,y}\cJ_{2,1,j}^{(2)}(t,x,y)\ls|\hbar|^{\gamma q},\ j=1,\cdots,6,$$
under conditions \eqref{list1} and \eqref{Cond_AppeC5}.

In similar way we can obtain the same bound
for $\cJ_{2,2} ^{(2)}(t,x,y)$ by  dividing  the domain of $|z|$ into subdomains  and  estimating  each terms. We omit the details  here.

Now we turn to the third and last terms $\cJ_3^{(2)}(t,x,y)$ and   $\cJ_4^{(2)}(t,x,y)$.  Analogously to the obtention of    \eqref{r2_regular}  and \eqref{Ineq_imp_2}, it is not hard to obtain for fixed $\eta\in(0,1),$
\begin{equation}\label{regular_J_3^2}
\lt|\fD_{\hbar}(r^2+z^2)^{-\frac{\alpha}{2}}\rt| \leq(r^2+z^2)^{-(\frac{\alpha}{2}+\eta ) }|z+\hbar|^{\eta  }|\hbar|^{\eta  },
\end{equation}
and
\begin{equation}\label{regular_J_4^2}
 \left| \fD_{\hbar}\cos\lc\alpha\tan^{-1}\lc\frac{|z|}{r}\rc\rc \right| \leq \frac{r^{\eta  } |\hbar|^{\eta  }}{(r^2+z^2)^{\eta  }} .
\end{equation}
Then we have
\begin{align}\label{Ineq_J_3+J_4}
\cJ_3^{(2)}(t,x,y)&+\cJ_4^{(2)}(t,x,y)\nonumber\\
\ls&|\hbar|^{\eta_4 q}\int_{0}^{t}\int_{\RR}r^{q(\theta-1)}(r^2+z^2)^{-(\frac{\alpha}{2}+\eta_4 )q}|z+\hbar|^{\eta_4 q}dzdr\nonumber\\
&+|\hbar|^{\eta_5 q}\int_{0}^{t}\int_{\RR}r^{q(\theta-1)}\frac{r^{\eta_5 q} }{(r^2+z^2)^{\eta_5 q}}(r^2+z^2)^{-\frac{\alpha}{2} q}dzdr.
\end{align}
By Lemma \ref{lem_J_3+J_4}, $\sup\limits_{t,x,y}\cJ_3^{(2)}(t,x,y)$ and $\sup\limits_{t,x,y}\cJ_4^{(2)}(t,x,y)$ can be bounded by a multiple of $|\hbar|^{\gamma q}$ under conditions \eqref{list1} and \eqref{Cond_AppeC6}.

As a result, for any $p>\frac 1H,$ $\gamma<H-\frac 1p$, $\sup\limits_{t,x,y}\cJ^{(2)}(t,x,y)\ls |\hbar|^{\gamma q}$  if $\alpha$, $\theta$ satisfy \eqref{list1},  $\eta_k$  $(k=1,\cdots,5)$ satisfy \eqref{Cond_AppeC5} and \eqref{Cond_AppeC6}.

\bigskip

%To summarize, we see  that the conditions
%\eqref{condi_eta_1}, \eqref{condi_J_21^2}, \eqref{condi_J_22^2}, \eqref{condi_J_23^2} and \eqref{Condi_J_3^2} can be  listed as the following conditions to guarantee \eqref{iv_casei=2}.
%\begin{enumerate}
%%\item $1-H<\alpha<1-\frac 1p$\,,\ $1+\alpha-\frac 2q+\gamma<\theta<H+\alpha-\frac 12$\,;
%\item $\eta_k>\gamma\,,\ \theta-\eta_k>1+\alpha-\frac 2q\,,\ \alpha+\eta_k>\frac 1q $ for $k=1,3,5$\,;
%\item $\alpha+\eta_2>\frac 1q,\ \eta_2>\gamma$\,;
%\item $\alpha+\eta'_2<\frac 1q,\ \eta'_2>\gamma$\,;
%\item $\alpha<\frac 1q,\ \alpha+\gamma<\frac 1q$\,;
%\item  $\eta_4>\gamma$\,, \ $\theta-2\eta_4>1+\alpha-\frac 2q\,, \ \alpha+\eta_4>\frac 1q$\,.
%\end{enumerate}

%\tcb{For instance, we may select $p$ large, $\alpha=1-H+\ep_1$, $\gamma=H-1/p-\ep_2$ and $\theta=H-\ep_3$ to be fixed. And we can choose $\eta_1=\bar{\eta}_2=H-1/p$, $\eta_2=\eta_3=H-1/p-\ep_2/2$ and $\bar{\eta}_1=H-1/p-\ep_2/4$. So we can let $\ep_2/4<\ep_1<\ep_2/2$, $\ep_1+\ep_3<\ep_2/2$ and $\ep_2/2+\ep_3<H-1/p$ ($\ep_1,\ep_2,\ep_3>0$ be small enough).}
%$$\bar{\cI}^{(2)}\lesssim C_{T,p,H,\gamma}|x-y|^{\gamma}\|v\|_{\cZ^p(T)}.$$

\medskip
\noindent$\textbf{Case  i=3}$. In this case  $\bar{\cK}_3(r,z)=\cE(r,z)=\frac{1}{\pi} \frac{r}{r^2+z^2}$.  Then
\begin{align}\label{Est.iv3}
\int_{0}^{t}&\int_{\RR}r^{q(\theta-1)}\lt|\fD_{\hbar}\bar{\cK}_{3}(r,z)\rt|^qdzdr=\int_{0}^{t}\int_{\RR}r^{q\theta}\lt|\frac{1}{r^2+(z+\hbar)^2}-\frac{1}{r^2+z^2}\rt|^qdzdr.
\end{align}
The $\hbar=|x-y|$ in \eqref{Est.iv3} plays  the same role  as $h$ in the second term  of \eqref{Est.III3}. So using the similar method as that in dealing with $\lt|\Delta_h\lt(\frac{1}{r^2+z^2}\rt)\rt|^q $ in \textbf{Case i=3} of  \textbf{Step 1}, we have
$$\int_{0}^{t}\int_{\RR}r^{q\theta}\lt|\fD_{\hbar}\lt(\frac{1}{r^2+z^2}\rt)\rt|^qdzdr\lesssim\hbar^{\gamma q},$$
if  $\theta-\gamma>2-\frac2q.$
Thus, under \eqref{Condi_iii.main} we have
$$\lt(\cJ^{(3)}(t,x,y)\rt)^{1/q}\times\lt(\int_0^T\|J_{\theta}^{\cK_3}(r,\cdot)\|_{L^p(\RR)}dr\rt)^{1/p}\lesssim C_{T,p,H,\gamma}|x-y|^{\gamma}\|v\|_{\cZ^p(T)}.$$

\bigskip
\noindent\textbf{Case i=4}. In this case  $\bar{\cK}_4(r,z)=\cS(r,z)=\frac12 \1_{\{|z|<r\}}$.  Then
\begin{align*}
\int_{0}^{t}\int_{\RR}&r^{q(\theta-1)}\lt|\fD_{\hbar}\bar{\cK}_{4}(r,z)\rt|^qdzdr\\
=&\int_{0}^{t}\int_{\RR}r^{q(\theta-1)}\lt|\frac12 \1_{\{|z+x-y|<r\}}-\frac12 \1_{\{|z|<r\}}\rt|^qdzdr\\
\simeq&\int_{0}^{t}r^{q(\theta-1)}\lt(\int_{y-x-r}^{-r}dz+\int_{y-x+r}^{r}dz\rt)dr\simeq\hbar\cdot\int_{0}^{t}r^{q(\theta-1)}\lesssim \hbar^{\gamma q},
\end{align*}
under the conditions  $q(\theta-1)>-1$ and $\gamma<\frac1q.$  Therefore, under \eqref{Condi_iii.main} we have
$$\lt(\cJ^{(4)}(t,x,y)\rt)^{1/q}\times\lt(\int_0^T\|J_{\theta}^{\cK_4}(r,\cdot)\|_{L^p(\RR)}dr\rt)^{1/p}\lesssim C_{T,p,H,\gamma}|x-y|^{\gamma}\|v\|_{\cZ^p(T)}.$$
\bigskip

In conclusion, with the choice of $~p>\frac{1}{H}\,,1-H<\alpha<1-\frac 1p\,,~0<\gamma<H-\frac 1p$, the conditions listed in \textbf{Case i=1,2,3,4} to ensure
\begin{align*}
	\sup\limits_{t,x,y}\cJ^{(i)}(t,x,y)\ls |\hbar|^{\gamma q},\  \,
\end{align*}
and the condition \eqref{condi_J_est} to ensure $$\lt(\int_0^T\|J_{\theta}^{\cK_i}(r,\cdot)\|_{L^p(\RR)}dr\rt)^{1/p}\ls\|v\|_{\cZ^p(T)},$$
are all satisfied. Thus, we have
\[
\|\underset{t\in[0,T],x,y\in\RR}{\sup}\left|\Phi(t,x)-\Phi(t,y)\right|\|_{L^p(\Omega)}\lesssim C_{T,p,H,\gamma}|x-y|^{\gamma}\|v\|_{\cZ^p(T)}\,.
\]
This completes  the proof of \textbf{(ii)}.
%\tcb{\textbf{Lemma 2.5} \ \ Let \tcr{$p>\frac 1H$, $1-2/q+\alpha<\theta<2H+\alpha-1$ and $\frac 32-2H<\alpha<1-\frac 1p$}. There is a constant $C$ independent of $r\in[0,T]$ such that
%  \begin{equation}\label{Est.Char2_lem}
%   \EE \Blk\int_{\RR}\lt|J^{\cK_{i}}_{\theta}(r,z+h)-J^{\cK_{i}}_{\theta}(r,z)\rt|^2|h|^{2H-2}dh\Brk^{\frac p2}\leq C  \|v\|^p_{\cZ^p(T)}\,.
%  \end{equation}
%  where $J^{\cK_{i}}_{\theta}$ depending on $\alpha,\theta$ and $\cK_{i}$ ($i=1,2,3,4$) depending on $\alpha$ are defined in \eqref{def.JK_theta}  and \eqref{def.cK} respectively.}
  \end{proof}

%\begin{lem}\label{weak_est}
% Let $H\in(\frac 14,\frac 12)$. Assume $\sigma(t,x,u)$ satisfies hypothesis \ref{H}  and the initial value $u_0(x)\in\cZ_{0}^p$. Then the approximate solutions $u_\ep$ satisfy
%   \begin{equation}\label{ReguBdd}
%     \sup_{\ep>0}\|u_\ep\|_{\cZ_{T}^p}:=\sup_{\ep>0}\sup_{t\in[0,T]}\|u_\ep(t,\cdot)\|_{L^p(\Omega\times\RR)} +\sup_{\ep>0}\sup_{t\in[0,T]}\cN^*_{\frac 12-H,p}u_\ep(t)<\infty.
%   \end{equation}
%\end{lem}
%\subsection{H\"older continuity of the solution}

\subsection{H\"older continuity of the approximate solutions and well-ponsedness}
Analogous  to Proposition \ref{p.4.1} we have the  following regularity results for the approximated solution $u_\ep$ defined in \eqref{MildSolRegu}. The proof is similar and we omit it.
\begin{lem}\label{TimeSpaceRegBdd}
   Let $u_\ep$ be the approximation mild solution defined by  \eqref{MildSolRegu} and assume   that $I_0(t,x)$ belongs to $\cZ^p(T)$.
   \begin{enumerate}[leftmargin=*]%[itemindent=-1.5em]
     \item[\textbf{(i)}] If $p>\frac{2}{4H-1}$, then
     \begin{equation}
       \Big\|\sup\limits_{t\in[0,T], x\in \RR} |\cN_{\frac 12-H}u_\ep(t,x)|\Big\|_{L^{p}(\Omega)}\leq C_{T,p,H} \|u_\ep\|_{\cZ^p(T)}\,.
     \end{equation}
%     where $\cN_{\frac 12-H}\Phi(t,x)=\cN_{\frac 12-H}^{\RR}\Phi(t,x)$ as defined in Definition \ref{NBNorm}.
     \item[\textbf{(ii)}] If $p>\frac{1}{H}$ and $0<\gamma<H-\frac{1}{p}$, then
     \begin{equation}
       \Big\|\sup\limits_{{t,t+h\in[0,T],  x\in \RR}} | u_\ep(t+h,x)-u_\ep(t,x)|\Big\|_{L^p(\Omega)}\leq C_{T,p,H,\gamma}|h|^{\gamma} \|u_\ep\|_{\cZ^p(T)}\,.
     \end{equation}

     \item[\textbf{(iii)}] If $p>\frac{1}{H}$ and $0<\gamma<H-\frac{1}{p}$, then
     \begin{equation}
       \lt\|\sup\limits_{{t\in[0,T], x,y\in \RR}}|u_\ep(t,x)-u_\ep(t,y)| \rt\|_{L^p(\Omega)}\leq C_{T,p,H,\gamma}|x-y|^{\gamma} \|u_\ep\|_{\cZ^p(T)}\,.
     \end{equation}

   \end{enumerate}
\end{lem}

Finally, we are in position  to prove our main results.
\begin{proof}[Proof of Theorem \ref{Thm_1} and Theorem \ref{Thm_2}]
As we know the uniformly H\"older continuity of the type specified in Lemma \ref{TimeSpaceRegBdd} is the most important ingredient in  the proof   (\cite[Theorem 1.5]{HW2019})  of  the existence of weak solution to the nonlinear stochastic heat equation.  It is also the most important one to show the existence of weak solution for nonlinear stochastic wave equation     \eqref{eq.SWE}. Hence we omit the details of  the proof of Theorem \ref{Thm_1}. Since the pathwise uniqueness implies the existence of strong solution by the Yamada-Watanabe theorem (in the SPDEs setting, e.g. \cite{KS1998, Kurz2007}), we only need to focus on the proof of pathwise uniqueness. We follow the same strategy in \cite{HHLNT2017,HW2019} together with the crucial estimate \eqref{Est.Z2.N} in Proposition \ref{prop_est}.
	
	Suppose $u(t,x)$ and $v(t,x)$ are two solution to \eqref{eq.SWE}. Define the following stopping times:
	\begin{equation*}
    \begin{split}
        \fT_k:= \inf \Big\{ t\in[0,T]:&  \sup_{0\leq s\leq t,x\in\RR} \cN_{\frac 12-H}u(s,x)\geq k, \\
         &\quad~\text{or}~\sup_{0\leq s\leq t,x\in\RR} \cN_{\frac 12-H}v(s,x)\geq k \Big\}\,,\quad k=1, 2, \cdots
    \end{split}
  \end{equation*}
  Recall that the inequality \eqref{Est.Z2.N} in Proposition \ref{prop_est} implies that $\fT_k \uparrow T$ almost surely as $k\rightarrow\infty$. This is a key fact to our method.
  We need to find appropriate bounds for   the following two quantities:
  \[
   \fJ_1(t)=\sup_{x\in\RR}\EE\blk\1_{\{t<\fT_k\}}|u(t,x)-v(t,x)|^2\brk
  \]
  and
  \[
   \fJ_2(t)=\sup_{x\in\RR}\EE\lt[\int_{\RR}\1_{\{t<\fT_k\}}|u(t,x)-v(t,x)-u(t,x+h)+v(t,x+h)|^2|h|^{2H-2}dh\rt].
  \]
  By the elementary properties of It\^o's integral, we have
  \begin{align*}
  	\1_{\{t<\fT_k\}}&[u(t,x)-v(t,x)] \\
  	=&\1_{\{t<\fT_k\}}\int_0^t\int_\RR G_{t-s}(x-y)\1_{\{s<\fT_k\}} [\sigma(s,y,u(s,y))-\sigma(s,y,v(s,y))] W(ds,dy)\,.
  \end{align*}
  Therefore, denoting $\triangle (t,x,y):=\sigma(t,x,u(t,y))-\sigma(t,x,v(t,y))$ we have
  \begin{align}
  	&\EE\blk\1_{\{t<\fT_k\}}|u(t,x)-v(t,x)|^2\brk \nonumber\\
     \lesssim&\ \EE\bigg(\int_{0}^{t}\int_{\RR^2}\1_{\{s<T_k\}} |\fD_{h}G_{t-s}(x-y)|^2 [ \triangle(s,y,y)]^2|h|^{2H-2}dhdyds\bigg) \nonumber\\
     +&\ \EE\bigg(\int_{0}^{t}\int_{\RR^2}\1_{\{s<T_k\}} G^2_{t-s}(x-y-h)  [\triangle(s,y+h,y)-\triangle(s,y,y) ]^2|h|^{2H-2}dhdyds\bigg) \nonumber\\
     +&\ \EE\bigg(\int_{0}^{t}\int_{\RR^2}\1_{\{s<T_k\}}  G_{t-s}^2(x-y)  [\triangle(s,y,y+h)-\triangle(s,y,y) ]^2|h|^{2H-2}dhdyds\bigg)\nonumber \\
     =:&\, I_{1,1}+I_{1,2}+I_{1,3}\,.     \label{e.I1_unq}
  \end{align}
  The assumption \eqref{LipSigam} on  $\sigma$  can be used to estimate $I_{1,1}$.
  This is,
  \begin{align*}
%    &\EE\bigg(\int_{0}^{t}\int_{\RR^2}\1_{\{s<T_k\}} |\fD_{h}G_{t-s}(x-y)|^2 [ \triangle(s,y,y)]^2|h|^{2H-2}dhdyds\bigg) \\
    I_{1,1}\lesssim&\ \EE\bigg(\int_{0}^{t}\int_{\RR^2}\1_{\{s<T_k\}} |\fD_{h}G_{t-s}(x-y)|^2 |u(s,y)-v(s,y)|^2|h|^{2H-2}dhdyds\bigg)\\
    \lesssim& \int_{0}^{t}(t-s)^{2H} \sup_{y\in\RR}\EE\lk\1_{\{s<T_k\}}|u(s,y)-v(s,y)|^2\rk ds \\
    =& \int_{0}^{t}(t-s)^{2H} \fJ_1(s)ds\,.
  \end{align*}
  Using the property \eqref{DuSigam} of $\sigma$,   we
  have if $|h|>1$
  \begin{align*}
     [\triangle(s,y+h,y)-\triangle(s,y,y)]^2&
    =\lt|\int_{u}^{v} \lk \frac{\partial}{\partial \xi}\sigma(s,y+h,\xi)-\frac{\partial}{\partial \xi}\sigma(s,y,\xi)\rk d\xi\rt|^2 \\
    &\lesssim\, |u(s,y)-v(s,y)|^2\,.
  \end{align*}
If $|h|\leq 1$,  with the help of  additional property \eqref{DxuSigam} we get
  \begin{align*}
    [\triangle(s,y+h,y)-&\triangle(s,y,y)]^2 \\
   =&\lt|\int_{u}^{v} \lk\frac{\partial}{\partial \xi}\sigma(s,y+h,\xi)-\frac{\partial}{\partial \xi}\sigma(s,y,\xi)\rk d\xi\rt|^2 \\
   =&\lt|\int_{u}^{v}\int_{0}^{h} \frac{\partial^2}{\partial \eta\partial \xi}\sigma(s,y+\eta,\xi) d\eta d\xi\rt|^2 \\
   \lesssim&\, |h|^2|u(s,y)-v(s,y)|^2\,.
  \end{align*}
  Thus,  the term $I_{1,2}$ in \eqref{e.I1_unq} is bounded by
  \begin{align*}
%    &\EE\bigg(\int_{0}^{t}\int_{\RR}\int_{|h|>1}\1_{\{s<T_k\}} G^2_{t-s}(x-y-h) |u(s,y)-v(s,y)|^2|h|^{2H-2}dhdyds\bigg)\\
%    &+\EE\bigg(\int_{0}^{t}\int_{\RR}\int_{|h|\leq 1}\1_{\{s<T_k\}} G^2_{t-s}(x-y-h) |u(s,y)-v(s,y)|^2|h|^{2H}dhdyds\bigg)\\
   I_{1,2}\lesssim& \int_{0}^{t}\fJ_1(s)\lc\int_{\RR} G^2_{t-s}(x-y)dy\rc ds\lesssim\int_{0}^{t} (t-s) \fJ_1(s)ds\,.
  \end{align*}
  For the last term $I_{1,3}$ in \eqref{e.I1_unq},  by \eqref{DuSigam} and  \eqref{DuSigamAdd} we have
  \begin{align*}
    &\big|\triangle(s,y,y+h)-\triangle(s,y,y)\big|^2 \\
    =&\Big|\int_{0}^{1}[u(s,y+h)-v(s,y+h)]\frac{\partial}{\partial \xi}\sigma(s,y,\theta u(s,y+h)+(1-\theta) v(s,y+h))d\theta \\
    &\quad-\int_{0}^{1}[u(s,y)-v(s,y)]\frac{\partial}{\partial \xi}\sigma(s,y,\theta u(s,y)+(1-\theta) v(s,y))d\theta \Big|^2 \\
    \lesssim& |u(s,y+h)-v(s,y+h)-u(s,y)+v(s,y)|^2 \\
    +&|u(s,y)-v(s,y)|^2\cdot\blk|u(s,y+h)-u(s,y)|^2+|v(s,y+h)-v(s,y)|^2\brk\,.
  \end{align*}
  Thus, we can get
  \begin{align*}
  	I_{1,3}\, \ls\, k\int_{0}^{t} \blk \fJ_1(s)+\fJ_2(s)\brk ds\,.
  \end{align*}
   Summarizing  the above  estimates   we have
  \begin{equation*}
    \fJ_1(t)\leq k\cdot C_{T}\int_{0}^{t} \blk \fJ_1(s)+\fJ_2(s)\brk ds\,,
  \end{equation*}
where $C_T>0$ and the constant $k$  depends on the stopping times $\fT_k$.

A  similar procedure to the obtention of \eqref{Est.u_n+1_Z2} can be applied to estimate term $\fJ_2(t)$  to obtain
  \begin{equation*}
    \fJ_2(t)\lesssim k\int_{0}^{t} (t-s)^{4H-1}\blk \fJ_1(s)+\fJ_2(s)\brk ds\,.
  \end{equation*}
  As a consequence,
  \begin{equation*}
    \fJ_1(t)+\fJ_2(t)\lesssim k\int_{0}^{t} (t-s)^{4H-1}\blk \fJ_1(s)+\fJ_2(s)\brk ds.
  \end{equation*}
Now  Gronwall's lemma implies $\fJ_1(t)+\fJ_2(t)=0$ for all $t\in[0,T]$.
This means  we have
  \[
   \EE\blk\1_{\{t<\fT_k\}}|u(t,x)-v(t,x)|^2\brk=0\,.
  \]
  Thus,  we have $u(t,x)=v(t,x)$ almost surely on the set $\{t<\fT_k\}$ for all $k\geq 1$, and the fact $\fT_k \uparrow T$ a.s as $k$ tends to infinity necessarily indicate $u(t,x)=v(t,x)$ a.s. for every $(t,x)\in[0,T]\times\RR$.

  It is clear   that hypothesis \ref{H2} implies the hypothesis  \ref{H1}. So equation \eqref{eq.SWE} has a weak solution by Theorem \ref{Thm_1}. This combined with the above pathwise uniqueness  yields Theorem \ref{Thm_2}.

\end{proof}

\section{Necessity of $H>\frac 14$}\label{s.5}
In   Theorem  \ref{Thm_1} and Theorem \ref{Thm_2}, we see  that $H>\frac 14$ is a sufficient condition for the solvability of equation \eqref{eq.SWE}.  In this section we shall prove   that it is also necessary for some specific stochastic wave equations,  namely,  the hyperbolic Anderson equation \eqref{eq.HAE}.
It is known that  if $\|v(t,x)\|_{L^2(\Omega)}<\infty$ the solution admits
 the following unique Wiener chaos expansion (see \cite{Hu2017,Nualart2006}):
 \begin{align}
 	v(t,x)=& I_0(t,x)+\sum_{n=1}^{\infty} I_n(g_n(t,x))\,,
% 	=\,& I_0(t,x)+\int_{0}^{t}\int_{\RR^d} G_{t-s}(x-y) I_0(s,y) W(d s,d y) \\
%	&+\int_{0}^{t}\int_{\RR^d}\int_{0}^{s}\int_{\RR^d} G_{t-s}(x-y) G_{s-r}(y-z) I_0(r,z) W(d r,d z)W(d s,d y)+\cdots\,, \nonumber
 \end{align}
where $I_n$ denotes the multiple It\^o-Wiener integrals and  $g_n(t,x)\ (n\geq1)$ are  defined by %\in \widetilde{\HH}^{^{\otimes n}}$ and defined by
 \begin{equation}\label{def.g_n}
	g_n(\vec{s},\vec{x};t,x)=\frac{1}{n!} G_{t-s_{\sigma(n)}}(x-x_{\sigma(n)})\cdots G_{s_{\sigma(2)}-s_{\sigma(1)}}(x_{\sigma(2)}-x_{\sigma(1)}) I_0(s_{\sigma(1)},x_{\sigma(1)})\,,
 \end{equation}
 where $\vec{x}=(x_1,\dots,x_n)$ and $\vec{s}=(s_1,\dots,s_n)$ such that $0<s_{\sigma(1)}<s_{\sigma(2)}<\cdots<s_{\sigma(n)}<t$ for a permutation $\sigma$. Then to verify the existence and uniqueness of the mild solution $v(t,x)$ is equivalent to show that
\begin{align}\label{Sec_v(t,x)}
	\EE[|v(t,x)|^2]=\sum_{n=0}^{\infty} n!\|g_n(\cdot;t,x) \|^2_{\HH^{\otimes n}}<\infty\,,
\end{align}
where $\HH$ is defined by \eqref{eq.def_H}.  In terms
of Fourier transformation, we have
\begin{align*}
  \|g_n(\cdot;t,x) \|^2_{\HH^{\otimes n}}= \int_{[0,t]^n}\int_{\RR^n}\lt| \cF g_n(\vec{s},\cdot;t,x)(\vec{\xi}) \rt|^2 \mu(d\vec{\xi} \,)d\vec{s}\,,
\end{align*}
with $\mu(d\vec{\xi} \,)=\prod_{j=1}^n |\xi_j|^{1-2H}d\vec{\xi}$.
% where $I_0(t,x)$ is defined in \eqref{def.I_0} and
% \begin{equation*}
% 	I_n(t,x)=\int_{0}^{t}\int_{\RR^d} G_{t-s}(x, y) I_{n-1}(s,y) W(d s,d y)\,,\qquad n\geq 1\,.
% \end{equation*}
% By orthogonality, we want to show the following theorem. \tcb{As a consequence, we must have $H>\frac{1}{4}$.}
%\begin{thm}\label{Thm_3}
%Let the initial condition $I_0(t,x)\in\cZ^p(0)$. If the SPDE \eqref{eq.HAE} has a local solution in $L^2(\Omega)$, then we must have $H>1/4$.
%In other words, when $H\leq \frac 14$,  the second moment of $v(t,x)$ satisfies
%\begin{align}\label{secV}
%	\|v(t,x)\|_{L^2(\Omega)}^2
%	=1+\sum_{n=1}^{\infty} \|I_n(t,x)\|_{L^2(\Omega)}^2\geq  \infty \,.
%%	=&1+\sum_{n=1}^{\infty} n! \|\tilde{f}_n (\cdot;t,x)\|^2_{\cH^{\otimes n}}
%%	= 1+\sum_{n=1}^{\infty}\frac{1}{n!} \Phi_n (t)\,.
%\end{align}
%\end{thm}

%To verify Theorem \ref{Thm_3}, we introduce the useful sharp regularity in space at first. In particular, the lower bound is crucial to show the necessity of $H>\frac 14$.
For national simplicity, we abbreviate $I_k(g_k(t,x))$  as $I_k(t,x)$  for $k=1,2$, i.e.
\begin{align*}
	I_1(t,x)=&\int_0^t \int_{\RR} G_{t-s}(x-y)I_0(s,y)W(ds,dy)\,,\\
	I_2(t,x)=&\int_0^t \int_{\RR} G_{t-s}(x-y)I_1(s,y)W(ds,dy)\,.
\end{align*}
Let us select some special initial conditions $u_0(x)=e^{-x^2}$ and $v_0(x)\equiv 0$ to proceed our argument. Then
\begin{align}\label{Cond.u0v0}
	I_0(t,x)=& \frac 12\int_{x-t}^{x+t}v_0(y)dy+\frac{1}{2}[u_0(x+t)+u_0(x-t)]\nonumber \\
	=&\frac{1}{2}\lk e^{-(x+t)^2}+e^{-(x-t)^2} \rk \,.
\end{align}
We do not consider the simple case $u_0(x)=1$ and $v_0(x)=0$. Because in this case, $I_0(t,x)$ is not in the space $\cZ^p(T)$ for any $p\geq 1$.
\begin{lem}\label{lem.Reg_x.sharp}
	 Suppose $I_0(t,x)$ are given in \eqref{Cond.u0v0}.
%	 uniformly bounded and uniformly $\frac 12$-H\"older continuous functions, i.e.
%	\begin{align}\label{Cond.u0v0_1}
%	\begin{cases}
%		0 \leq u_0(x)\leq C\,,&\qquad |u_0(x)-u_0(y)|\leq C|x-y|^{1/2}\,; \\
%		0 \leq v_0(x)\leq C\,,&\qquad |v_0(x)-v_0(y)|\leq C|x-y|^{1/2}\,.
%	\end{cases}
%	\end{align}
%	In addition, assume $u_0$ (or $v_0$) satisfies
%	\begin{align}\label{Cond.u0v0_2}
%		u_0(x) \geq c>0\,,\quad (\text{or }v_0(x) \geq c>0\,,)\qquad \text{when }x\in[-2T,2T]\,.
%	\end{align}
	Then for $H\in(0,1/2)$, there exist positive constants $c_{T,H}$ and $C_{T,H}$ such that for any $(t,x)\in [0,T]\times\RR$  and $h$ small enough satisfying $0<h<1\wedge \frac{t}{2}$,
	\begin{equation}\label{Reg_x.sharp}
		c_{t,H} \cdot |h|^{2H}\leq \EE[|\fD_h I_1(t,x)|^2]\leq C_{T,H} \cdot |h|^{2H}\,.
	\end{equation}
\end{lem}
%\begin{rmk}\label{Rmk.u0v0}
%	We mention here that the initial data $u_0$ and $v_0$ exist. For example, we can take $v_0(x)=0$ and $u_0(x)=\int_{-2T}^{2T}\alpha_T(x-y)dy$ where $\alpha_T(x)=\frac{1}{T}\Psi(\frac{x}{T})$ and $\Psi(x)=e^{\frac{1}{|x|^2-1}}\1_{\{|x|<1\}}$ is the modifier function. So $u_0(x)$ is a continuous function with support on $[-3T,3T]$ and $u_0(x)\equiv c$ on $[-2T,2T]$.
%	
%	We do not consider the simple case $u_0(x)=1$ and $v_0(x)=0$. Because in this case, $I_0(t,x)$ is not in the space $\cZ^p(T)$ for any $p\geq 2$.
%\end{rmk}
\begin{proof}
	At first, from \eqref{Cond.u0v0}  we see easily  that
	\begin{equation}\label{Est.u0v0_u}
		|I_0(t,x)|\leq C_T\,,\quad \lt|\fD_l I_0(t,x) \rt| \leq C_T \cdot |l|\wedge 1\,.
	\end{equation}
%	We shall give upper bounds for $|I_0(t,x)|$ and $\lt|\fD_l I_0(t,x) \rt|$. This is
%	\begin{align}\label{Est.u0v0_1}
%		|I_0(t,x)|=&\frac 12\int_{x-t}^{x+t}v_0(y)dy+\frac{1}{2}[u_0(x+t)+u_0(x-t)] \nonumber \\
%		\leq & \frac 12 C\cdot t+\frac 12 C \leq C_T\,,
%	\end{align}
%	and
%	\begin{align}\label{Est.u0v0_2}
%		\lt|\fD_l I_0(t,x) \rt| \ls& \lt| \int_{x-t}^{x+t} \fD_l v_0(y) dy \rt| +\lt| \fD_l [u_0(x+t)+u_0(x-t)]\rt| \nonumber \\
%		\ls& (t+1)\cdot |l|^{1/2}\wedge 1 \leq C_T \cdot |l|^{1/2}\wedge 1\,.
%	\end{align}
	Moreover, on the set $(t,x)\in[0,T]\times[-T,T]$, we have a lower bound for  $|I_0(t,x)|$:
	\begin{equation}\label{Est.u0v0_l}
		 I_0(t,x) =\frac{1}{2}\lk e^{-(x+t)^2}+e^{-(x-t)^2} \rk\geq   c_T\,.
	\end{equation}

	Now we are in a position to estimate $\EE[|\fD_h I_1(t,x)|^2]$. Let us consider the lower bound first. Recall an elementary inequality: $(a+b)^2\geq \frac 34 a^2-3b^2$, then
	\begin{align*}
		\EE[|\fD_h I_1(t,x)|^2]
		&=\EE\lk\lt| \int_0^t\int_{\RR}\fD_h G_{t-s}(x-y) \cdot I_0(s,y)W(ds,dy) \rt|^2 \rk \\
		&=\int_0^t \int_{\RR^2}\Big|\fD_h G_{t-s}(x-(y+l)) \cdot I_0(s,y+l) \\
		&\qquad\quad-\fD_h G_{t-s}(x-y) \cdot I_0(s,y) \Big|^2 \cdot |l|^{2H-2}dldyds \\
%		&\simeq \int_0^t \int_{\RR^2} \lt|\1_{|y+l+h|<s}-\1_{|y+l|<s}-\1_{|y+h|<s}+\1_{|y|<s}\rt|^2 \cdot |l|^{2H-2}dldyds\\
		&\geq \frac 34 \int_0^t \int_{\RR^2} \lt|\Box_{l,h} G_{s}(y)\cdot I_0(s,y)   \rt|^2 \cdot |l|^{2H-2}dldyds \\
		&\qquad\quad-3\int_0^t \int_{\RR^2} \lt|\fD_h G_s(x-y) \rt|^2\cdot \lt|\fD_l I_0(s,y) \rt|^2 \cdot |l|^{2H-2}dldyds  \,.
%		&\geq \frac 34\int_0^t \int_{y>0} \int_{l\geq h} \lt|\Box_{l,h} G_{s}(y)I_0(s,y)   \rt|^2 \cdot |l|^{2H-2}dldyds\,.
	\end{align*}
	By H\"older's  inequality and \eqref{Est.u0v0_u}, we  see that $\sup\limits_{s\in[0,T],y\in\RR} \int_{\RR}\lt|\fD_l I_0(s,y) \rt|^2 \cdot |l|^{2H-2}dl\leq C_{T,H}<\infty$ for $H\in(0,\frac 12)$. Then
	\begin{align*}
		\int_0^t \int_{\RR^2} \lt|\fD_h G_s(x-y) \rt|^2\cdot &  \lt|\fD_l I_0(s,y) \rt|^2 \cdot |l|^{2H-2}dldyds \\
		 \,\ls & \, \int_0^t  \int_{\RR}\lt|\fD_h G_s(y) \rt|^2 dy  ds  \leq C_T\cdot |h|\,.
%		 \,\ls & \, \int_0^t \lc \int_{\RR}\lt|\fD_h G_s(y) \rt|^4dy \rc^{\frac 12} ds \tcr{\leq C_T\cdot |h|^{2}}\,,
	\end{align*}
%	\tcb{since $\sup_{t\in[0,T],y\in\RR} \int_{\RR}\lt|\fD_l I_0(s,y) \rt|^2 \cdot |l|^{2H-2}dl\leq C_T<\infty$ for $H\in(0,\frac 12)$ ($u_0(t,x)$ and $v_0(t,x)$ satisfying the conditions we assumed).}
	Moreover, we have
	\begin{align*}
		\int_0^t \int_{\RR^2}& \lt|\Box_{l,h} G_{s}(y)\cdot I_0(s,y)   \rt|^2 \cdot |l|^{2H-2}dldyds \\
		&\geq \int_0^t \int_{y>0} \int_{l\geq h} \lt|\Box_{l,h} G_{s}(y)I_0(s,y)   \rt|^2 \cdot |l|^{2H-2}dldyds\,.
	\end{align*}
	Notice that on the set $\{y>0\}\times\{l\geq h\}$
	\begin{align*}
		|\Box_{l,h} G_{s}(y)|^2 \simeq&\  |\1_{\{|y+l+h|<s\}}-\1_{\{|y+l|<s\}}-\1_{\{|y+h|<s\}}+\1_{\{|y|<s\}}|^2 \\
		\simeq&\ | \1_{\{y+l<|s|<y+l+h\}}-\1_{\{y<|s|<y+h\}} |^2\\
		=&\1_{\{y+l<|s|<y+l+h\}}+\1_{\{y<|s|<y+h\}}\,.
	\end{align*}
%	Thus, taking $u_0(x)=\1_{(-2T\leq x \leq 2T)}$ and $v_0=0$,
%	Since \tcb{$|I_0(t,x)|\geq c>0$ on $[0,T]\times [-2T,2T]$} (e.g. taking $u_0(x)=\1_{(-2T\leq x \leq 2T)}$ and $v_0=0$) and
	Letting $h<1\wedge \frac{t}{2}$ be  small enough and noticing the lower bound \eqref{Est.u0v0_l}, we have
	\begin{align*}
		\int_0^t \int_{\RR^2}& \lt|\Box_{l,h} G_{s}(y)I_0(s,y)   \rt|^2 \cdot |l|^{2H-2}dldyds \\
		&\gtrsim \int_{\frac t2}^t \int_{s-h}^s \int_{l\geq h}  |l|^{2H-2}|I_0(s,y)|^2 dldyds\\
		&\gtrsim \int_{\frac t2}^t \int_{s-h}^s   |h|^{2H-1} |I_0(s,y)|^2 dyds\gtrsim c_{t,H} \cdot h^{2H}\,.
	\end{align*}
	Thus, we obtain when $H\in(0,1/2)$ and $|h|$ is relatively small
	\begin{align*}
		\EE[|\fD_h I_1(t,x)|^2] \gtrsim c_{t,H} \cdot h^{2H}-C_T\cdot |h| \gtrsim c_{t,H} \cdot h^{2H}.
	\end{align*}

The upper bound can be derived by the Fourier transformation.	 By \eqref{Est.u0v0_u}, we have
\begin{align*}
	\EE[|\fD_h I_1(t,x)|^2]
		&\leq 2 \int_0^t \int_{\RR^2} \lt|\Box_{l,h} G_{s}(y)I_0(s,y)   \rt|^2 \cdot |l|^{2H-2}dldyds \\
		&\qquad+2\int_0^t \int_{\RR^2} \lt|\fD_h G_s(x-y) \rt|^2\cdot \lt|\fD_l I_0(s,y) \rt|^2 \cdot |l|^{2H-2}dldyds \\
    &\ls \int_0^t\int_{\RR^2}\lt|\Box_{l,h} G_{s}(y)\rt|^2  \cdot |l|^{2H-2}dldyds + \int_0^t \int_{\RR} \lt|\fD_h G_s(y) \rt|^2 dyds \\
	&\ls \int_0^t\int_{\RR} |e^{\iota h\xi}-1|^2 \lc\frac{\sin(s|\xi|)}{\xi} \rc^2 |\xi|^{1-2H}d\xi ds+ |h| \\
	&\ls \int_0^t\int_{\RR} [1-\cos(h|\xi|)]\frac{s^2}{1+s^2|\xi|^2}|\xi|^{1-2H}d\xi ds+ |h|  \\
	%&\ls \int_{\RR} [1-\cos(h|\xi|)]\cdot |\xi|^{-1-2H}d\xi + |h|   \\
	&= C_{T,H}\cdot (|h|^{2H} + |h|) \ls C_{T,H}\cdot |h|^{2H} \,,
\end{align*}
for $H\in(0,1/2)$ and $|h|$ is sufficiently  small. Therefore, we finish the proof.
% of Theorem \ref{Thm_3}.	
\end{proof}

Now we begin to prove Theorem \ref{Thm_3}.
\begin{proof}[Proof of Theorem \ref{Thm_3}]
We only need to consider $\|I_2(t,x)\|_{L^2(\Omega)}^2$ with  some special initial data  \eqref{Cond.u0v0}.  Let us denote
$$F_{t,x}(s,y):=G_{t-s}(x-y)I_1(s,y)\,. $$
Noting that
\begin{align*}
	|\fD_h F_{t,x}(s,y)|^2=&\ |G_{t-s}(x-y)\fD_h I_1(s,y)+\fD_h G_{t-s}(x-y)I_1(s,y)|^2 \\
%	\geq &\ |G_{t-s}(x-y)\fD_h I_1(s,y)|^2+|\fD_h G_{t-s}(x-y)I_1(s,y)|^2 \\
%	&-2|G_{t-s}(x-y)\fD_h I_1(s,y)|\cdot |\fD_h G_{t-s}(x-y)I_1(s,y)| \\
	\geq &\ \frac{3}{4}|G_{t-s}(x-y)\fD_h I_1(s,y)|^2-3|\fD_h G_{t-s}(x-y)I_1(s,y)|^2\,,
\end{align*}
so we have
\begin{align}
	\EE\lk|I_2(t,x)|^2 \rk=&\ \EE \int_{0}^{t} \int_{\RR^2} |\fD_h F_{t,x}(s,y)|^2 |h|^{2H-2}dhdyds \nonumber \\
	\geq&\ \frac{3}{4}\int_{0}^{t} \int_{\RR^2} |G_{t-s}(x-y)|^2 \EE|\fD_h I_1(s,y)|^2 |h|^{2H-2}dhdyds \label{eq.I2_1} \\
	&-3\int_{0}^{t} \int_{\RR^2} |\fD_h G_{t-s}(x-y)|^2 \EE|I_1(s,y)|^2 |h|^{2H-2}dhdyds \label{eq.I2_2}\,.
%	&-2\int_{0}^{t} \int_{\RR^2} |G_{t-s}(x-y) \fD_h G_{t-s}(x-y)| \nonumber \\
%	&\qquad\qquad\qquad\ \cdot \EE\lk|I_1(s,y) \fD_h I_1(s,y)| \rk |h|^{2H-2}dhdyds\,. \label{eq.I2_3}
\end{align}
 Without loss of generality, we assume $t=2$  and estimate term \eqref{eq.I2_1} first. By Lemma \ref{lem.Reg_x.sharp} with $h<1\wedge \frac t2=1$, it is clear that when $H\leq \frac 14$,
\begin{align}\label{est_first}
\int_{0}^{t} \int_{\RR^2}& |G_{t-s}(x-y)|^2 \EE|\fD_h I_1(s,y)|^2 |h|^{2H-2}dhdyds\nonumber\\
	\gtrsim& \int_{1}^{2} \int_{\RR}\int_{|h|<1} |G_{2-s}(x-y)|^2 \EE|\fD_h I_1(s,y)|^2 |h|^{2H-2}dhdyds \nonumber\\
	\gtrsim& \int_{1}^{2} \int_{\RR}\int_{|h|<1} |G_{2-s}(x-y)|^2 \cdot|h|^{4H-2}dhdyds=\infty.
\end{align}
%For the term \eqref{eq.I2_2}, still taking $u_0(x)$ and $v_0(x)$ as in\eqref{Cond.u0v0_1} and \eqref{Cond.u0v0_2}.
For any $H\in(0,1/2)$ we can get  that  $\sup_{y\in\RR}\EE|I_1(s,y)|^2\ls s^{2H}+s^2$ in the term \eqref{eq.I2_2}, thus
\begin{align}\label{est_second}
	\int_{0}^{2}& \int_{\RR^2} |\fD_h G_{2-s}(x-y)|^2 \EE|I_1(s,y)|^2 |h|^{2H-2}dhdyds \nonumber\\
%&=\int_{0}^{2} \int_{\RR^2} |\fD_h G_{2-s}(x-y)|^2 |h|^{2H-2}\times\nonumber\\ &\ \ \ \ \EE\lt[\int_0^s\int_{\RR^2}\lt|G_{s-r}(y-z+l)I_0(r,z+l)-G_{s-r}(y-z)I_0(r,z)\rt|^2|l|^{2H-2}dzdldr\rt]dhdyds \nonumber\\
	&\ls \int_{0}^{2} (s^{2H}+s^2) \int_{\RR}\lt(\frac{\sin((2-s)|\xi|)}{|\xi|} \rt)^2 |\xi|^{1-2H}d\xi ds\nonumber \\
	&= \int_{0}^{2} (s^{2H}+s^2)(2-s)^{2H} ds <\infty.
\end{align}
 Plugging \eqref{est_first} and \eqref{est_second} into  \eqref{eq.I2_1} and\eqref{eq.I2_2}, we obtain that for $t=2$
\begin{equation*}
	\EE\lk|I_2(t,x)|^2\rk=\infty
\end{equation*}
when $H\leq \frac 14$. The proof is complete.
\end{proof}

\appendix
\section{Some technical lemmas for  wave kernel}\label{s.6}
In this Appendix, we show some technical  lemmas used several times in our work. Let us start by proving the Fourier transform of $\cE(t,x)$, $\cS_{\alpha}(t,x)$ and $\cC_{1-\alpha}(t,x)$.

\begin{lem}\label{Lem_A.1}
  Let $\cE(t,x)$, $\cS_{\alpha}(t,x)$ and $\cC_{1-\alpha}(t,x)$ be defined by  \eqref{eq.cC_alpha}. Then  they are all in $L^1(\RR)$, and their Fourier transforms  are given  by  \eqref{eq.Fourier}. Consequently, the wave kernel $G_{t-s}(x-y)$ can be expressed as the representation $\eqref{eq.SumKer}$.
\end{lem}
\begin{proof}
  We treat $\cE(t,x)$ at first,
\begin{align}\label{F.cE}
  \cE(t,x)&=\cF^{-1}[\hat{\cE}(t,\cdot)](x)=\frac{1}{2\pi}\int_{\RR}e^{-t|\xi|}e^{\iota x\xi} d\xi=\frac{1}{\pi}\frac{t}{t^2+x^2}\,,
\end{align}
which is obviously in $L^1(\RR)$. Similarly, for $\cS_{\alpha}(t,x)$,
\begin{align}\label{F.cS}
  \cS_{\alpha}(t,x)&=\cF^{-1}[\hat{\cS}_{\alpha}(t,\cdot)](x)=\frac{1}{2\pi}\int_{\RR}\frac{\sin(t|\xi|)}{|\xi|^{\alpha}}e^{\iota x\xi} d\xi=\frac{1}{\pi}\int_{0}^{\infty}\frac{\sin(t\xi)}{\xi^{\alpha}}\cos(|x|\xi) d\xi \nonumber\\
    &=\frac{1}{2\pi}\int_{0}^{\infty}\frac{\sin\blc(t+|x|)\xi\brc}{\xi^{\alpha}}d\xi+\frac{1}{2\pi}\int_{0}^{\infty}\frac{\sin\blc(t-|x|)\xi\brc}{\xi^{\alpha}}d\xi \nonumber\\
    &=\frac{\Gamma(1-\alpha)}{2\pi}\cos\lc\frac{\alpha\pi}{2}\rc \lk \blc t+|x|\brc^{\alpha-1}+\sgn(t-|x|)\blc t-|x|\brc^{\alpha-1} \rk\,,
\end{align}
where the last equality can be found in 17.33(2) in \cite{GR2014}. For fixed $t>0$, if $|x|$ is close to $t$, $|\cS_{\alpha}(t,x)|$ can be bounded by
\[
 \blc t+|x|\brc^{\alpha-1}+\big | t-|x|\big |^{\alpha-1}\,.
\]
And when $|x|$ is large enough, $|\cS_{\alpha}(t,x)|$ behaves   like
%\[
% t\big( |x|+t\big)^{\alpha-2}\ls \big( |x|-t\big)^{\alpha-1}-\blc |x|+t\brc^{\alpha-1}\ls  t\big( |x|-t\big)^{\alpha-2} \,.
%\]
\[\big( |x|-t\big)^{\alpha-1}-\blc |x|+t\brc^{\alpha-1},\]
which can be bounded by $t\big( |x|-t\big)^{\alpha-2}$.   Therefore, $\cS_{\alpha}(t,x)$ is in $L^1(\RR)$ since $\alpha\in(0,1)$.

The last one $\cC_{1-\alpha}(t,x)$ is more involved because of the term $\cF^{-1}\lk\frac{e^{-t|\xi|}}{|\xi|^{1-\alpha}}\rk$. But we can apply the formula 17.34(14) in \cite{GR2014} to get
\begin{align}\label{F.cC}
  \cC_{1-\alpha}(t,x)&=\cF^{-1}[\hat{\cC}_{1-\alpha}(t,\cdot)](x)
%  =\frac{1}{2\pi}\int_{\RR}\frac{\cos(t|\xi|)-e^{-t|\xi|}}{|\xi|^{1-\alpha}}\cdot e^{\iota x\xi} d\xi
  \nonumber\\
    &=\frac{1}{\pi}\int_{0}^{\infty}\frac{\cos(t\xi)}{\xi^{1-\alpha}}\cos(|x|\xi) d\xi-\frac{1}{\pi}\int_{0}^{\infty}\frac{e^{-t|\xi|}}{\xi^{1-\alpha}}\cos(|x|\xi) d\xi \nonumber\\
%    &=\frac{\Gamma(\alpha)}{2\pi}\bigg[\cos\lc\frac{\alpha\pi}{2}\rc \lk\big|t+|x|\big|^{-\alpha}+\big|t-|x|\big|^{-\alpha}\rk \nonumber \\
%    &\qquad\qquad\qquad-\lk(t-i|x|)^{-\alpha}+(t+i|x|)^{-\alpha} \rk \bigg] \nonumber\\
    &=\frac{\Gamma(\alpha)}{2\pi}\bigg[\cos\lc\frac{\alpha\pi}{2}\rc \lk\big|t+|x|\big|^{-\alpha}+\big|t-|x|\big|^{-\alpha}\rk \nonumber\\
    &\qquad\qquad\qquad-2\cos\lc\alpha\tan^{-1}\lc\frac{|x|}{t}\rc\rc \blk t^2+x^2\brk^{-\frac{\alpha}{2}}\bigg].\
\end{align}
Similarly, when $|x|$ is close to $t$, $|\cC_{1-\alpha}(t,x)|$ can be bounded by
\[
 \big|t+|x|\big|^{-\alpha}+\big|t-|x|\big|^{-\alpha}+\blk t^2+x^2\brk^{-\frac{\alpha}{2}}\,.
\]
It is more interesting to know the above asymptotics when    $|x|$ is large. Since
\begin{align}\label{A.cC_alpha}
\cC_{1-\alpha}(t,x)\simeq& \cos\lc\frac{\alpha\pi}{2}\rc
    \lk|t+|x||^{-\alpha}+|t-|x||^{-\alpha} \rk-2\cos\lc\frac{\alpha\pi}{2}\rc\lk t^2+x^2\rk^{-\frac{\alpha}{2}}\nonumber\\
    &+2 \lk \cos\lc\frac{\alpha\pi}{2}\rc- \cos\lc\alpha\tan^{-1}\lc\frac{|x|}{t}\rc\rc\rk (t^2+x^2)^{-\frac{\alpha}{2}},\
\end{align}
setting $y_0=\frac{|x|}{t}$, then for $|x|$ large enough,
\begin{align}\label{appenII}
 \cos\lc\frac{\alpha\pi}{2}\rc- \cos \lc\alpha\tan^{-1}\lc\frac{|x|}{t}\rc\rc=&\int_{y_0}^{+\infty}\frac{d}{d\omega} \lt[ \cos\lc\alpha\tan^{-1}\lc\omega\rc\rc\rt] d\omega\nonumber\\
%&\le  \alpha\left|\int_{y_0}^{+\infty}\frac{ 1}{1+\omega^2}d\omega\right|
 \le &\alpha\int_{y_0}^{+\infty}\frac{ 1}{\omega^2}d\omega\simeq  C_{\alpha}\cdot t|x|^{-1}.
\end{align}
Therefore,
\begin{equation}\label{appendixII}
\lk 2\cos\lc\frac{\alpha\pi}{2}\rc-2\cos\lc\alpha\tan^{-1}\lc\frac{|x|}{t}\rc\rc\rk (t^2+x^2)^{-\frac{\alpha}{2}}\simeq C_{\alpha}\cdot t|x|^{-1}(t^2+x^2)^{-\frac{\alpha}{2}},
\end{equation}
which is integrable with respect to  $x$ when $|x|$ is large enough since $\alpha\in(0,1).$ \ Moreover, since the following important asymptotic behavior holds, which will be explained in Remark \ref{explation},
\begin{align}\label{eq.Cos[arctan]}
	\big|t+|x|\big|^{-\alpha}&+\big|t-|x|\big|^{-\alpha} =\,2\blc |x|^2-t^2\brc^{-\frac{\alpha}{2}} \cos\lk\alpha \tan^{-1}\lc\frac{\iota t}{|x|} \rc\rk\sim 2\blc |x|^2-t^2\brc^{-\frac{\alpha}{2}}\,,
\end{align}
it is clear that
\[
 t\blc |x|^2+t^2\brc^{-\frac{\alpha}{2}-1}\ls\, \blc |x|^2-t^2\brc^{-\frac{\alpha}{2}}-\blc |x|^2+t^2\brc^{-\frac{\alpha}{2}}\ls\, t\blc |x|^2-t^2\brc^{-\frac{\alpha}{2}-1}\,.
\]
We see that for $\alpha\in(0,1)$, $\cC_{\alpha}(t,x)$ is also integrable with respect to $x$ when $|x|$ sufficiently  large. As a result, $\cC_{\alpha}(t,x)$ is in $L^1(\RR)$.

Combining \eqref{F.cE}, \eqref{F.cS} and \eqref{F.cC}, we can conclude  \eqref{eq.SumKer}.
\end{proof}

\begin{rmk}\label{explation}
We provide  details of the equation \eqref{eq.Cos[arctan]} we used  in the above proof of Lemma \ref{Lem_A.1}. Noticing that
\begin{equation*}
\arctan(z)=-\frac{\iota}{2}\ln\left(\frac{\iota- z}{\iota+ z}\right)=-\frac{\iota}{2}\ln\left(\frac{1+\iota z}{1-\iota z} \right)\,,
\end{equation*}
  we have
\begin{eqnarray*}
\cos[\alpha\tan^{-1}(z)]&=&\frac12\left\{\exp\lt[\iota \alpha\tan^{-1}(z)\rt]+\exp\lt[-\iota \alpha\tan^{-1}(z)\rt]\right\}\\
&=&\frac12\left\{\exp\lt[\iota \alpha \cdot (-\frac{\iota}{2})\ln\left(\frac{1+\iota z}{1-\iota z}\right)\rt]+\exp\lt[-\iota \alpha\cdot(-\frac{\iota}{2})\ln\left(\frac{1+\iota z}{1-\iota z}\right)\rt]\right\}\\
&=&\frac12\left\{\exp\lt[\frac{\alpha}{2}\ln\left(\frac{1+\iota z}{1-\iota z}\right)\rt]+\exp\lt[-\frac{\alpha}{2}\ln\left(\frac{1+\iota z}{1-\iota z}\right)\rt]\right\}\\
&=&\frac12\left\{ \left(\frac{1+\iota z}{1-\iota z}\right)^{\frac{\alpha}{2}}+\left(\frac{1-\iota z}{1+\iota z}\right)^{\frac{\alpha}{2}} \right\}\\
&=&\frac12\left\{(1+z^2)^{\frac{\alpha}{2}}\lt[(1-\iota z)^{-\alpha}+(1+\iota z)^{-\alpha}\rt]\right\}.\\
\end{eqnarray*}
Letting $\ z=\frac{\iota t}{|x|}$, we see the equation \eqref{eq.Cos[arctan]} holds.
\end{rmk}

\begin{lem}\label{Lem_A.2}
If $\frac12< \alpha<  1 $, then $\hat{\cC}_{ \alpha}(t,\xi):=\frac{\cos(t|\xi|)-e^{-t|\xi|}}{|\xi|^{ \alpha}}$   and $\hat{\cS}_{\alpha}(t,\xi):=\frac{\sin(t|\xi|)}{|\xi|^{\alpha}}$  are  in $L^2(\RR)$ for any  $t>0$.
Hence, $\cC_{ \alpha}(t,x)$
%is also in $L^2(\RR)$. If $\frac 12< \alpha\leq 1$, then $\hat{\cS}_{\alpha}(t,\xi):=\frac{\sin(t|\xi|)}{|\xi|^{\alpha}}$ is in $L^2(\RR)$ for fixed $t>0$. Hence,
and $\cS_{\alpha}(t,x)$ are also in $L^2(\RR)$.
\end{lem}
\section{Lemmas for Proposition \ref{prop_est}}\label{Lemma for 3.3}
\begin{lem}\label{Lem.Est_J}
	If $p>\frac 1H$, $1-H<\alpha<1-\frac 1p$ and $1-2/q+\alpha<\theta<H+\alpha-1/2$ , then there exists a constant $C$ independent of $r$ such that %the fol%\eqref{Est.CharI2} holds, i.e.
  \begin{equation}\label{Est.CharI2}
    \EE\|J^{\cK_{i}}_{\theta}(r,\cdot)\|_{L^p(\RR)}^p \leq C\|v\|_{\cZ^p(T)}^p\,, \quad  i=1,2,3,4\,,
  \end{equation}
  where $J^{\cK_{i}}_{\theta}$ (depending on $\alpha,\theta$) and $\cK_{i}$ (depending on $\alpha$) are defined by  \eqref{def.JK_theta}  and \eqref{def.cK} respectively.
\end{lem}
\begin{proof}
	We will prove the above bound \eqref{Est.CharI2}   for $i=1, 2, 3, 4$ separately. We deal with the term $i=1$ first. In this case  $\cK_{1}=\cC_{\alpha}$ and $\bar{\cK}_{1}=\cS_{1-\alpha}$ as defined by  \eqref{eq.cC_alpha}. From the definition \eqref{def.JK_theta} of $J^{\cK_{1}}_{\theta}$ and  from   Burkholder-Davis-Gundy's  inequality and the  triangle inequality it follows
  \begin{equation*}
    \int_{\RR}\EE \big|J^{\cC_{\alpha}}_{\theta}(r,z) \big|^p dz\lesssim \int_{\RR} \big| \cD_1(r,z)\big|^{\frac p2} +\big| \cD_2(r,z)\big|^{\frac p2} dz\,,
  \end{equation*}
  where we have used two notations
  \begin{align}\label{eq.D1_def}
     \cD_1(r,z):= &\int_{0}^{r}\int_{\RR^2}(r-s)^{-2\theta} \lt|\fD_h\cC_{\alpha}(r-s,y)\rt|^2 \nonumber \\
      &\qquad\   \cdot \|v(s,y+z)\|_{L^p(\Omega)}^2|h|^{2H-2}dhdyds\,,
  \end{align}
  and
  \begin{align}\label{eq.D2_def}
      \cD_2(r,z) &:=\int_{0}^{r}\int_{\RR^2}(r-s)^{-2\theta}|\cC_{\alpha}(r-s,y)|^2 \nonumber \\
      &\qquad\qquad \cdot\|\fD_h v(s,z+y)\|_{L^p(\Omega)}^2 |h|^{2H-2} dhdyds\,.
  \end{align}

  By the definition of $\cZ_1^p(T)$ in \eqref{def.ZNorm}, we can bound $D_1(r):=\int_{\RR}  \lt| \cD_1(r,z)\rt|^{\frac p2} dz$ as follows.
  \begin{align}\label{Est.Char1D1}
    D_1(r)
    \lesssim&\,  \bigg(\int_{0}^{r}\int_{\RR^2}(r-s)^{-2\theta} \lt|\fD_h\cC_{\alpha}(r-s,y)\rt|^2
                |h|^{2H-2}dhdyds\bigg)^{\frac p2}\times \|v\|_{\cZ_{1}^p(T)}^p \nonumber\\
              \simeq&\, \bigg(\int_{0}^{r}\int_{\RR^2}s^{-2\theta} \lt|\fD_h\hat{\cC}_{\alpha}(s,\xi)\rt|^2 |h|^{2H-2}dhd\xi ds\bigg)^{\frac p2}\times \|v\|_{\cZ_{1}^p(T)}^p\,.
  \end{align}
  In the last line of \eqref{Est.Char1D1}, we have applied Parseval's formula which is legitimate  since $\cC_{\alpha}(s,\cdot)$ is in $L^2(\RR)$ when
  \begin{equation}\label{Ineq.alpha_I2}
    \frac 12<\alpha\leq 1\,,
  \end{equation}
  by Lemma \ref{Lem_A.2}.
  % in the Appendix \ref{s.6}.
  Through \eqref{eq.Fourier}, we can write \eqref{Est.Char1D1} as
  \begin{align}\label{Est.Char1D1F}
    D_1(r)
    \lesssim& \bigg(\int_{0}^{r}\int_{\RR^2}s^{-2\theta} \lt|\frac{\cos(s|\xi|)-e^{-s|\xi|}}{|\xi|^{\alpha}} \rt|^2[1-\cos(h|\xi|)] |h|^{2H-2}dhd\xi ds\bigg)^{\frac p2}\times \|v\|_{\cZ_{1}^p(T)}^p \nonumber\\
    \simeq& \bigg(\int_{0}^{r}\int_{\RR}s^{-2\theta} \lt|\cos(s|\xi|)-e^{-s|\xi|}\rt|^2|\xi|^{1-2\alpha-2H} d\xi ds\bigg)^{\frac p2}\times \|v\|_{\cZ_{1}^p(T)}^p \nonumber\\
    \simeq& \bigg(\int_{0}^{r}s^{2(H+\alpha-\theta-1)}ds\cdot \int_{0}^{\infty}\xi^{1-2\alpha-2H}\lt|\cos(\xi)-e^{-\xi}\rt|^2 d\xi\bigg)^{\frac p2}\times \|v\|_{\cZ_{1}^p(T)}^p\,,
  \end{align}
  which is finite if
  \begin{equation}\label{Ineq.alpha_I3}
    1-2\alpha-2H<-1\,,~2(H+\alpha-\theta-1)>-1\, \Leftrightarrow\, \alpha>1-H\,,~\theta<H+\alpha-\frac 12\,.
  \end{equation}

  Similarly, by the definition of $\cZ_2^p(T)$ in \eqref{def.ZNorm},   for $D_2(r):=\int_{\RR}  \left| \cD_2(r,z)\rt|^{\frac p2} dz$, Parseval's formula implies
  \begin{align}\label{Est.Char1D2}
     D_2(r)
     \lesssim& \bigg( \int_{0}^{r}\int_{\RR}(r-s)^{-2\theta}|\cC_{\alpha}(r-s,y)|^2 dyds\bigg)^{\frac p2}\times \|v\|_{\cZ_{2}^p(T)}^p \nonumber\\
     \simeq& \bigg( \int_{0}^{r}\int_{\RR}s^{-2\theta}|\hat{\cC}_{\alpha}(s,\xi)|^2 d\xi ds\bigg)^{\frac p2}\times \|v\|_{\cZ_{2}^p(T)}^p\,,
  \end{align}
  if $\alpha$ satisfies \eqref{Ineq.alpha_I2}. Then plugging \eqref{eq.Fourier}, we have
  \begin{align}\label{Est.Char1D2F}
    D_2(r)
    \lesssim& \bigg(\int_{0}^{r}\int_{\RR}s^{-2\theta} \lt|\frac{\cos(s|\xi|)-e^{-s|\xi|}}{|\xi|^{\alpha}} \rt|^2 d\xi ds\bigg)^{\frac p2}\times \|v\|_{\cZ_{2}^p(T)}^p \nonumber\\
    \simeq& \bigg(\int_{0}^{r}s^{2(\alpha-\theta)-1}ds\cdot \int_{0}^{\infty} \xi^{-2\alpha} \lt|\cos(s\xi)-e^{-s\xi}\rt|^2 d\xi\bigg)^{\frac p2}\times \|v\|_{\cZ_{2}^p(T)}^p\,,
  \end{align}
  which is finite since $\frac 12<\alpha\leq1$ and $\alpha>\frac 12-H+\theta>\theta$ by \eqref{Ineq.alpha_I3}.

  Thus, with the choice of $\theta<H+\alpha-1/2$ and $\alpha>1-H$,   we have finished the proof \eqref{Est.CharI2} for $i=1$.
%  we see (\eqref{Ineq.alpha_I1}, \eqref{Ineq.theta_I1},) \eqref{Ineq.alpha_I2}, and \eqref{Ineq.alpha_I3} are satisfied by noticing that $\frac 1H>\frac{4}{2H+1}$ if $H<\frac 12$. \tcr{Therefore we have finished the proof \eqref{Est.Char1}.}

Now let us deal with the case when $i=2$. Similar to the proof in the case  $i=1$, now we need to show
 \begin{equation*}
    \|J^{\cS_{\alpha}}_{\theta}(r,z)\|_{L^p(\Omega\times\RR)}^p \leq C\|v\|_{\cZ^p(T)}^p\,.
  \end{equation*}
 From the definition \eqref{def.JK_theta} of $J_\theta$ and  from  Burkholder-Davis-Gundy's inequality it follows
\begin{equation*}
    \int_{\RR}\EE \big|J^{\cS_{\alpha}}_{\theta}(r,z)\big|^p dz\lesssim \int_{\RR} \blk \widetilde{\cD}_1(r,z)\brk^{\frac p2} +\blk \widetilde{\cD}_2(r,z)\brk^{\frac p2} dz \,,
  \end{equation*}
  where   $\widetilde{\cD}_1(r,z)$ and $\widetilde{\cD}_2(r,z)$ are  defined by   \eqref{eq.D1_def} and \eqref{eq.D2_def}, respectively, with $\cC_\alpha$ replaced by $\cS_\alpha$.
%  \begin{align*}
%      \widetilde{\cD}_1(r,z) &:=\int_{0}^{r}\int_{\RR^2}(r-s)^{-2\theta} \lt|\fD_h\cS_{\alpha}(r-s,y)\rt|^2
%               \times \|v(s,y+z)\|_{L^p(\Omega)}^2|h|^{2H-2}dhdyds\,,
%  \end{align*}
%  and
%  \begin{align*}
%      \widetilde{\cD}_2(r,z) &:=\int_{0}^{r}\int_{\RR^2}(r-s)^{-2\theta}|\cS_{\alpha}(r-s,y)|^2 \times\|\fD_h v(s,z+y)\|_{L^p(\Omega)}^2 |h|^{2H-2} dhdyds\,.
%  \end{align*}

  By the definition of $\cZ_1^p(T)$ in \eqref{def.ZNorm} and Minkowski's
   inequality, we have
  \begin{align}\label{Est.Char1TilD1}
    \widetilde{D}_1(r):=& \int_{\RR} \left| \widetilde{\cD}_1(r,z)\right|^{\frac p2} dz\nonumber\\
    \lesssim&  \lt(\int_{0}^{r}\int_{\RR^2}(r-s)^{-2\theta} \big|\fD_h\cS_{\alpha}(r-s,y)\big|^2|h|^{2H-2}dhdyds\rt)^{\frac p2}
     %\nonumber\\
%    &\qquad\qquad\qquad\qquad\quad
	\times\|v\|_{\cZ_{1}^p(T)}^p \nonumber\\
%    \simeq& \bigg(\int_{0}^{r}\int_{\RR^2}s^{-2\theta} \big|\hat{\cS}_{\alpha}(s,\xi)-\hat{\cS}_{\alpha}(s,\xi+h)\big|^2 |h|^{2H-2}dhd\xi ds\bigg)^{\frac p2} \|v\|_{\cZ_{1}^p(T)}^p\nonumber\\
%     =&\bigg(\int_{0}^{r}\int_{\RR^2}s^{-2\theta} \big|\hat{\cS}_{\alpha}(s,\xi)\big|^2 (e^{i\xi h}-1)^2 |h|^{2H-2}dhd\xi ds\bigg)^{\frac p2} \|v\|_{\cZ_{1}^p(T)}^p\nonumber\\
%    \lesssim& \bigg(\int_{0}^{r}\int_{\RR^2}s^{-2\theta} \lt|\frac{\sin(s|\xi|)}{|\xi|^{\alpha}} \rt|^2[1-\cos(h|\xi|)] |h|^{2H-2}dhd\xi ds\bigg)^{\frac p2} \|v\|_{\cZ_{1}^p(T)}^p \nonumber\\
%    \simeq& \bigg(\int_{0}^{r}\int_{\RR}s^{-2\theta} \lt|\sin(s|\xi|)\rt|^2|\xi|^{1-2\alpha-2H} d\xi ds\bigg)^{\frac p2} \|v\|_{\cZ_{1}^p(T)}^p \nonumber\\
    \ls& \bigg(\int_{0}^{r}s^{2(H+\alpha-\theta-1)}ds\cdot \int_{0}^{\infty}\xi^{1-2\alpha-2H}\lt|\sin(\xi)\rt|^2 d\xi\bigg)^{\frac p2} \times\|v\|_{\cZ_{1}^p(T)}^p\,,
  \end{align}
which is finite under the condition \eqref{Ineq.alpha_I3}.
%Similar to the estimation of $\widetilde{D}_2(r)=\int_{\RR}\widetilde{\cD}_2(r,z)dz$,
In a similar way we can get
\begin{align}\label{Est.Char1TilD2}
    \widetilde{D}_2(r):=&   \int_{\RR}\left|\widetilde{\cD}_2(r,z)\right|^{\frac p2} dz\nonumber\\
      \lesssim&\bigg( \int_{0}^{r}\int_{\RR}(r-s)^{-2\theta}|\cS_{\alpha}(r-s,y)|^2 dyds\bigg)^{\frac p2}\times \|v\|_{\cZ_{2}^p(T)}^p \nonumber\\
%      \simeq& \bigg( \int_{0}^{r}\int_{\RR}s^{-2\theta}|\hat{\cS}_{\alpha}(s,\xi)|^2 d\xi ds\bigg)^{\frac p2}\times \|v\|_{\cZ_{2}^p(T)}^p\nonumber\\
     \ls& \bigg(\int_{0}^{r}s^{2(\alpha-\theta)-1}dr\cdot \int_{0}^{\infty} \xi^{-2\alpha} \lt|\sin(\xi)\rt|^2 d\xi\bigg)^{\frac p2}\times \|v\|_{\cZ_{2}^p(T)}^p\,,
\end{align}
which is clearly bounded   by \eqref{Ineq.alpha_I3}
since $\frac12<\alpha< 1$ and $\alpha>\theta$.

Therefore, with the choice of $\theta\in(1-2/q+\alpha,H+\alpha-1/2)$, we finish the proof of \eqref{Est.CharI2} when $i=2$.
  The remaining parts  of \eqref{Est.CharI2}, i.e.
  the cases $\cK_{3}=\cS$ and   $\cK_{4}=\cE$ can be completed in the same spirit and we omit the details since they are actually  simpler.
\end{proof}

\begin{lem}\label{Lem.Est_D}
	If $p>\frac 1H$, $\frac 32-2H<\alpha<1-\frac 1p$ and  $1-2/q+\alpha<\theta<2H+\alpha-1$ , then there exists a constant $C$ independent of $r\in[0,T]$ such that for $   i=1,2,3,4$
  \begin{equation}\label{Est.Char2_lem}
   \EE\int_{\RR} \Blk\int_{\RR}\lt|J^{\cK_{i}}_{\theta}(r,z+h)-J^{\cK_{i}}_{\theta}(r,z)\rt|^2|h|^{2H-2}dh\Brk^{\frac p2} dz \leq C  \|v\|^p_{\cZ^p(T)}\,,
  \end{equation}
  where $J^{\cK_{i}}_{\theta}$ (depending on $\alpha,\theta$) and $\cK_{i}$ (depending on $\alpha$) are defined by  \eqref{def.JK_theta}  and \eqref{def.cK} respectively.
\end{lem}
\begin{proof}
Recall that $\fD_hJ^{\cK_{i}}_{\theta}(r,z):=J^{\cK_{i}}_{\theta}(r,z+h)-J^{\cK_{i}}_{\theta}(r,z)$. We still  first consider the case when $i=1$, i.e. $\cK_{1}=\cC_{\alpha}$ and $\bar{\cK}_{1}=\cS_{1-\alpha}$ defined by  \eqref{eq.cC_alpha}. We only need to prove that there exists some constant $C $, independent of $r\in [0, T]$,  such that
  \begin{equation}\label{Est.Char2}
  \begin{split}
  	\int_{\RR} \EE& \lk\int_{\RR}\lt|\fD_hJ^{\cC_{\alpha}}_{\theta}(r,z)\rt|^2 |h|^{2H-2}dh\rk^{\frac p2}dz \\
  	&\leq \lc \int_{\RR} \|\fD_h J^{\cC_{\alpha}}_{\theta}(r,z) \|^2_{L^p(\RR\times\Omega)} |h|^{2H-2}dh\rc^{\frac p2} \leq C  \|v\|^p_{\cZ^p(T)}\,,
  \end{split}
  \end{equation}
  where we  employed Minkowski's inequality in the above first inequality.
%  because $\int_{\RR}\lambda(z)dz=1$.

  Thanks to Burkholder-Davis-Gundy's  inequality, the triangle inequality and then a change of variable $y\to z-y$, we have
  \begin{align*}
%     \EE\blk|\fD_h u^{n+1}_\ep(t,x)|^p\brk =&\EE\blk|u^{n+1}_\ep(t,x)-u^{n+1}_\ep(t,x+h)|^p\brk \\
     \EE\blk|\fD_h &J^{\cC_{\alpha}}_{\theta}(r,z)|^p\brk \\
     \leq& C_p \bigg( \int_{0}^{r}(r-s)^{-2\theta}\int_{\RR^2} \bigg[ \EE\Big|\fD_h \cC_{\alpha}(r-s,z-y-l) v(s,y+l) \\
       &\qquad\qquad\  -\fD_h \cC_{\alpha}(r-s,z-y) v(s,y)\Big|^p\bigg]^{\frac{2}{p}} |l|^{2H-2} dl dyds\bigg)^{\frac p2} \\
     \leq& C_p \bigg( \int_{0}^{r}(r-s)^{-2\theta}\int_{\RR^2} |\fD_h\cC_{\alpha}(r-s,y)|^2 \|v(s,y+z)\|^2_{L^p(\Omega)} |l|^{2H-2} dl dyds\bigg)^{\frac p2} \\
     &+ C_p \bigg( \int_{0}^{r}(r-s)^{-2\theta}\int_{\RR^2} \Big|\Box_{h,l}\cC_\alpha(r-s,y)\Big|^2 \|\fD_l v(s,y+z)\|^2_{L^p(\Omega)}  |l|^{2H-2} dl dyds\bigg)^{\frac p2}\,.
%     =:& C_p\lk \cI^{\cC_{\alpha}}_1(r,z,h)+\cI^{\cC_{\alpha}}_2(r,z,h)\rk\,.
   \end{align*}
  Therefore, by Minkowski's inequality
  \begin{align*}
  	\int_{\RR} \|\fD_h &J^{\cC_{\alpha}}_{\theta}(r,\cdot) \|^2_{L^p(\RR\times\Omega)} |h|^{2H-2}dh\\
  	=&\ \int_{\RR} \lc\int_{\RR} \EE\blk|\fD_h J^{\cC_{\alpha}}_{\theta}(r,z)|^p\brk dz \rc^{\frac{2}{p}} |h|^{2H-2}dh \\
  	\leq&\  \int_{0}^{r}\int_{\RR^2} (r-s)^{-2\theta} |\fD_h\cC_{\alpha}(r-s,y)|^2 |h|^{2H-2}dhdyds \times \|v\|^2_{\cZ_{2}^p(T)} \\
  	&+\int_{0}^{r}\int_{\RR^3} (r-s)^{-2\theta} \Big|\Box_{h,l}\cC_\alpha(r-s,y)\Big|^2|l|^{2H-2}|h|^{2H-2}dldhdyds \times \|v\|^2_{\cZ_{1}^p(T)}\\
  	=:& \cJ_1(r,z)\times \|v\|^2_{\cZ_{2}^p(T)}+\cJ_2(r,z)\times \|v\|^2_{\cZ_{1}^p(T)}\,.
  \end{align*}
%  \begin{align*}
%     \int_{\RR} \EE\bigg[\int_{\RR}&\lt|\fD_hJ^{\cC_{\alpha}}_{\theta}(r,z)\rt|^2|h|^{2H-2}dh\bigg]^{\frac p2}dz \\
%    \lesssim& \int_{\RR}\bigg(\int_{0}^{r}\int_{\RR^3} (r-s)^{-2\theta}\Blk \EE\big|[\fD_h\cC_{\alpha}(r-s,z-y-l)]v(s,y+l) \\
%    &~~~\ \ \ \ \ -[\fD_h\cC_{\alpha}(r-s,z-y)]v(s,y)\big|^p \Brk^{\frac 2p}|l|^{2H-2}|h|^{2H-2}dldhdyds\bigg)^{\frac p2}dz \\
%    \lesssim& \int_{\RR}\bigg(\int_{0}^{r}\int_{\RR^3} (r-s)^{-2\theta}\Blk \EE\big|[\fD_h\cC_{\alpha}(r-s,y+l)]v(s,y+z+l) \\
%    &~~~\ \ \ \ \ -[\fD_h\cC_{\alpha}(r-s,y)]v(s,y+z)\big|^p \Brk^{\frac 2p}|l|^{2H-2}|h|^{2H-2}dldhdyds\bigg)^{\frac p2}dz \\
%    \lesssim& \lt|\cJ_1(r,z) \rt|^{\frac p2}\times \|v\|^p_{\cZ_{2}^p(T)} + \lt|\cJ_2(r,z)\rt|^{\frac p2}\times \|v\|^p_{\cZ_{1}^p(T)}\,,
%  \end{align*}
%  where  $\cJ_1(r,z)$ and $\cJ_2(r,z)$ are defined as
%  \begin{align*}
%  	\cJ_1(r,z):=&\int_{0}^{r}\int_{\RR^2} (r-s)^{-2\theta} |\fD_h\cC_{\alpha}(r-s,y)|^2 |h|^{2H-2}dhdyds \,,\\
%  	\cJ_2(r,z):=&\int_{0}^{r}\int_{\RR^3} (r-s)^{-2\theta} \Big|\Box_{h,l}\cC_\alpha(r-s,y)\Big|^2|l|^{2H-2}|h|^{2H-2}dldhdyds\,.
%  \end{align*}
Applying \eqref{eq.Fourier} and Parseval's formula again, one can find
  \begin{align}\label{Est.Char1J1}
%     \cJ_1(r,z):=&\int_{0}^{r}\int_{\RR^2} (r-s)^{-2\theta} |\fD_h\cC_{\alpha}(r-s,y)|^2 |h|^{2H-2}dhdyds \nonumber\\
     \cJ_1(r,z)\es& \int_{0}^{r}\int_{\RR^2} (r-s)^{-2\theta} \Big|\fD_h\hat{\cC}_{\alpha}(r-s,\xi)\Big|^2 |h|^{2H-2}dhd\xi ds \nonumber\\
     \lesssim& \int_{0}^{r}s^{2(H+\alpha-\theta-1)}dr\cdot \int_{0}^{\infty}\xi^{1-2\alpha-2H}\lt|\cos(\xi)-e^{-\xi}\rt|^2 d\xi\,,
  \end{align}
  which is finite if \eqref{Ineq.alpha_I3} is  satisfied. Similarly, we have
  \begin{align}\label{Est.Char1J2}
%    \cJ_2(r,z):=&\int_{0}^{r}\int_{\RR^3} (r-s)^{-2\theta} \Big|\Box_{h,l}\cC_\alpha(r-s,y)\Big|^2|l|^{2H-2}|h|^{2H-2}dldhdyds \nonumber\\
    \cJ_2(r,z) \es&\int_{0}^{r}\int_{\RR^3} (r-s)^{-2\theta}|\xi|^{-2\alpha}\lt|\cos\blc(r-s)|\xi|\brc-e^{-(r-s)|\xi|}\rt|^2 \nonumber\\
    &~~~\times[1-\cos(|l\xi|)][1-\cos(|h\xi|)]\cdot|l|^{2H-2}|h|^{2H-2}dldhd\xi ds \nonumber\\
    \es&\int_{0}^{r}(r-s)^{2(\alpha+2H-\theta)-3}ds\cdot\int_{0}^{\infty} \xi^{2(1-\alpha-2H)}|\cos(\xi)-e^{-\xi}|^{2}d\xi\,.
  \end{align}
  In order to guarantee the integrals in \eqref{Est.Char1J2} converge, we must have
  \begin{align}\label{Ineq.alpha_J}
    &\ 2(\alpha+2H-\theta)-3>-1\,,~2(1-\alpha-2H)<-1 \nonumber\\
 &\qquad\qquad    \Leftrightarrow \ \theta<\alpha+2H-1\,,~\alpha>\frac 32-2H\,.
  \end{align}

 Therefore, with the choice of $\theta\in(1-2/q+\alpha,2H+\alpha-1)$ and $\alpha\in (\frac 32-2H, 1-\frac 1p)$ which implies $p>\frac{1}{H}$, by noting that $\frac{2}{4H-1}>\frac 1H$ when $H<\frac 12$, then the conditions   \eqref{Ineq.alpha_I3} and \eqref{Ineq.alpha_J} are satisfied. Thus, we complete the proof of \eqref{Est.Char2}.

  Now we   show   \eqref{Est.Char2_lem} for  $i=2$, i.e. $\cK_{2}=\cS_{\alpha}$ and $\bar{\cK}_{2}=\cC_{1-\alpha}$ only briefly since   the idea will be similar as in  the above  case   $i=1$. We only need to show that there exists some constant $C $   independent of $r\in [0, T]$,  such that
  \begin{equation}\label{Est.Char2_2}
     \EE \Blk\int_{\RR}\lt|\fD_hJ^{\cS_{\alpha}}_{\theta}(r,z)\rt|^2|h|^{2H-2}dh\Brk^{\frac p2}\leq C  \|v\|^p_{\cZ^p(T)}\,.
  \end{equation}
  Using Burkholder-Davis-Gundy's  inequality, Minkowski's inequality and then the  triangle inequality, we have the left hand side of \eqref{Est.Char2_2} is bounded by
  \begin{align*}
%     \EE \Blk\int_{\RR}&\lt|\fD_hJ^{\cS_{\alpha}}_{\theta}(r,z)\rt|^2|h|^{2H-2}dh\Brk^{\frac p2} \\
    \left(\widetilde{\cJ}_1(r,z)\right)^{\frac p2}\times \|v\|^p_{\cZ_{2}^p(T)} +\left(\widetilde{\cJ}_2(r,z)\right)^{\frac p2}\times \|v\|^p_{\cZ_{1}^p(T)}\,.
  \end{align*}
   Applying \eqref{eq.Fourier} and Parseval's formula again, one   finds
  \begin{align}\label{Est.Char1TilJ1}
     \widetilde{\cJ}_1(r,z):=&\int_{0}^{r}\int_{\RR^2} (r-s)^{-2\theta} |\fD_h\cS_{\alpha}(r-s,y)|^2 |h|^{2H-2}dhdyds \nonumber\\
%     \es& \int_{0}^{r}\int_{\RR^2} (r-s)^{-2\theta} \Big|\fD_h\hat{\cS}_{\alpha}(r-s,\xi)\Big|^2 |h|^{2H-2}dhd\xi ds \nonumber\\
     \lesssim& \int_{0}^{r}s^{2(H+\alpha-\theta-1)}dr\cdot \int_{0}^{\infty}\xi^{1-2\alpha-2H}\lt|\sin(\xi)\rt|^2 d\xi\,,
  \end{align}
  which is obviously bounded if \eqref{Ineq.alpha_I3} is satisfied. Similarly, we have
  \begin{align}\label{Est.Char1TilJ2}
    \widetilde{\cJ}_2(r,z):=&\int_{0}^{r}\int_{\RR^3} (r-s)^{-2\theta} \lt|\fD_h\cS_{\alpha}(r-s,y+l)
    -\fD_h\cS_{\alpha}(r-s,y)\rt|^2\nonumber\\ &\qquad\qquad\qquad\qquad\qquad\qquad\times|l|^{2H-2}|h|^{2H-2}dldhdyds \nonumber\\
%    \es&\int_{0}^{r}\int_{\RR^3} (r-s)^{-2\theta} |\xi|^{-2\alpha} \lt|\sin((r-s)|\xi|)\rt|^2 \nonumber\\
    %&~~~-[\cS_{\alpha}(r-s,y+h)-\cS_{\alpha}(r-s,y)]\big|^2 |l|^{2H-2}|h|^{2H-2}dldhdyds \nonumber\\
%    \es&\int_{0}^{r}\int_{\RR^3} (r-s)^{-2\theta}|\xi|^{-2\alpha}\lt|\cos\blc(r-s)|\xi|\brc-e^{-(r-s)|\xi|}\rt|^2 \nonumber\\
%    &\quad\times[1-\cos(|l\xi|)][1-\cos(|h\xi|)]\cdot|l|^{2H-2}|h|^{2H-2}dldhd\xi ds \nonumber\\
    \es&\int_{0}^{r}(r-s)^{2(\alpha+2H-\theta)-3}ds\cdot\int_{0}^{\infty} \xi^{2(1-\alpha-2H)}|\sin(\xi)|^{2}d\xi\,,
  \end{align}
  which is finite under \eqref{Ineq.alpha_J}.

  Therefore, with the  choice   $\theta\in(1-\frac2q+\alpha,2H+\alpha-1)$ and   $\alpha\in (\frac 32-2H, 1-\frac 1p)$ which implies $p>\frac{1}{H}$, we see the  conditions  \eqref{Ineq.alpha_I3} and \eqref{Ineq.alpha_J} are satisfied. So we finish the proof of \eqref{Est.Char2_2}. The other cases of \eqref{Est.Char2_lem} when $i=3$ and $i=4$ can be done by using the same strategy and  we omit them here.
\end{proof}

\section{Lemmas for Proposition \ref{p.4.1}}\label{Appen.C}
Our aim is to show for any $p>\frac 1H$ and $\gamma<H-\frac 1p$, the temporal-spatial H\"older continuity in Proposition \ref{p.4.1} hold by selecting appropriate $\alpha$, $\theta$ and $\eta$. Above all, we list some conditions which will be used frequently in   our technical lemmas.
\begin{enumerate}[start=1,label=\textrm{$\Pi$.\arabic*}]
\item \label{list1} \ $1-H<\alpha<\frac 1q$,\ \ $\alpha+\gamma<\frac 1q$,\  \ $\frac 1p<\theta<H+\alpha-\frac 12$\,;
\item \label{list2} \  $\theta>1+\alpha-\frac 2q+2\eta$, \ \ $\eta>\gamma$\,;
\item \label{list3} \ $\alpha+\eta>\frac 1q$,\ \ $\eta>\gamma$ \,;
\item \label{list4} \ $\alpha+\eta<\frac 1q$,\ \ $\eta>\gamma$\,.
%\item \label{list5} $\theta>1+\alpha-\frac 2q+2\eta$,\ \ $\eta>\gamma$\,.
%\item \label{list6} $\alpha<\frac 1q=1-\frac 1p,$\ \ $\theta>\frac 1p$\,;
%\item \label{list7} $\theta-\eta>1+\alpha-\frac 2q$,\ \ $\theta-\gamma>1+\alpha-\frac 2q$ \,;
\end{enumerate}
Throughout Appendix \ref{Appen.C}, we always assume $p>\frac 1H$ and $\gamma<H-\frac 1p$.

%\begin{enumerate}
%\item $\eta_1>\gamma\,,~\alpha+\eta_1>\frac 1q\,,~ \theta>1+\alpha-\frac 2q+\eta_1 $\,;
%\item $\eta_2>\gamma, \ \frac 1q>\alpha+\eta_2,\ \theta>1+\alpha-\frac 2q+\eta_2$
%\item $\frac 1q>\gamma+\alpha, \ \alpha+\eta_3>\frac 1q ,\  \theta>1-\frac 1q=\frac 1p.$
%\item $\alpha <\frac 1q=1-\frac 1p \,,\quad \theta>1-\frac 1q=\frac 1p,\quad \alpha+\gamma<\frac 1q<\frac 2q.$
%\item $\eta_4>\gamma\,,\quad \theta-2\eta_4>1+\alpha-\frac 2q\,,\quad \alpha+2\eta_4>\frac 1q$\,.
%\item $\alpha+\gamma<\frac 1q\,,\ \alpha+\eta_2>\frac 1q\,,\ \theta>\frac 1p$
%\item $\eta_3>\gamma, \ \frac 1q>\alpha+\eta_3,\ \theta>1+\alpha-\frac 2q+\eta_3$
%\item $\alpha<\frac 1q=1-\frac 1p\,,\ \theta>\frac 1p\,,\ \alpha+\gamma<\frac 1q$
%\item $\alpha<\frac{1}{q}=1-\frac 1p\,,\ \theta>\frac 1p\,,\ \theta>1+\alpha-\frac 2q+\gamma$
%\item   $ \alpha+\eta_2>\frac 1q,\ \eta_2>\gamma$
%\item $\alpha+\eta'_2<\frac 1q,\ \eta'_2>\gamma.$
%\item $\alpha<\frac 1q,\ \alpha+\gamma<\frac 1q$
%\item $\eta_4>\gamma\,, \ \theta-2\eta_4>1+\alpha-\frac 2q\,,\  \alpha+\eta_4>\frac 1q.$
%\end{enumerate}

\begin{lem}\label{I_{2,k}^{1}}
Suppose $\alpha,\ \theta$ satisfy $\eqref{list1}$ and
\begin{equation}
	\begin{cases}
		\eta_1 \text{ satisfies } \eqref{list2} \text{ and }  \eqref{list3};\\
		\eta_2 \text{ satisfies } \eqref{list2} \text{ and }  \eqref{list4};\\
		\eta_3 \text{ satisfies } \eqref{list3}\,.\label{Cond_AppeC1}
	\end{cases}
\end{equation}
Then   $\cI_{2,k}^{(1)}(t,h),\ k=1,2,3,4$ in \eqref{Est.I2(1)} can be bounded by $|h|^{\gamma q}$.
\end{lem}
\begin{proof}
For $\cI_{2,1}^{(1)}(t, h)$, since $(r+h)^{q(\theta-1)}\leq r^{q(\theta-1)}$ it can be bounded by
\begin{align*}
%	\int_0^t\int_{\mathbb{R}} r^{q(\theta-1)} &|r+|z||^{(-\alpha-\eta_1)q}h^{\eta_1 q} dzdr \\
	\cI_{2,1}^{(1)}(t, h) \ls\,&h^{\eta_1 q} \cdot \int_0^t\int_{\mathbb{R}} r^{q(\theta-1)} \cdot r^{1-(\alpha+\eta_1)q}|1+\tilde{z}|^{(-\alpha-\eta_1)q} d\tilde{z}dr \\
	\simeq \,&h^{\eta_1 q} \cdot \int_0^t r^{q(\theta-1)} \cdot r^{1-(\alpha+\eta_1)q} dr\,\simeq \,h^{\eta_1 q}\,\leq\, h^{\gamma q}
\end{align*}
where we  require   $\eta_1$ satisfy
\[
\eta_1>\gamma\,,\quad (\alpha+\eta_1)q>1
\,,\quad q(\theta-1)+1-(\alpha+\eta_1)q>-1\,,
\]
 which is
\begin{equation}\label{condi_eta_1}
 \eta_1>\gamma\,,~\alpha+\eta_1>\frac 1q\,,~ \theta>1+\alpha-\frac 2q+\eta_1 \,.
\end{equation}
Similarly, for $\cI_{2,2}^{(1)}(t, h) $  we have
\begin{align*}
%	\int_0^t\int_{\mathbb{R}} r^{q(\theta-1)} &|r-|z||^{(-\alpha-\eta_2)q}h^{\eta_2 q}\cdot \1_{A_1}dzdr \\
	\cI_{2,2}^{(1)}(t, h) \simeq \,&h^{\eta_2 q}\cdot\int_0^t\int_0^{r} r^{q(\theta-1)} (r-z)^{(-\alpha-\eta_2)q} dzdr \\
	\simeq \,& h^{\eta_2 q}\cdot\int_0^t r^{q(\theta-1)} r^{1-(\alpha+\eta_2)q} dr \,\simeq\, h^{\eta_2 q}\,\leq\, h^{\gamma q} \,,
\end{align*}
if we require
\begin{equation}\label{condi_eta_2}
\eta_2>\gamma, \ \frac 1q>\alpha+\eta_2,\ \theta>1+\alpha-\frac 2q+\eta_2.
\end{equation}
For $\cI_{2,3}^{(1)}(t, h) $  we have
\begin{align*}
%	\int_0^t\int_{\mathbb{R}}& r^{q(\theta-1)} ||z|-r-h|^{(-\alpha-\eta_3)q}h^{\eta_3 q}\cdot \1_{A_2} dzdr \\
\cI_{2,3}^{(1)}(t, h) 	\simeq\,&h^{\eta_3 q}\cdot\int_0^t\int_{h}^{\infty} r^{q(\theta-1)} z^{(-\alpha-\eta_3)q}dzdr  \\
	\simeq\,& h^{\eta_3 q}\cdot\int_0^t r^{q(\theta-1)} h^{1-(\alpha+\eta_3)q} dr \,\simeq\,h^{1-\alpha q}\,\leq\, h^{\gamma q} \,,
\end{align*}
under conditions
\begin{equation}\label{condi_eta_3}
\frac 1q>\gamma+\alpha, \ \alpha+\eta_3>\frac 1q ,\  \theta>1-\frac 1q=\frac 1p.
\end{equation}
For the last term  $\cI_{2,4}^{(1)}(t, h)  $   we have
\begin{align*}
%	\int_0^t\int_{\mathbb{R}} & (r+h)^{q(\theta-1)} \lt|(r+h-|z|)^{-\alpha}+(|z|-r)^{-\alpha}\rt|^q \cdot \1_{A_3}dzdr \\
\cI_{2,4}^{(1)}(t, h) 	\ls\,& h\cdot \int_0^t  r^{q(\theta-1)}\cdot \int_{\mathbb{R}}\lk |r+h-|z||^{-\alpha q}+||z|-r|^{-\alpha q}\rk \cdot \1_{A_3}dzdr \\
	\simeq\,& h\cdot \int_0^t  r^{q(\theta-1)}\cdot\lk 2\int_0^h z^{-\alpha q}dz+ \int_0^{2h} z^{-\alpha q}dz\rk dr \,\\
	\leq\,& h^{2-\alpha q}\cdot \int_0^t  r^{q(\theta-1)} dr\,\simeq\,h^{2-\alpha q}\leq h^{\gamma q}
\end{align*}
if we set
\begin{equation}\label{condi_eta_4}
\alpha <\frac 1q=1-\frac 1p \,,\quad \theta>1-\frac 1q=\frac 1p,\quad \alpha+\gamma<\frac 1q<\frac 2q.
\end{equation}
Notice that once $\alpha,\ \theta$ satisfy $\eqref{list1}$ and $\eta_1$, $\eta_2$ and $\eta_3$ satisfy \eqref{Cond_AppeC1}, then the conditions \eqref{condi_eta_1}-\eqref{condi_eta_4} hold automatically. The proof is complete.
\end{proof}
\begin{rmk}
	We remark here that the conditions \eqref{Cond_AppeC1} for  $\alpha$,   $\theta$, $\eta$'s  are compatible with $p>\frac 1H$ and $\gamma<H-\frac 1p$.   Let us summarize all the restrictions in Lemma \ref{I_{2,k}^{1}}:
	\begin{enumerate}
		\item $p>\frac 1H$,\ \ $\gamma<H-\frac 1p$\,;
		\item $1-H<\alpha<\frac 1q$,\ \ $\alpha+\gamma<\frac 1q$,\  \ $\frac 1p<\theta<H+\alpha-\frac 12$\,;
		\item $\eta_1>\gamma$, \ \ $\theta>1+\alpha-\frac 2q+2\eta_1$, \ \ $\alpha+\eta_1>\frac 1q$\,;
		\item $\eta_2>\gamma$, \ \ $\theta>1+\alpha-\frac 2q+2\eta_2$, \ \ $\alpha+\eta_2<\frac 1q$\,;
		\item $\eta_3>\gamma$, \ \ $\alpha+\eta_3<\frac 1q$\,.
	\end{enumerate}
	For any (fixed) $p>\frac 1H$, we can choose for (small enough) $\ep_k>0$ $k=1,\dots,6$
	\begin{align*}
		\gamma=H-\frac 1p-\ep_1\,,~\alpha=&1-H+\ep_2\,,~\theta=H-\ep_3\,,\\
		\eta_1=H-\frac 1p-\ep_4\,,~\eta_2=&H-\frac 1p-\ep_5\,,~\eta_3=H-\frac 1p-\ep_6\,.
	\end{align*}
	For  arbitrary (small enough) $\ep>0$  let %$\ep_k$'s take values
	\[
	\ep_1=7\ep\,,\ep_2=4\ep\,,\ep_3=\ep\,,\ep_4=3\ep\,,\ep_5=6\ep\,,\ep_6=6\ep\,.
	\]
	Then all the restrictions (1)-(5)  are satisfied with $\gamma$ arbitrarily close  to   $H-\frac 1p$. The following lemmas can be verified similarly. We omit the details.
\end{rmk}

\begin{lem}\label{Lem.I_{2,5+6}^2}
Suppose $\alpha,\ \theta$ satisfy $\eqref{list1}$ and
\begin{equation}
	\eta_4,\eta_5 \text{ satisfy } \eqref{list2} \text{ and }  \eqref{list3}\,.\label{Cond_AppeC2}
%	\begin{cases}
%		\eta_4 \text{ satisfy } \eqref{list2} \text{ and }  \eqref{list3};\\
%		\eta_5 \text{ satisfy } \eqref{list2} \text{ and }  \eqref{list3}\,.\label{Cond_AppeC2}
%	\end{cases}
\end{equation}
%Moreover, let $\eta_4$ and $\eta_5$ both satisfy $\eqref{list2}$ and  $\eqref{list3}$,
Then the terms $\cI_{2,5}^{(2)}(t, h)$ and $\cI_{2,6}^{(2)}(t, h)$ in equation \eqref{Ineq_I_2,k^2} can be bounded by $|h|^{\gamma q}$.

\end{lem}

\begin{proof}
For the term $\cI_{2,5}^{(2)}(t, h)$, from inequality  \eqref{r2_regular} and $(r+h)^{\eta q}\leq r^{\eta q}+h^{\eta q}$  it follows
\begin{align*}
 \cI_{2,5}^{(2)}(t, h) \ls &\int_0^t\int_{\mathbb{R}} (r+h)^{q(\theta-1)} \lt|\Delta_h\lt(r^2+|z|^2\rt)^{-\frac{\alpha}{2}}\rt|^{q} dzdr\\
	\ls&h^{\eta_4q}\int_0^t\int_{\mathbb{R}}(r+h)^{q(\theta-1)}\lt(r^2+z^2\rt)^{-(\frac{\alpha}{2}+\eta_4)q}(r+h)^{\eta_4q}dzdr \\
\ls&\, h^{\eta_4q}\cdot\int_0^t r^{q(\theta-1)+1-(\al+\eta_4)q}dr \cdot\int_{\mathbb{R}}\lt(1+z^2\rt)^{-(\frac{\alpha}{2}+\eta_4)q}dz  \\
&+ h^{2\eta_4q}\cdot\int_0^t r^{q(\theta-1)+1-(\al+2\eta_4)q}dr \cdot\int_{\mathbb{R}}\lt(1+z^2\rt)^{-(\frac{\alpha}{2}+\eta_4)q}dz\ls\ h^{q\gamma}
\end{align*}
if $\eta_4$ satisfies the following conditions
\begin{equation}\label{Condi_I_22^4}
\eta_4>\gamma\,,\quad \theta-2\eta_4>1+\alpha-\frac 2q\,,\quad \alpha+2\eta_4>\frac 1q \,.
\end{equation}

Now we deal with   $\cI_{2,6}^{(2)}(t, h) $. For fixed $\eta\in(0,1)$ by \eqref{Ineq_imp_2} and then
% we have
%$$\lt|\Delta_h\cos\lc\alpha\tan^{-1}\lc\frac{|z|}{r}\rc\rc\rt|\ls\frac{|z|^{\eta}|h|^{\eta}}{(r^2+z^2)^{\eta}} .$$
  by changing of variable $z\to rz$,
\begin{align*}
\cI_{2,6}^{(2)}(t, h) \ls &\int_0^t\int_{\mathbb{R}} (r+h)^{q(\theta-1)}\frac{|z|^{\eta_5q}|h|^{\eta_5q}}{(r^2+z^2)^{\eta_5q}}\lt(r^2+z^2\rt)^{-\frac{\alpha}{2}q}dzdr\\
\ls&h^{\eta_5q}\int_0^t\int_{\mathbb{R}} r^{q(\theta-1)}\frac{|z|^{\eta_5q}}{(r^2+z^2)^{\eta_5q}}\lt(r^2+z^2\rt)^{-\frac{\alpha}{2}q}dzdr\\
=&h^{\eta_5q}\int_0^tr^{q(\theta-1)-\eta_5q-\alpha q+1}dr\cdot\int_{\mathbb{R}}\frac{|z|^{\eta_5q}}{(1+z^2)^{\eta_5q}}\lt(1+z^2\rt)^{-\frac{\alpha}{2}q}dz,
\end{align*}
which can be bounded by $h^{\gamma q}$ under conditions \eqref{condi_eta_1} with $\eta_1$ replaced with $\eta_5$, i.e.
\begin{equation}\label{Condi_I_22^4_2}
 \eta_5>\gamma\,,~\alpha+\eta_5>\frac 1q\,,~ \theta>1+\alpha-\frac 2q+\eta_5 \,.
\end{equation}
Therefore, under conditions \eqref{Condi_I_22^4} and \eqref{Condi_I_22^4_2},  we have  for $k=5,6$,
\[
\sup_{t}\cI_{2,k}^{(2)}(t, h) \ls
|h|^{\gamma q}\,.
\]
Notice that once $\alpha,\ \theta$ satisfy $\eqref{list1}$ and $\eta_4$, $\eta_5$ satisfy \eqref{Cond_AppeC2}, then the conditions \eqref{Condi_I_22^4}-\eqref{Condi_I_22^4_2} hold automatically. The proof is complete.
\end{proof}

\begin{lem}\label{J_{1,k}^{(1)}}
Suppose $\alpha,\ \theta$ satisfy $\eqref{list1}$ and
\begin{equation}
	\begin{cases}
		\eta_1 \text{ satisfies } \eqref{list2} \text{ and }  \eqref{list3};\\
		\eta_2 \text{ satisfies } \eqref{list3};\\
		\eta_3 \text{ satisfies } \eqref{list2} \text{ and }  \eqref{list4}\,.\label{Cond_AppeC3}
	\end{cases}
\end{equation}
Then the term $\cJ_{1,k}^{(1)}(t,x,y)$ in \eqref{Ineq_J_1,k^1} can be bounded as $$\sup\limits_{t,x,y}\cJ_{1,k}^{(1)}(t,x,y)\ls C_{T,p,H,\gamma}|\hbar|^{\gamma q}\quad \hbox{
for $\ k=1,2,3$.}$$
%if moreover $\eta_1$ satisfies $\eqref{list2}$ and $\eqref{list3}$; $\eta_2$ satisfies $\eqref{list3}$; $\eta_3$ satisfies $\eqref{list2}$ and $\eqref{list4}$.
\end{lem}

\begin{proof}
Similar to  the proof of $\cI_{2,1}^{(1)}$ in  part \textbf{(i)} of Proposition \ref{p.4.1}, $\cJ_{1,1}^{(1)}(t,x,y)$ can be bounded by $|\hbar|^{\gamma q}$ under  the same condition as   \eqref{condi_eta_1} which is implied by conditions on  $\eta_1$ in \eqref{Cond_AppeC3}.

Now we deal with $\cJ_{1,2}^{(1)}(t,x,y)$.  By triangle inequality%$\footnote{Refer the definitions of $B_k$?}$
\begin{align}\label{iv1est}
\cJ_{1,2}^{(1)}(t,x,y)=&\int_{\hbar}^t\int_{\RR}r^{q(\theta-1)}\lt|\fD_{\hbar}(r-|z|)^{-\alpha}\rt|^q\cdot\lt(\1_{B_1}+\1_{B_2}+\1_{B_3}\rt)dzdr\nonumber\\
\leq&\int_{\hbar}^t\int_{z<-r-\hbar}r^{q(\theta-1)}|r-|z||^{(-\alpha-\eta_2)q}\hbar^{\eta_2q}dzdr\nonumber\\
&+\int_{\hbar}^t\int_{z>r+\hbar}r^{q(\theta-1)}|r-|z||^{(-\alpha-\eta_2)q}\hbar^{\eta_2q}dzdr\nonumber\\
%&&\tcg{+\int_{\hbar}^t\int_{r<z<r+1}r^{q(\theta-1)}|r-|z||^{(-\alpha-\eta_3)q}\hbar^{\eta_3q}dzdr}\nonumber\\
&  +\int_{\hbar}^t\int_{-r+\hbar}^{r-\hbar}r^{q(\theta-1)}|r-\hbar-|z||^{(-\alpha-\eta_3)q}\hbar^{\eta_3q}dzdr \nonumber\\
=:&\sum_{j=1}^3\cJ_{1,2,j}^{(1)}(t,x,y)\,.,
%&%\simeq&|x-y|^{1-\alpha q}\int_0^tr^{q(\theta-1)}dr+|x-y|^{\eta_2 q}\int_0^tr^{q(\theta-1)}dr\\
%&&+\int_0^t\int_{-r}^{0}r^{q(\theta-1)}(r+z)^{(-\alpha-\eta_2)q}|x-y|^{\eta_2q}dzdr\\
%&&+\int_0^t\int_{0}^{y-x+r}r^{q(\theta-1)}(r-z)^{(-\alpha-\eta_2)q}|x-y|^{\eta_2q}dzdr.
\end{align}
where $B_1$, $B_2$ and $B_3$ are defined by \eqref{Def.B_set}.

For the term $\cJ_{1,2,1}^{(1)}(t,x,y)$  in \eqref{iv1est}, we have
\begin{align*}
%\int_{\hbar}^t\int_{z<-\hbar-r}&r^{q(\theta-1)}|r-|z||^{(-\alpha-\eta_2)q}\hbar^{\eta_2 q}dzdr \\
%=\int_0^t\int_{z<-\hbar-r}r^{q(\theta-1)}(-z-r)^{(-\alpha-\eta_2)q}\hbar^{\eta_2q}dzdr \\
\cJ_{1,2,1}^{(1)}(t,x,y)\simeq\,&\hbar^{\eta_2 q}\int_0^tr^{q(\theta-1)}\int_{z>\hbar}z^{-(\alpha+\eta_2)q}dzdr\\
\simeq\,& \hbar^{1-(\alpha+\eta_2) q+\eta_2 q}\int_0^tr^{q(\theta-1)}dr
\simeq\, \hbar^{1-\alpha q}\lesssim \hbar^{\gamma q}\,,
\end{align*}
under  the same conditions as \eqref{condi_eta_3}:
\begin{equation}\label{condi_eta2_iv}
	\alpha+\gamma<\frac 1q\,,\ \alpha+\eta_2>\frac 1q\,,\ \theta>\frac 1p\,.
\end{equation}
%$\gamma<\frac 1q -\alpha$, $(\alpha+\eta_2)q>1$ and $q(\theta-1)>-1.$
Similar to $\cJ_{1,2,1}^{(1)}(t,x,y)$,
%then $\cJ_{1,2,2}^{(1)}(t,x,y)$ can be bounded as follows
if the conditions in \eqref{condi_eta2_iv} hold,
then we have
\begin{align*}
%\int_{\hbar}^t \int_{r+\hbar}^{+\infty}&r^{q(\theta-1)}|r-|z||^{(-\alpha-\eta_2)q}\hbar^{\eta_2q}dzdr \\
\cJ_{1,2,1}^{(1)}(t,x,y)&=\int_{\hbar}^t\int_{r+\hbar}^{+\infty}r^{q(\theta-1)}(z-r)^{(-\alpha-\eta_2)q}\hbar^{\eta_2q}dzdr \\
&=\hbar^{\eta_2q}\int_{\hbar}^t\int_{\hbar}^{+\infty}r^{q(\theta-1)}z^{-(\alpha+\eta_2)q}dzdr \\
&\simeq \hbar^{1-\alpha q}\int_0^tr^{q(\theta-1)}dr
\lesssim \hbar^{\gamma q}\,.
\end{align*}
%which requires $\gamma<\frac 1q -\alpha$, $(\alpha+\eta_2)q>1$ and $q(\theta-1)>-1.$

To estimate $\cJ_{1,2,3}^{(1)}(t,x,y)$  in \eqref{iv1est}, letting $\eta_3$ satisfy  the conditions
\eqref{condi_eta_2} with $\eta_2$ replaced by
  $\eta_3$, namely,
\begin{equation}\label{condi_eta3_iv}
\eta_3>\gamma, \ \frac 1q>\alpha+\eta_3,\ \theta>1+\alpha-\frac 2q+\eta_3\,,
\end{equation}
    we have
\begin{align*}
\cJ_{1,2,3}^{(1)}(t,x,y)
%=&\int_{\hbar}^t\int_{-r+\hbar}^{r-\hbar}r^{q(\theta-1)}|r-\hbar-|z||^{(-\alpha-\eta_3)q}\hbar^{\eta_3q}dzdr\\
=&\int_{\hbar}^t\int_{-r+\hbar}^{0}r^{q(\theta-1)}|r-\hbar+z|^{(-\alpha-\eta_3)q}\hbar^{\eta_3q}dzdr \\
&+\int_{\hbar}^t\int_{0}^{r-\hbar}r^{q(\theta-1)}|r-\hbar-z|^{(-\alpha-\eta_3)q}\hbar^{\eta_3q}dzdr\\
\simeq\, & \hbar^{\eta_3q}\cdot\int_{\hbar}^t\int_{0}^{r-\hbar}r^{q(\theta-1)}z^{(-\alpha-\eta_3)q}dzdr\\
\lesssim\, &\hbar^{\eta_3q}\cdot\int_{0}^t r^{1-(\alpha+\eta_3)q+q(\theta-1)}dr
\lesssim \hbar^{\eta_3q}\lesssim\hbar^{\gamma q}\,.
\end{align*}
%which requires $\eta_3>\gamma$ , $(\alpha+\eta_3)q<1$ and $1-(\alpha+\eta_3)q+q(\theta-1)>-1.$

Now we proceed to deal with $\cJ_{1,3}^{(1)}(t,x,y)$ in \eqref{Est.iv1}.   By the   similar way as  dealing with $\cJ_{1,2}^{(1)}(t,x,y)$,  we have
with   $B_4$ and  $B_5$   defined by \eqref{Def.B_set}.
\begin{align*}
\int_{\hbar}^t\int_{\RR}r^{q(\theta-1)}&\lt|(r-|z+\hbar|)^{-\alpha}+(r-|z|)^{-\alpha}\rt|^q\cdot\lt( \1_{B_4}+\1_{B_5} \rt)dzdr\\
=&\int_{\hbar}^t\int_{-r-\hbar}^{-r+\hbar} r^{q(\theta-1)}\lt(\lt|r-|z+\hbar|\rt|^{-\alpha q}+\lt|r-|z|\rt|^{-\alpha q} \rt)dzdr\\
&+\int_{\hbar}^t\int_{r-\hbar}^{r+\hbar} r^{q(\theta-1)}\lt(\lt|r-|z+\hbar|\rt|^{-\alpha q}+\lt|r-|z|\rt|^{-\alpha q}\rt) dzdr\\
\simeq&\  \int_{\hbar}^t \int_{-\hbar}^{\hbar}r^{q(\theta-1)}|z|^{-\alpha q} dzdr+ \int_{\hbar}^t \int_{0}^{2\hbar}r^{q(\theta-1)}|z|^{-\alpha q} dzdr \\
%\lesssim&\int_{\hbar}^t\int_{-r-\hbar}^{-r+\hbar}  \lt[(r+z+\hbar)^{-\alpha q}+(-z-r)^{-\alpha q}\rt]dzdr\\
%&+\int_{\hbar}^t\int_{r-\hbar}^{r+\hbar}\lt[(z+\hbar-r)^{-\alpha q}+(r-z)^{-\alpha q}\rt]dzdr\\
\lesssim &\ \hbar^{1-\alpha q}\int_0^tr^{q(\theta-1)}dr\lesssim \hbar^{1-\alpha q}\lesssim \hbar^{\gamma q}\,,
\end{align*}
under  the same conditions as \eqref{condi_eta_4}:
\begin{equation}\label{condi_gamma_iv}
	\alpha<\frac 1q=1-\frac 1p\,,\ \theta>\frac 1p\,,\ \alpha+\gamma<\frac 1q\,.
\end{equation}
Therefore, if $\alpha,\ \theta$ satisfy $\eqref{list1}$ and $\eta_1$, $\eta_2$, $\eta_3$ satisfy \eqref{Cond_AppeC3}, then we have our desire upper bounds for $\sup\limits_{t,x,y} \cJ_{1,k}^{(1)}(t,x,y)$ $(k=1,2,3)$.
\end{proof}

\begin{lem}\label{J_{2,k}^{(1)}}

Suppose $\alpha,\ \theta$ satisfy $\eqref{list1}$, and moreover
\begin{equation}
	\eta_4\text{ satisfies  } \eqref{list2} \text{ and }  \eqref{list3}\,.\label{Cond_AppeC4}
%	\begin{cases}
%		\eta_4 \text{ satisfy } \eqref{list2} \text{ and }  \eqref{list3};\\
%		\eta_5 \text{ satisfy } \eqref{list2} \text{ and }  \eqref{list3}\,.\label{Cond_AppeC2}
%	\end{cases}
\end{equation}
%$\eta_4$ satisfies $\eqref{list2}$ and $\eqref{list3}$,
Then the terms $\cJ_{2,k}^{(1)}(t,x,y)$ in \eqref{J_{2,K}^1} can be bounded as follows $$\sup\limits_{t,x,y}\cJ_{2,k}^{(1)}(t,x,y)\ls C_{T,p,H,\gamma}|\hbar|^{\gamma q}
\quad \hbox{for  $\ k=1,2,3$.}$$
\end{lem}

\begin{proof}
Similar to the way when we deal with   $\cI_2^{(1)}$ in the proof of part \textbf{(i)} of Proposition \ref{p.4.1}, $\cJ_{2,1}^{(1)}(t,x,y)$ can be bounded by $\hbar^{\gamma q}$ under the condition  \eqref{condi_eta_1} which holds under condition \eqref{Cond_AppeC4}.  Let us recall the definitions of $C_1$, $C_2$ and $C_3$ in \eqref{Def.C_set}, then for $\cJ_{2,2}^{(1)}(t,x,y)$ we have%$\footnote{Definition of $C_1,~C_2$.}$
\begin{align}\label{iv1_2est}
\cJ_{2,2}^{(1)}(t,x,y)=&\int_0^{\hbar}\int_{\RR}r^{q(\theta-1)}\lt|\fD_{\hbar}(r-|z|)^{-\alpha}\rt|^q\cdot\lt(\1_{C_1}+\1_{C_2}\rt)dzdr\nonumber\\
\leq&\int_0^{\hbar}\int_{z<-r-\hbar}r^{q(\theta-1)}|r-|z||^{(-\alpha-\eta_4)q}\hbar^{\eta_4q}dzdr\nonumber\\
& +\int_0^{\hbar}\int_{z>r+\hbar}r^{q(\theta-1)}|r-|z||^{(-\alpha-\eta_4)q}\hbar^{\eta_4q}dzdr\,.
%&&\tcg{+\int_{\hbar}^t\int_{r<z<r+1}r^{q(\theta-1)}|r-|z||^{(-\alpha-\eta_3)q}\hbar^{\eta_3q}dzdr}\nonumber\\
%& \tcb{+\int_0^{\hbar}\int_{-r+\hbar}^{r-\hbar}r^{q(\theta-1)}|r-\hbar-|z||^{(-\alpha-\eta_3)q}\hbar^{\eta_3q}dzdr}\,.
%&%\simeq&|x-y|^{1-\alpha q}\int_0^tr^{q(\theta-1)}dr+|x-y|^{\eta_2 q}\int_0^tr^{q(\theta-1)}dr\\
%&&+\int_0^t\int_{-r}^{0}r^{q(\theta-1)}(r+z)^{(-\alpha-\eta_2)q}|x-y|^{\eta_2q}dzdr\\
%&&+\int_0^t\int_{0}^{y-x+r}r^{q(\theta-1)}(r-z)^{(-\alpha-\eta_2)q}|x-y|^{\eta_2q}dzdr.
\end{align}
For the first term  of the summation in \eqref{iv1_2est}, we have
\begin{align*}
\int_0^{\hbar}\int_{z<-r-\hbar}&r^{q(\theta-1)}|r-|z||^{(-\alpha-\eta_4)q}\hbar^{\eta_4q}dzdr\\
&=\hbar^{\eta_4q}\int_0^{\hbar}\int_{z<-r-\hbar}r^{q(\theta-1)}(-z-r)^{(-\alpha-\eta_4)q}dzdr\\
&=\hbar^{\eta_4q}\int_0^{\hbar}\int_{z>\hbar}r^{q(\theta-1)}z^{(-\alpha-\eta_4)q}dzdr\\
&\simeq \hbar^{\eta_4q}\hbar^{1-(\alpha+\eta_4)q}\int_0^{\hbar}r^{q(\theta-1)}dr\\
&\simeq \hbar^{\eta_4q+1-(\alpha+\eta_4)q+1+q(\theta-1)}\lesssim\hbar^{\gamma q},
\end{align*}
under  the same conditions  as   \eqref{condi_eta_1} with $\eta_1$ replaced by $\eta_4$.
%$(\alpha+\eta_2)q>1$, $q(\theta-1)>-1$ and $q\gamma <2-\alpha q+q(\theta-1).$

Similarly, we have for the second term of the sum  in \eqref{iv1_2est}
\begin{align*}
\int_0^{\hbar}&\int_{z>r+\hbar}r^{q(\theta-1)}|r-|z||^{(-\alpha-\eta_4)q}\hbar^{\eta_4q}dzdr \\
&=\int_0^{\hbar}\int_{z>r+\hbar}r^{q(\theta-1)}(z-r)^{(-\alpha-\eta_4)q}\hbar^{\eta_4q}dzdr \\
&=\hbar^{\eta_4q}\int_0^{\hbar}\int_{z>\hbar}r^{q(\theta-1)}z^{-(\alpha+\eta_4)q}dzdr\\
&\simeq \hbar^{\eta_4q}\hbar^{1-(\alpha+\eta_4)q}\int_0^{\hbar}r^{q(\theta-1)}dr
\lesssim \hbar^{\eta_4q+1-(\alpha+\eta_4)q+1+q(\theta-1)},
\end{align*}
which can be bounded by $\hbar^{\gamma q}$ if the condition \eqref{condi_eta_1}  with $\eta_1$ replaced by $\eta_4$   holds.
% we require $(\alpha+\eta_2)q>1$, $q(\theta-1)>-1$ and $q\gamma<2-\alpha q+q(\theta-1).$

 For the last term $\cJ_{2,3}^{(1)}(t,x,y)$, if $\alpha,\ \theta$ satisfy $\eqref{list1}$, then the conditions
 \begin{equation}\label{condi_J_{2,3}}
 	\alpha<\frac{1}{q}=1-\frac 1p\,,\ \theta>\frac 1p\,,\ \theta>1+\alpha-\frac 2q+\gamma\,,
 \end{equation}
are satisfied. So we have%$\footnote{Definition of $C_3$.}$
\begin{align*}
\cJ_{2,3}^{(1)}(t,x,y)=&\int_0^{\hbar}\int_{\RR}r^{q(\theta-1)}\lt|(r-|z+\hbar|)^{-\alpha}+(r-|z|)^{-\alpha}\rt|^q\cdot \1_{C_3}  dzdr\\
=&\int_0^{\hbar}\int_{-r-\hbar}^{r+\hbar} r^{q(\theta-1)}\lt(\lt|r-|z+\hbar|\rt|^{-\alpha q}+\lt|r-|z|\rt|^{-\alpha q}\rt) dzdr\\
%&+\int_0^{\hbar}\int_{r-\hbar}^{r+\hbar} r^{q(\theta-1)}\lt|r-|z+\hbar|\rt|^{-\alpha q}+\lt|r-|z|\rt|^{-\alpha q} dzdr\\
\simeq&\ \int_0^{\hbar} \int_{0}^{r}r^{q(\theta-1)}|z|^{-\alpha q} dzdr+ \int_0^{\hbar} \int_{0}^{\hbar}r^{q(\theta-1)}|z|^{-\alpha q} dzdr \\
%\lesssim&\int_{\hbar}^t\int_{-r-\hbar}^{-r+\hbar}  \lt[(r+z+\hbar)^{-\alpha q}+(-z-r)^{-\alpha q}\rt]dzdr\\
%&+\int_{\hbar}^t\int_{r-\hbar}^{r+\hbar}\lt[(z+\hbar-r)^{-\alpha q}+(r-z)^{-\alpha q}\rt]dzdr\\
\simeq &\ \int_0^{\hbar} r^{q(\theta-1)+1-\alpha q}dr+ \hbar^{1-\alpha q} \int_{0}^{\hbar}r^{q(\theta-1)}dr \\
\lesssim &\ \hbar^{2-\alpha q+q(\theta-1)}\lesssim \hbar^{\gamma q}\,.
\end{align*}
Thus, the proof is complete.
%which requires $\alpha q<1$, $q(\theta-1)>-1$ and $2-\alpha q+q(\theta-1)>\gamma q$.
\end{proof}

\begin{lem}\label{J_{2,j}^2}
Suppose $\alpha,\ \theta$ satisfy $\eqref{list1}$ and
\begin{equation}
	\begin{cases}
		\eta_2 \text{ satisfies } \eqref{list3};\\
		\eta_3 \text{ satisfies } \eqref{list4};\\
		\eta_4 \text{ satisfies } \eqref{list2} \text{ and }  \eqref{list3}\,.\label{Cond_AppeC5}
	\end{cases}
\end{equation}
%Moreover, let $\eta_2$ satisfy $\eqref{list3};$ $\eta'_2$ satisfy $\eqref{list4};$ $\eta_3$ satisfy $\eqref{list2}$ and $\eqref{list3}.$
Then  the $\cJ_{2,1, j}^{(2)}(t,x,y),\ j=1,\cdots,6$ in \eqref{Ineq_J_21^2} can be bounded as follows.
 $$\sup\limits_{t,x,y}\cJ_{2,1,j}^{(2)}(t,x,y)\ls|\hbar|^{\gamma q}.$$

\end{lem}

\begin{proof}
Let us recall the definitions of $D_1,\cdots,D_6$ in \eqref{Def.D_set}. Firstly, we deal with $\cJ_{2,1,1}^{(2)}$ and $\cJ_{2,1,5}^{(2)}$ on $D_1$ and $D_5$  successively. We have
\begin{align}\label{Ineq_F_1,5}
\cJ_{2,1,1}^{(2)}&(t,x,y)+\cJ_{2,1,5}^{(2)}(t,x,y)\nonumber\\
=&|\hbar|^{\eta_2 q}\int_{0}^{t}\int_{z<-r-\hbar}r^{q(\theta-1)}(-z-r)^{-(\alpha+\eta_2)q}dzdr\nonumber\\
&+|\hbar|^{\eta_2 q}\int_{0}^{t}\int_{r}^{r+\hbar}r^{q(\theta-1)}(z-r)^{-(\alpha+\eta_2)q}dzdr\nonumber\\
\ls&|\hbar|^{\eta_2 q}\int_{0}^{t}r^{q(\theta-1)}dr\cdot\int_{\widetilde{z}>\hbar}(\widetilde{z})^{-(\alpha+\eta_2)q}dz\nonumber\\
&+|\hbar|^{\eta_2 q}\int_{0}^{t}r^{q(\theta-1)}dr\cdot\int_0^{\hbar}(\widehat{z})^{-(\alpha+\eta_2)q}dz\,,
\end{align}
through changing of variables $\widetilde{z}=-z-r$ and $\widehat{z}=z-r$. Thus, it can be bounded by $|\hbar|^{\gamma q}$ if
\begin{equation}\label{condi_J_21^2}
\alpha+\eta_2>\frac 1q,\ \eta_2>\gamma.
\end{equation}
In the same way, we can deal with $\cJ_{2,1,6}^{(2)}(t,x,y)$ by changing of variable $\widehat{z}=z-r,$
\begin{align*}
|\hbar|^{\eta_3 q}\int_{0}^{t}&\int_{z>r+\hbar}r^{q(\theta-1)}(z-r)^{-(\alpha+\eta_3)q}dzdr\\
\ls&|\hbar|^{\eta_3 q}\int_{0}^{t}r^{q(\theta-1)}dr\cdot\int_{\widehat{z}>\hbar}(\widehat{z})^{-(\alpha+\eta_3)q}dz\ls|\hbar|^{\gamma q},
\end{align*}
which requires $\eta_3$ satisfying the  conditions \eqref{condi_J_21^2} and
\begin{equation}\label{condi_J_22^2}
\alpha+\eta_3<\frac 1q,\ \eta_3>\gamma.
\end{equation}
Similarly, by changing of variable $z\to z+\hbar$ and then $z\to rz$, we have on $D_3,$
\begin{align}\label{Ineq_F_3}
\cJ_{2,1,3}^{(2)}(t,x,y)\ls&\hbar^{\eta_4 q}\int_{0}^{t}\int_{-r}^{r-\hbar}r^{q(\theta-1)}|r-|z+\hbar||^{-(\alpha+\eta_4)q}dzdr\nonumber\\
\ls&\hbar^{\eta_4 q}\int_{0}^{t}\int_{-r}^{r}r^{q(\theta-1)}|r-|z||^{-(\alpha+\eta_4)q}dzdr\nonumber\\
=&\hbar^{\eta_4 q}\int_{0}^{t}r^{q(\theta-1)-(\alpha+\eta_4)q+1}dr\cdot\int_{0}^{1}|1-|z||^{-(\alpha+\eta_4)q}dz\ls\hbar^{\gamma q}\,,
\end{align}
which requires  the same condition as \eqref{condi_eta_1} with $\eta_1$ replaced by $\eta_4$ here.
%Notice that it holds under the condition of $\eta_3$ in \eqref{Cond_AppeC5}.

As for $\cJ_{2,1,2}^{(2)}(t,x,y)$ and $\cJ_{2,1,4}^{(2)}(t,x,y)$, we have
\begin{align}\label{J_22^2}
\cJ_{2,1,2}^{(2)}(t,x,y)&+\cJ_{2,1,4}^{(2)}(t,x,y)\nonumber\\
=&\int_{0}^{t}\lt(\int_{-r-\hbar}^{-r}+\int_{r-\hbar}^{r}\rt)r^{q(\theta-1)}\lt|\fD_{\hbar}|r-|z||^{-\alpha}\rt|^qdzdr\nonumber\\
\ls&\int_{0}^{t}\lt(\int_{-r-\hbar}^{-r}+\int_{r-\hbar}^{r}\rt)r^{q(\theta-1)}|r-|z+\hbar||^{-\alpha q}dzdr\nonumber\\
&+\int_{0}^{t}\lt(\int_{-r-\hbar}^{-r}+\int_{r-\hbar}^{r}\rt)r^{q(\theta-1)}|r-|z||^{-\alpha q}dzdr\nonumber\\
\ls&\int_{0}^{t}r^{q(\theta-1)}dr\cdot\int_{0}^{\hbar}|z|^{-\alpha q}dz\ls |\hbar|^{1-\alpha q}\ls |\hbar|^{\gamma q},
\end{align}
if we require
\begin{equation}\label{condi_J_23^2}
\theta>\frac 1p,\ \alpha<\frac 1q,\ \alpha+\gamma<\frac 1q.
\end{equation}
Thus, if $\alpha,\ \theta$ satisfy $\eqref{list1}$ and if \eqref{Cond_AppeC5} holds, then all the restrictions on $\eta$'s are satisfied. The proof is then complete.
\end{proof}

\begin{lem}\label{lem_J_3+J_4}
Suppose $\alpha,\ \theta$ satisfy $\eqref{list1}$ and moreover
\begin{equation}
	\eta_4,\eta_5 \text{ satisfy } \eqref{list2} \text{ and }  \eqref{list3}\,.\label{Cond_AppeC6}
%	\begin{cases}
%		\eta_4 \text{ satisfy } \eqref{list2} \text{ and }  \eqref{list3};\\
%		\eta_5 \text{ satisfy } \eqref{list2} \text{ and }  \eqref{list3}\,.\label{Cond_AppeC2}
%	\end{cases}
\end{equation}
%if moreover $\eta_4$ and $\eta_5$ both satisfy $\eqref{list2}$ and $\eqref{list3},$
Then the terms $\sup\limits_{t,x,y}\cJ_3^{(2)}(t,x,y)$ and $\sup\limits_{t,x,y}\cJ_4^{(2)}(t,x,y)$ in \eqref{Ineq_J_3+J_4} can be bounded by a constant  multiple of $|\hbar|^{\gamma q}$.
\end{lem}

\begin{proof}
For the term $\cJ_3^{(2)}(t,x,y)$, by \eqref{regular_J_3^2} and the inequality $|z+\hbar|^{\eta_4 q}\ls |z|^{\eta_4 q}+|\hbar|^{\eta_4 q}$ , we have
\begin{align}\label{Ineq_J_3^2}
\cJ_3^{(2)}(t,x,y)\ls&|\hbar|^{\eta_4 q}\int_{0}^{t}\int_{\RR}r^{q(\theta-1)}(r^2+z^2)^{-(\frac{\alpha}{2}+\eta_4 )q}|z|^{\eta_4 q}dzdr\nonumber\\
&+|\hbar|^{2\eta_4 q}\int_{0}^{t}\int_{\RR}r^{q(\theta-1)}(r^2+z^2)^{-(\frac{\alpha}{2}+\eta_4 )q}dzdr\nonumber\\
=&|\hbar|^{\eta_4 q}\int_{0}^{t}r^{q(\theta-1)-(\alpha+2\eta_4 )q+\eta_4 q+1}dr\cdot\int_{\RR}(1+z^2)^{-(\frac{\alpha}{2}+\eta_4 )q}|z|^{\eta_4 q}dz\nonumber\\
&+|\hbar|^{2\eta_4 q}\int_{0}^{t}r^{q(\theta-1)-(\alpha+2\eta_4 )q+1}dr\cdot\int_{\RR}(1+z^2)^{-(\frac{\alpha}{2}+\eta_4 )q}dz\,,
\end{align}
which can be bounded by $|\hbar|^{\gamma q}$ under the following conditions
\begin{equation}\label{Condi_J_3^2}
\eta_4>\gamma\,, \ \theta-2\eta_4>1+\alpha-\frac 2q\,,\  \alpha+\eta_4>\frac 1q.
\end{equation}

As for the term $\cJ_4^{(2)}(t,x,y)$, by inequality \eqref{regular_J_4^2} and by changing of variable $z\to rz,$
\begin{align}\label{Ineq_J_4^2}
\cJ_4^{(2)}(t,x,y)\ls&|\hbar|^{\eta_5 q}\int_{0}^{t}\int_{\RR}r^{q(\theta-1)}\frac{r^{\eta_5 q} }{(r^2+z^2)^{\eta_5 q}}(r^2+z^2)^{-\frac{\alpha}{2} q}dzdr\nonumber\\
\ls&|\hbar|^{\eta_5 q}\int_{0}^{t}r^{q(\theta-1)-\eta_5 q-\alpha q+1}dr\cdot\int_{\RR}(1+z^2)^{-\frac{\alpha}{2} q-\eta_5 q}dz\,,
\end{align}
which can be bounded by $|\hbar|^{\gamma q}$ under conditions \eqref{condi_eta_1} with $\eta_1$ substituted by $\eta_5$. So we  complete the proof by noticing that \eqref{Condi_J_3^2} and \eqref{condi_eta_1} are  implied by \eqref{Cond_AppeC6}.
\end{proof}

\bibliography{SWE_HLW_OCT}
\bibliographystyle{plain}
%\bibliographystyle{amsplain}
%\begin{thebibliography}{10}
%
%\bibitem{BJQ2015} Balan, R.;  Jolis,  M. and  Quer-Sardanyons,  L. SPDEs with affine multiplicative fractional noise in space with index $\frac {1}{4}< H<\frac {1}{2} $. Electronic Journal of Probability.  20 (2015), no. 54, 36 pp.
%
%\bibitem{HHLNT2017} Hu, Y.;  Huang, J.;  L\^e, K.;  Nualart, D. and  Tindel,  S. Stochastic heat equation with rough dependence in space. The Annals of Probability.   45  (2017), no. 6B,  4561-616.
%\bibitem{HHLNT2019} Hu, Y.;  Huang, J.;  L\^e, K.;  Nualart, D. and  Tindel,  S.  Parabolic Anderson model with rough dependence in space. Computation and combinatorics in dynamics, stochastics and control, 477-498, Abel Symp., 13, Springer,   2018.
%
%\bibitem{HW2019} Hu, Y. and   Wang, X. Stochastic Heat Equation with general noise. 	arXiv:1912.05624 (2019).
%
%%\bibitem{Talagrand2014} Talagrand, M. Upper and lower bounds for stochastic processes: modern methods and classical problems. Vol. 60. Springer Science and Business Media, 2014. xvi+626 pp.
%
%\bibitem{KS1998} Karatzas I, Shreve SE. Brownian Motion and Stochastic Calculus. Springer. 1998.
%
%\bibitem{GR2014} Gradshteyn, I.S. and Ryzhik, I.M. Table of integrals, series, and products. Translation edited and with a preface by Daniel Zwillinger and Victor Moll. Eighth edition. Elsevier/Academic Press, Amsterdam, 2015. xlvi+1133 pp.
%\bibitem{SSX2019}Song,J., Song,X., Xu,F. Fractional stochastic wave equation driven by a Gaussian noise rough in space. 	arXiv:1904.09905 (2019).
%\end{thebibliography}

\end{document}